\documentclass[oneside,reqno,11pt,a4paper]{amsart}
\usepackage[T1]{fontenc}
\usepackage[margin=2.4cm]{geometry}
\usepackage{amsmath,amsfonts,amsbsy,amssymb,amscd,amsthm}
\usepackage[usenames,dvipsnames,svgnames]{xcolor}
\usepackage[utf8]{inputenc}
\usepackage{graphicx}
\usepackage{svg}
\graphicspath{}
\usepackage{caption}
\usepackage{subcaption}
\usepackage{doi}
\usepackage[foot]{amsaddr}
\usepackage{makecell}
\usepackage{hyperref}
\usepackage{acronym}
\usepackage{listings}
\usepackage{textgreek}
\usepackage{url}
\usepackage{color}
\usepackage{framed}
\usepackage{verbatim}
\usepackage{fancyvrb}
\usepackage{newtxtext}
\usepackage{newtxmath}
\usepackage{microtype} 
\usepackage[final]{pdfpages}
\usepackage{tikz}
\usepackage{booktabs}
\usepackage{multirow}
\usepackage[backend=biber,
sorting=none,
defernumbers=true,
maxbibnames=5,
style=numeric-comp,
isbn=false,
bibencoding=utf8,
safeinputenc, % https://github.com/plk/biblatex/issues/819
url=false,
doi=true,
giveninits=true]{biblatex}
\hypersetup{
breaklinks=true,
bookmarksopen=true,
pdftitle={Finite Element Interpolated Neural Networks},    % insert title
pdfauthor={Santiago Badia, Wei Li and Alberto F. Mart{\'{\i}}},     % insert author
colorlinks=true,       % false: boxed links; true: colored links
linkcolor=blue,          % color of internal links (change box color with linkbordercolor)
citecolor=blue,        % color of links to bibliography
filecolor=black,      % color of file links
urlcolor=blue           % color of external links
}
\definecolor{bg}{rgb}{0.93,0.93,0.93}

\newtheorem{theorem}{Theorem}[section]

\newtheorem{proposition}[theorem]{Proposition}

\newtheorem{remark}[theorem]{Remark}

\acrodef{pde}[PDE]{partial differential equation}
\acrodef{rhs}[RHS]{right-hand side}
\acrodef{fe}[FE]{finite element}
\acrodef{dg}[DG]{discontinuous Galerkin}
\acrodef{fem}[FEM]{finite element method}
\acrodef{dof}[DoF]{degree of freedom}
\acrodef{nwp}[NWP]{numerical weather prediction}
\acrodef{nn}[NN]{neural network}
\acrodef{cnn}[CNN]{convolutional neural network}
\acrodefplural{dof}[DoFs]{degrees of freedom}
\acrodef{pinn}[PINN]{physics-informed \ac{nn}}
\acrodef{vpinn}[VPINN]{variational \ac{pinn}}
\acrodef{ivpinn}[IVPINN]{interpolated \ac{vpinn}}
\acrodef{feinn}[FEINN]{\ac{fe} interpolated \ac{nn}}
\acrodef{vjp}[VJP]{vector Jacobian product}
\acrodef{gmg}[GMG]{geometric multigrid}
\acrodef{spd}[SPD]{symmetric positive definite}
\acrodef{ihcp}[IHCP]{inverse heat conduction problem}

\graphicspath{{../plots}{./}}
\newcommand{\tnor}[1]{{\left\vert\kern-0.25ex\left\vert\kern-0.25ex\left\vert #1 
\right\vert\kern-0.25ex\right\vert\kern-0.25ex\right\vert}}

\newcommand{\norm}[1]{\left\lVert #1 \right\rVert}
\newcommand{\ltwonorm}[1]{\left\lVert #1 \right\rVert _{L^2(\Omega)}}
\newcommand{\honenorm}[1]{\left\lVert #1 \right\rVert _{H^1(\Omega)}}

\begin{document}

% Title
\title[Finite element interpolated neural networks]{Finite element interpolated neural networks for solving forward and inverse problems}
\author{Santiago Badia$^{1,2,*}$}
\email{santiago.badia@monash.edu}
\author{Wei Li$^1$}
\email{wei.li@monash.edu}
\author{Alberto F. Mart\'{\i}n$^3$}
\email{alberto.f.martin@anu.edu.au}
\address{$^1$ School of Mathematics, Monash University, Clayton, Victoria 3800, Australia.}
\address{$^2$ Centre Internacional de M\`etodes Num\`erics a l'Enginyeria, Campus Nord, UPC, 08034, Barcelona, Spain.}
\address{$^3$ School of Computing, The Australian National University, Canberra ACT 2600, Australia.}
\address{$^*$ Corresponding author.}

\date{\today}

\begin{abstract}
    We propose a general framework for solving forward and inverse problems constrained by partial differential equations, where we interpolate neural networks onto finite element spaces to represent the (partial) unknowns. The framework overcomes the challenges related to the imposition of boundary conditions,  the choice of collocation points in physics-informed neural networks, and the integration of variational physics-informed neural networks. A numerical experiment set confirms the framework's capability of handling various forward and inverse problems. In particular, the trained neural network generalises well for smooth problems, beating finite element solutions by some orders of magnitude. We finally propose an effective one-loop solver with an initial data fitting step (to obtain a cheap initialisation) to solve inverse problems.
\end{abstract}


\maketitle

\section{Introduction}

Many problems in science and engineering are modelled by \emph{low-dimensional} (e.g., 2 or 3 space dimensions plus time) \acp{pde}. Since \acp{pde} can rarely be solved analytically, their solution is often approximated using numerical methods, among which the \ac{fem} has been proven to be effective and efficient for a broad range of problems. The \ac{fem} enjoys a very solid mathematical foundation~\cite{Ern2021}. For many decades, advanced discretisations have been proposed, e.g., preserving physical structure~\cite{Arnold2006}, and optimal (non)linear solvers that can efficiently exploit  large-scale supercomputers have been designed~\cite{Brune2015,Badia2016,Drzisga2017}. 

Grid-based numerical discretisations can readily handle forward problems. In a forward problem, all the data required for the \ac{pde} model to be well-posed is provided (geometry, boundary conditions and physical parameters), and the goal is to determine the state of the model. In an inverse problem setting, however, the model parameters are not fully known, but one can obtain some observations, typically noisy and/or partial, of the model state. Inverse problem solvers combine the partially known model and the observations to infer the information which is missing to complete the model. Inverse problems can be modelled using \ac{pde}-constrained minimisation~\cite{hinze2008optimization}.

Traditional numerical approximations for low-dimensional \acp{pde}, like \ac{fem}, are linear. \Ac{fe} spaces are finite-dimensional vector spaces in which one seeks for the best approximation in some specific measure. As a result, the method/grid is not adapted to local features (e.g., sharp gradients or discontinuities) and convergence can be slow for problems that exhibit multiple scales. Although adaptive \ac{fe} methods can efficiently handle this complexity, they add an additional loop to the simulation workflow (the mark and refine loop) and problem-specific robust error estimates have to be designed~\cite{Ainsworth1997}. 

When the \ac{fem} is used to solve \ac{pde}-constrained inverse problems, the unknown model parameters are usually described using \ac{fe}-like spaces, even though \ac{nn} representations have recently been proposed~\cite{NNAugumentedFEM2017,Hybrid_FEM-NN_2021}. The loss function accounts for the misfit term between the observation and the state, which in turns depends on the unknown model parameters. The gradient of the loss function with respect to the unknown model parameters requires a chain rule that involves the solution of the forward problem. An efficient implementation of this gradient relies on the adjoint method~\cite{Givoli2021}. Inverse problem solvers add an additional loop to the simulation workflow, which involves the solution of the full forward problem and the adjoint of its linearisation at each iteration. There is usually a burden of computation cost in the first stages of the adjoint method, when full forward problems are solved despite being far from the desired solution.

The tremendous success of \acp{nn} in data science has motivated many researchers to explore their application in \ac{pde} approximation. 
\Acp{pinn} have been proposed in~\cite{PINNs2019} to solve forward and inverse problems. A \ac{nn} approximates the \ac{pde} solution, while the loss function evaluates the strong \ac{pde} residual on a set of randomly selected collocation points. \acp{nn} can also be combined with a weak statement of the \ac{pde} (see, e.g., Deep Ritz Method~\cite{E2018} or \acp{vpinn}~\cite{Kharazmi2021}). \acp{nn} have some very interesting properties that make them perfectly suited for the approximation of forward and inverse \ac{pde} problems. First, \acp{nn} are genuinely \emph{nonlinear approximations}, similar to, e.g., free-knot B-splines~\cite{DeVore1998}. The solution is sought in a nonlinear manifold in the parameter space, which automatically adapts to the specific problem along the training process.\footnote{It is illustrative to observe how e.g. linear regions in \acp{nn} with ReLU activation functions adapt to the solution being approximated~\cite{https://doi.org/10.48550/arxiv.2303.11617}. The decomposition of the physical domain into linear regions is a polytopal conforming mesh.} 
Unlike \ac{fe} bases, \acp{nn} are also perfectly suited (and originally designed) for data fitting. The \ac{nn} parameters usually have a \emph{global} effect on the overall solution. As a result, one can design solvers for \ac{pde}-constrained inverse problems in which both the unknown model parameters and state variables are learnt along the same training process~\cite{Karniadakis2021}. The loss function includes the data misfit and a penalised \ac{pde} residual term. State and unknown model parameters are not explicitly linked by the forward problem, and thus no forward problems are involved in each iteration of the optimisation loop. As a result, one can use \acp{nn} to design adaptive forward and inverse problems with a one-loop solver.

Despite all these efforts, \acp{pinn} and related methods have not been able to outperform traditional numerical schemes {for low dimensional} \acp{pde}; see, e.g., the study in~\cite{https://doi.org/10.48550/arxiv.2205.14249}. There are some (intertwined) reasons for this. First, nonlinear approximability comes at the price of non-convex optimisation at the training process. Currently, non-convex optimisation algorithms for \ac{nn} approximation of \acp{pde} are costly and unreliable, especially when the \ac{pde} solutions contain multi-scale features or shocks~\cite{Zhu_2019,Fuks_2020}. 
As a result, despite the enhanced expressivity of \acp{nn}, this improvement is overshadowed by poor and costly training. Second, the integration of the \ac{pde} residual terms is not exact, and the error in the integration is either unbounded or not taken into account. Poor integration leads to poor convergence to the desired solution (due to a wrong cost functional) and one can find examples for which the optimal solution is spurious~\cite{Rivera2022}. In~\cite{https://doi.org/10.48550/arxiv.2303.11617}, the authors propose adaptive quadratures for \ac{nn} in low dimensions that are proven to be more accurate that standard Monte Carlo, especially for sharp features. Finally, the usual \ac{pde} residual norms being used in the loss function are ill-posed at the continuous level in general. It is well-known that such a variational crime has negative effects in the convergence of iterative solvers for \ac{fe} discretisations \cite{Mardal2010}, and it will also hinder the non-convex optimisation at the training process. 
These issues prevent a solid mathematical foundation of these methods, and strong assumptions are required to prove partial error estimates \cite{PINNsPDEError2020,PINNsInvErr2020}. 

Additionally, \acp{nn} have not been designed to strongly satisfy Dirichlet boundary conditions. Thus, the loss function must include penalty terms that account for the boundary conditions, which adds an additional constraint to the minimisation and has a very negative effect on the training ~\cite{Chen2020}. Such imposition of boundary conditions is not consistent, and Nitsche's method comes with the risk of ending up with an ill-posed formulation.\footnote{The coefficient in Nitsche's method must be \emph{large enough} for stability, which can be mathematically quantified in \ac{fem} using inverse inequalities. However, \acp{nn} nature do not enjoy inverse inequalities; gradients can be arbitrarily large, and can only be indirectly bounded via regularisation.} Recently, some authors have proposed to multiply the \ac{nn} with a distance function that vanishes on the Dirichlet boundary \cite{Sukumar2022}. However, this arguably complicates the geometrical discretisation step compared to grid-based methods. The computation of such distance functions is complex in general geometries and has only been used for quite simple cases in 2D. Furthermore, it is unclear how to use this approach for non-homogeneous boundary conditions, which require a lifting of the Dirichlet values inside the domain. In comparison, (unstructured) mesh generation is a mature field and many mesh generators are available \cite{Gmsh2009}. Unfitted \acp{fe} have become robust and general schemes that can handle complex geometries on Cartesian meshes \cite{dePrenter2023}. With a mesh, the definition of the lifting is trivial, e.g., one can use a \ac{fe} offset function.

Lately, significant efforts have been made to combine \ac{fem} and \acp{nn}. The authors in~\cite{neufenet2021} propose a methodology to approximate parametric PDEs. It makes use of an energy minimisation approach and a \ac{cnn} that returns the \acp{dof} of a \ac{fe} spaces. The method solves the integration issues of \acp{pinn}. However, this approach cannot handle non-trivial domains and/or non-uniform meshes, as \acp{cnn} are primarily designed for processing image-like data. \acp{cnn} that return \ac{fe} functions have also been proposed in \cite{mallon2023neural} to learn level-sets in topology optimisation that minimise a given cost function, but make use of a standard \ac{fe} solver at each iteration of the optimiser.  

Another interesting study that combines \ac{fem} and \acp{nn} is presented in~\cite{BerroneIVPINN2022}. The idea of this method is to interpolate a deep \ac{nn} onto a \ac{fe} space and design a well-posed \ac{pde}-residual loss functional. The authors compare the solution of the interpolated \ac{nn} (a \ac{fe} function) with different standard \ac{pinn} formulations. Despite the fact that the solution belongs to a fixed \ac{fe} space (and cannot exploit nonlinear approximation, compared to the other \ac{pinn} strategies), the results are superior in general. This technique, coined \acp{ivpinn}, has been applied to forward coercive grad-conforming \acp{pde} on rectangular domains. Unlike other \acp{pinn}, a priori error bounds have been obtained~\cite{BerroneIVPINN2022}, even though suboptimal compared to the \ac{fem} solution. 

One can argue what is the benefit of getting sub-optimal \ac{fe} solutions (measured in the energy norm) using a far more expensive non-convex optimisation solver. However, \acp{ivpinn} shed light on the negative impact that integration, residual definition, imposition of boundary conditions, lack of well-posedness, and training have on a straightforward approximation of \acp{pde} using \acp{nn}. In~\cite{radaptiveDL2022}, the authors propose an $r-$adaptive deep learning method, in which the \acp{nn} are interpolated onto a \ac{fe} space with a mesh that dynamically changes during training. Compared to \acp{ivpinn}, the  \ac{fe} mesh is not fixed but learned during the training process. However, the proposed method is limited to tensor product meshes, which also prevents its application to complicated geometries and reduces the type of meshes that can be learned. Moreover, both methods rely on a distance function and an offset function for imposition of the Dirichlet boundary conditions~\cite{https://doi.org/10.48550/arxiv.2210.14795}, which can be problematic when Dirichlet functions are complicated or Dirichlet boundaries are irregular.

In this work, we build upon \acp{ivpinn} ideas. However, instead of enforcing boundary conditions at the \ac{nn}, we propose to strongly impose the boundary conditions at the \ac{fe} space level. This allows us to readily handle complex geometries without the need to define, e.g., distance functions. To distinguish the two approaches, we coin the proposed method \acp{feinn}. Besides, we explore the benefits of considering the trained \ac{nn} (instead of the \ac{fe} interpolation) as the final solution of the problem, i.e., evaluate how the trained \ac{nn} \emph{generalises}. We also discuss different \ac{pde} residual norms and suggest to use Riesz preconditioning techniques to end up with a well-posed formulation in the continuous limit. We perform a numerical analysis of the method, and prove that the proposed formulation can recover (at least) the \emph{optimal} \ac{fe} bounds. Next, we apply these techniques to inverse problems, using a one-loop algorithm, as it is customary in \acp{pinn}. We exploit the excellent properties of \acp{nn} to fit data. We propose a first step in which we get a state initial guess by data fitting. In a second step, we learn the unknown model parameters by \ac{pde}-residual minimisation for a fixed state. The previous steps provide an initialisation for a third fully coupled step with a mixed data-\ac{pde} residual cost function. 

We carry out a comprehensive set of numerical experiments for forward problems. We check that expressive enough \acp{nn} can return \ac{fe} solutions for different polynomial orders. For smooth problems, the generalisation results for the trained \acp{nn} are striking. The solution obtained with the non-interpolated \ac{feinn} solution can be orders of magnitude more accurate than the \ac{fe} solution on the same mesh, while \acp{ivpinn} do not generalise that well. The definition of the residual norm (and its preconditioned version) can have a tremendous impact in the convergence of the minimisation algorithm. Finally, we test the proposed algorithm for inverse problems. Unlike standard inverse solvers for grid-based methods, we can solve inverse problems with effective and cheap initialisation and one-loop algorithms, even without any kind of regularisation terms. 

The outline of the article is the following. Sec.~\ref{sec:method_forward} states the model elliptic problem that we tackle, its \ac{fe} discretisation, the \ac{nn} architecture, and the proposed loss functions in the \ac{feinn} discretisation. Sec.~\ref{sec:num-an} proves that the interpolation of an expressive enough \ac{nn} recovers the \ac{fe} solution. In Sec.~\ref{sec:method_inverse}, the proposed discretisation is applied to inverse problems, by defining a suitable loss function that includes data misfit and a multi-step minimisation algorithm. Sec.~\ref{sec:implementation} describes the implementation of the methods and Sec.~\ref{sec:experiments} presents the numerical experiments on several forward and inverse problems. Finally, Sec.~\ref{sec:conclusions} draws conclusions and lists potential directions for further research.

\section{Forward problem discretisation using neural networks} \label{sec:method_forward}

\subsection{Continuous problem}\label{subsec:cont-prob}

In this work, we aim to approximate elliptic PDEs using a weak (variational) setting. As a model problem, we consider a convection-diffusion-reaction equation, even though the proposed methodology can readily be applied to other coercive problems. The problem reads: find $u \in H^1(\Omega)$ such that  
\begin{equation} \label{eq:conv_diff_react_strong_form}
- \pmb{\nabla}\cdot(\kappa \pmb{\nabla} u) + (\pmb{\beta} \cdot \pmb{\nabla}) u + \sigma u = f \quad \hbox{in }  \Omega, \quad 
u = g \quad \hbox{on }  \Gamma_D, \quad
\kappa \pmb{n} \cdot \pmb{\nabla} u = \eta \quad \hbox{on } \Gamma_N,
\end{equation}
where 
$\Omega \subset \mathbb{R}^d$ is a  Lipschitz polyhedral domain, $\Gamma_D$ and $\Gamma_N$ are a partition of its boundary such that  $\mathrm{meas}(\Gamma_D) > 0$ , $\kappa$, $\sigma \in L^\infty(\Omega)$, $\pmb{\beta} \in W^{1,\infty}(\Omega)^d$ such that $\sigma - \pmb{\nabla}\cdot \pmb{\beta} > 0$ and $\pmb{\beta} \cdot \pmb{n}_{|\Gamma_N} \ge 0$ , $f \in H^{-1}(\Omega)$, $g \in H^{1/2}(\Gamma_D)$, and $\eta \in H^{-1/2}(\Gamma_N)$.     

Consider the space $U \doteq H^1(\Omega)$, $\tilde{U} \doteq H^1_{0,\Gamma_D}(\Omega) \doteq \left\{  v \in U \ : \ v_{|\Gamma_D} = 0 \right\}$, a continuous lifting $\bar{u}\in U$ of the Dirichlet boundary condition (i.e., $\bar{u} = g$ on $\Gamma_D$), and the forms 
\[
a(u,v) = \int_{\Omega} \kappa \pmb{\nabla}u \cdot \pmb{\nabla}v + (\pmb{\beta} \cdot \pmb{\nabla}) u v + \sigma u v, \quad 
\ell(v) = \int_{\Omega} f v + \int_{\Gamma_N}^{} \eta v.
\] 
(We use the symbol $\tilde{\cdot}$ to denote trial functions and spaces with zero traces.) The variational form of the problem reads: find $u = \bar{u} + \tilde{u}$ where 
\begin{equation} \label{eq:conv_diff_react_weak_form}
  \tilde{u} \in \tilde{U} \ : \ a(\tilde{u},v) = \ell(v) - a(\bar{u},v), \quad \forall v \in \tilde{U}.
\end{equation}
In this setting, the problem with a non-homogeneous Dirichlet boundary condition is transformed into a homogeneous one via the lifting and a modification of the \ac{rhs}. The well-posedness of the problem relies on the coercivity and continuity of the forms:
\[
a(u,u) \geq \gamma \|u\|_U^2, \quad a(u,v) \leq \xi \|u\|_U \|v\|_U, \quad \ell(v) \leq \chi \|v\|_U.
\] 
Below, we will make use of the PDE residual 
\begin{equation} \label{eq:conv_diff_react_weak_residual}
  \mathcal{R}(\tilde{u}) \doteq \ell(\cdot) - a(\tilde{u} + \bar{u},\cdot) \in \tilde{U}' .
\end{equation}

\subsection{Finite element approximation}\label{sec:fem-appr}

Next, we consider a family of conforming shape-regular partitions $\{\mathcal{T}_h\}_{h > 0}$ of $\Omega$ such that their intersection with $\Gamma_N$ and $\Gamma_D$ is also a partition of these lower-dimensional manifolds; $h$ represents a characteristic mesh size. On such partitions, we can define a trial \ac{fe} space $U_h \subset U$ of order $k_U$ and the subspace $\tilde{U}_h \doteq U_h \cap H_{0,\Gamma_D}^1(\Omega)$ of \ac{fe} functions with zero trace. 

We define a \ac{fe} interpolant $\pi_h : \mathcal{C}^0 \rightarrow U_h$ obtained by evaluation of the \acp{dof} of $U_h$. In this work, we consider grad-conforming Lagrangian (nodal) spaces (and thus composed of piece-wise continuous polynomials), and \acp{dof} are pointwise evaluations at the Lagrangian nodes. Analogously, we define the interpolant $\tilde{\pi}_h$ onto $\tilde{U}_h$. We can pick a \ac{fe} lifting $\bar{u}_{h} \in U_h$ such that $\bar{u}_h = \pi_h(g)$ on $\Gamma_D$. (The interpolant is restricted to $\Gamma_D$ and could be, e.g., a Scott-Zhang interpolant if $g$ is non-smooth.) Usually in \ac{fem}, $\bar{u}_h$ is extended by zero on the interior.   

Using the Galerkin method, the test space is defined as $V_h \doteq \tilde{U}_h$, and let $k_V$ be its order. Following \cite{BerroneIVPINN2022}, we also explore Petrov-Galerkin discretisations. To this end, we consider $k_V$ such that $s = k_U / k_V \in \mathbb{N}$, and a family of partitions $\mathcal{T}_{h/s}$ obtained after $s$ levels of uniform refinement of $\mathcal{T}_h$. In this case, we choose $V_h$ to be the \ac{fe} space of order $k_V$ on $\mathcal{T}_{h/s}$ with zero traces on $\Gamma_D$. We note that the dimension of $\tilde{U}_h$ and $V_h$ are identical. In this work, we only consider $k_V = 1$, i.e., a \emph{linearised} test \ac{fe} space. The well-posedness of the Petrov-Galerkin discretisation is determined by the discrete inf-sup condition:
\[
\underset{u_h \in \tilde{U}_h }{\mathrm{inf}} \underset{v_h \in V_h}{\mathrm{sup}} \frac{a(u_h,v_h)}{\|u_h\|_U \|v_h\|_U} \geq \beta > 0.
\]
In both cases, the problem can be stated as: find $u_h = \bar{u}_h + \tilde{u}_h$ where
\begin{equation} \label{eq:conv_diff_react_fe_weak_form}
  \tilde{u}_h \in \tilde{U}_h \ : \ a(\tilde{u}_h, v_h) = \ell(v_h) - a(\bar{u}_h,v_h), \quad \forall v_h \in V_h.
\end{equation}
We represent with $\mathcal{R}_h$ the restriction  $\mathcal{R}|_{{U}_h \times V_h}$. 
Given $u_h \in U_h$, $\mathcal{R}_h(u_h) \in V_h'$. $V_h'$ is isomorphic to $\mathbb{R}^N$, where $N$ is the dimension of $V_h$ (and $\tilde{U}_h$). This representation depends on the basis chosen to span $V_h$.          

\subsection{Neural networks}\label{subsec:nns}

We consider a fully-connected, feed-forward \ac{nn}, obtained by the composition of affine maps and nonlinear activation functions. The network architecture is represented by a tuple $(n_0, \ldots n_L)\in \mathbb{N}^{(L+1)}$, where $L$ is the number of layers and $n_k$ is the number of neurons on layer $1 \leq k \leq L$. We take $n_0 = d$ and, for scalar-valued \acp{pde}, we have $n_L = 1$. In this work, we use $n_1 = n_2 = ... = n_{L-1} = n$, i.e. all the hidden layers have an equal number of neurons $n$.

At each layer $1 \leq k \leq L$, we represent with $\pmb{\Theta}_k: \mathbb{R}^{n_{k-1}} \to \mathbb{R}^{n_k}$ the affine map at layer $k$, defined by $\pmb{\Theta}_k \pmb{x} = \pmb{W}_k \pmb{x} + \pmb{b}_k$ for some weight matrix $\pmb{W}_k \in \mathbb{R}^{n_k \times n_{k-1}}$ and bias vector $\pmb{b}_k \in \mathbb{R}^{n_k}$. 
The activation function $\rho: \mathbb{R} \to \mathbb{R}$ is applied element-wise after every affine map except for the last one. Given these definitions, the network is a parametrizable function $\mathcal{N}(\pmb{\theta}): \mathbb{R}^d \to \mathbb{R}$ defined as: 
\begin{equation} \label{eq:nn_structure}
  \mathcal{N}(\pmb{\theta}) = \pmb{\Theta}_L \circ \rho \circ \pmb{\Theta}_{L-1} \circ \ldots \circ \rho \circ \pmb{\Theta}_1,
\end{equation}
where $\pmb{\theta}$ stands for the collection of all the trainable parameters $\pmb{W}_k$ and $\pmb{b}_k$ of the network. Although the activation functions could be different at each layer or even trainable, we apply the same, fixed activation function everywhere.
However, we note that the proposed methodology is not restricted to this specific \ac{nn} architecture. In this work, we denote the \ac{nn} architecture with $\mathcal{N}$ and a realisation of the \ac{nn} with $\mathcal{N}(\pmb{\theta})$.

\subsection{Finite element interpolated neural networks} \label{subsec:feinns}
In this work, we propose the following discretisation of~\eqref{eq:conv_diff_react_weak_form}, which combines the \ac{nn} architecture in~\eqref{eq:nn_structure} and the \ac{fe} problem in~\eqref{eq:conv_diff_react_fe_weak_form}. Let us consider a norm $\|\cdot\|_Y$ for the discrete residual (choices for this norm are discussed below). 
 We aim to find
\begin{equation} \label{eq:feinn_forward_loss}
  u_\mathcal{N} \in \mathrm{arg} \underset{w_\mathcal{N} \in \mathcal{N}}{\mathrm{min}}  \mathscr{L}(w_\mathcal{N}), \qquad \mathscr{L}(w_\mathcal{N}) \doteq \| \mathcal{R}_h(\tilde{\pi}_h(w_\mathcal{N})) \|_Y.
\end{equation} 
The computation of $u_\mathcal{N}$ involves a non-convex optimisation problem (due to the nonlinear dependence of $u_\mathcal{N}$ on $\pmb{\theta}$). We prove in the next section that the  $\tilde{\pi}_h(u_\mathcal{N})$ equal to the \ac{fe} solution is a global minimum of this functional.
  
In this method, the \ac{nn} is \emph{free} on $\Gamma_D$, the imposition of the Dirichlet boundary conditions relies on a \ac{fe} lifting $\bar{u}_h$ and the interpolation $\tilde{\pi}_h$ onto $\tilde{U}_h$ applied to the \ac{nn} (thus vanishing on $\Gamma_D$). Conceptually, the proposed method trains a \ac{nn} \emph{pinned} on the \acp{dof} of the \ac{fe} space $\tilde{U}_h$, with a loss function that measures the \ac{fe} residual of the interpolated \ac{nn} for a given norm. The motivation behind the proposed method is to eliminate the Dirichlet boundary condition penalty term in standard \acp{pinn} and related methods~\cite{PINNs2019,E2018}, while avoiding enforcing the conditions at the \ac{nn} level (see, e.g.,~\cite{Sukumar2022} for \acp{pinn} and~\cite{https://doi.org/10.48550/arxiv.2210.14795} for \acp{vpinn}). It also solves the issues related to Monte Carlo integration~\cite{Rivera2022} and avoids the need to use adaptive quadratures~\cite{https://doi.org/10.48550/arxiv.2303.11617}. Using standard element-wise integration rules, the integrals in $\mathcal{R}_h$ can be exactly computed (or, at least, its error can be properly quantified for non-polynomial physical parameters and forcing terms). Moreover, in the current setting, we can consider different alternatives for the residual norm and better understand the deficiencies and variational crimes related to standard choices. 

\subsection{Loss function}\label{subsec:forward_loss}

As discussed above, the loss function involves the norm of the \ac{fe} residual. The residual is isomorphic to the vector $[\mathbf{r}_h(w_h)]_i = \left< \mathcal{R}(w_h), \varphi^i \right> \doteq \mathcal{R}(w_h)(\varphi^i)$, where $\{\varphi^i\}_{i=1}^N$ are the \ac{fe} shape functions that span the test space $V_h$. As a result, we can consider the loss function:
  \[
  \mathscr{L}(u_\mathcal{N}) = \|\mathbf{r}_h(\tilde{\pi}_h(u_\mathcal{N}))\|_{\ell^2}.
  \] 
This is the standard choice (possibly squared) in the methods proposed so far in the literature that rely on variational formulations \cite{Kharazmi2021,BerroneIVPINN2022,pmlr-v120-khodayi-mehr20a}. However, as it is well-known in the \ac{fe} setting, this quantity is ill-posed in the limit $h \downarrow 0$ \cite{Mardal2010}. At the continuous level, the norm  of $\mathcal{R}(u)$ is not defined. 

If the problem is smooth enough and $\mathcal{R}$ is well-defined on $L^2(\Omega)$ functions, we can define its $L^2(\Omega)$ projection onto $V_h$ as follows:
\[
\mathcal{M}_h^{-1}\mathcal{R}_h(w_h) \in V_h \ : \
\int_{\Omega} \mathcal{M}_h^{-1}\mathcal{R}_h(w_h) v_h = \mathcal{R}_h(w_h)(v_h), \quad \forall v_h \in V_h.
\] 
Next, one can define the cost function 
\[
  \mathscr{L}(u_\mathcal{N}) =  \| \mathcal{M}_h^{-1}\mathcal{R}_h ( \tilde{\pi}_h(u_\mathcal{N}))\|_{L^2(\Omega)},
\] 
which, for quasi-uniform meshes, is equivalent (up to a constant) to the scaling of the Euclidean norm, i.e., $h^d \|\mathbf{r}_h (\tilde{\pi}_h(u_\mathcal{N}))\|_{\ell^2}$. However, for non-smooth solutions, the $L^2$ norm of the residual still does not make sense at the continuous level, and thus, the convergence must deteriorate as $h \downarrow 0$. One can instead define a discrete Riesz projector $\mathcal{B}_h^{-1}: V_h' \to V_h$ such that
\[
\mathcal{B}_h^{-1}\mathcal{R}_h(w_h) \in V_h \ : \left( \mathcal{B}_h^{-1}\mathcal{R}_h(w_h), v_h \right)_U  = \mathcal{R}(w_h)(v_h), \quad \forall v_h \in V_h.
\]
For the model case proposed herein, $\|\cdot\|_U$ is the $H^1$ or $H^1_{0,\Gamma_D}$-norm and $\mathcal{B}_h^{-1}$  is the inverse of the discrete Laplacian. 
Then, one can consider the cost function: 
\begin{equation} \label{eq:preconditioner_l2}
  \mathscr{L}(u_\mathcal{N}) =  \| \mathcal{B}_h^{-1}\mathcal{R}_h ( \tilde{\pi}_h(u_\mathcal{N}))\|_{L^2(\Omega)},
\end{equation}
or 
  \[
  \mathscr{L}(u_\mathcal{N}) =  \| \mathcal{B}_h^{-1}\mathcal{R}_h ( \tilde{\pi}_h(u_\mathcal{N}))\|_{H^1(\Omega)}.
  \]
These cost functions are well-defined in the limit $h \downarrow 0$. In practice, one can replace $\mathcal{B}_h^{-1}$ by any spectrally equivalent approximation in order to reduce computational demands. For example, in the numerical experiments section, we consider several cycles of a \ac{gmg} preconditioner. 

\section{Analysis}\label{sec:num-an}
In this section, we first show that the proposed loss functions are differentiable. Next, we show that the interpolation of the \ac{nn} architecture can return any \ac{fe} function in a given \ac{fe} space. Combining these two results, we observe that there exists a global minimum of the \ac{feinn} problem in ~\eqref{eq:feinn_forward_loss} such that its interpolation is the solution of the \ac{fe} problem ~\eqref{eq:conv_diff_react_fe_weak_form}.

\begin{proposition}\label{prop:differentiable}
The loss function is differentiable for $\mathcal{C}^0$ activation functions.  
\end{proposition}
\begin{proof}
Using the chain rule, we observe that
\[
  \frac{\mathrm{d} \mathscr{L}}{\mathrm{d} \pmb{\theta}} = 
  \frac{\mathrm{d} \mathscr{L}}{\mathrm{d} \mathbf{r}_h} 
  \frac{\mathrm{d} \mathbf{r}_h}{\mathrm{d} \mathbf{u}_h} 
  \frac{\mathrm{d} \mathbf{u}_h}{\mathrm{d} \pmb{\theta}},
\]
for $\mathbf{u}_h$ being the \acp{dof} of the \ac{fe} space $U_h$. The first derivative in the \ac{rhs} simply involves the squared root of a quadratic functional. The second derivative is the standard Jacobian of the \ac{fe} problem. The third derivative is the vector of derivatives of the \ac{nn} at the nodes of $U_h$, which is well-defined for $\mathcal{C}^0$ activation functions. As a result, $\mathscr{L}$ is differentiable.
\end{proof} 

Consequently, one can use gradient-based minimisation techniques. We note that this is not the case when the \ac{nn} is evaluated without \ac{fe} interpolation. For instance, refer to~\cite{https://doi.org/10.48550/arxiv.2303.11617} for a simple example that shows ReLU activation functions cannot be used for \ac{pde} approximation using \acp{pinn} and related methods. In \acp{pinn}, one must compute $\pmb{\nabla}_{\pmb{\theta}} \pmb{\nabla}_{\pmb{x}} \mathcal{N}$, which poses additional smoothness requirements on the activation function. However, in the proposed methodology (as in \cite{BerroneIVPINN2022}), the spatial derivatives are computed by the interpolated function, not the \ac{nn}, and thus not affected by this constraint. For simplicity, we prove the result for the ReLU activation function. 

\begin{proposition}\label{prop:emulation}
  Let $U_h$ be a \ac{fe} space on a mesh $\mathcal{T}_h$ in $\mathbb{R}^d$ with \acp{dof} equal to $N \simeq h^{1/d}$. Let $\mathcal{N}$ be a neural network architecture with 3 layers, $(3 d N, d N, N)$ neurons per layer, and a ReLU activation function. For any $u_h \in U_h$, there exists a choice of the \ac{nn} parameters $\pmb{\theta}$ such that $\pi_h(u_\mathcal{N}) = u_h$.  
  \end{proposition}
  
  \begin{proof}
  
  At each node $\pmb{n} \in \mathcal{T}_h$ , one can define a box $\mathcal{B}_{\pmb{n}}$ centred at $\pmb{n}$ that only contains this node of the mesh. Let us consider first the 1D case. 
  For ReLU activation functions, one can readily define a hat function with support in $[0,1]$ as follows. First, we consider 
  \[
  f_1(x) = 2x, \quad f_2(x) = 4x-2, \quad f_3(x) = 2x-2. 
  \]  
  One can check that 
  $f = \rho(f_1) - \rho(f_2) + \rho(f_3)$ is a hat function with value 1 at $x = 1/2$ and support in $[0,1]$. One can readily consider a scaling and translation to get $\mathrm{supp}(f) \subset \mathcal{B}_{\pmb{n}}$. This way, assuming one has $3 N$ neurons in the first layer and $N$ neurons in the second layer, one can emulate the 1D \ac{fe} basis in the second layer.
  
  In 2D, one can create the 1D functions for both $x$ and $y$ directions. It requires $6N$ neurons in the first layer and $2N$ neurons in the second layer. Thus, for each node, we have two hat functions, namely $b_1$ and $b_2$, that depend on $x$ and $y$, respectively. Now, in a third layer with $N$ neurons, we can compute $\rho(b_1 + b_2 -1)$ at each node. We can generalise this construction to an arbitrary dimension $d$ . We need $3dN$ neurons in the first layer to create the 1D functions in all directions. The hat functions are created in a second layer with $dN$ neurons. The final functions are combined as $\rho(\sum_{i=1}^{d}  b_i - d + 1)$. We note that, by the construction of $b_i$, these functions have value one in the corresponding node and their support is contained in the corresponding box.
  
  In the last layer, we end up with a set of functions $\psi_i$  that are equal to 1 on one node and zero on the rest. Besides, the \ac{fe} function can also be expressed as $u_h = \sum_{i=1}^{N} u^i \varphi^i(\pmb{x})$ and $\pi_h(u_\mathcal{N}) = \sum_{i=1}^{N} u_\mathcal{N}(\pmb{n}_i) \varphi^i(\pmb{x})$. Linearly combining the last layer functions with the \ac{dof} values $\{u^i\}_{i=1}^{N}$ we construct a \ac{nn} realisation that proves the proposition.
  \end{proof}
 
  \begin{remark}
For other activation functions like tanh or sigmoid, it is not possible to construct localised functions with compact support as in the proof above. However, one can consider a piecewise polynomial approximation of these activation functions (e.g., using B-splines) with this property \cite{Sunat2006}. Then, one can use a similar construction as in ReLU.
  \end{remark}
  
  We note that this construction can be further optimised by exploiting the structure of the underlying \ac{fe} mesh $\mathcal{T}_h$. For instance, for a structured mesh of a square with $n$ parts per direction ($N = n^d$), only $3n$ neurons are needed. We can exploit the fact that many nodes share the same coordinates in some directions. For the same reason, only $dn$ neurons are required in the second layer. On the other hand, for more than 3 layers, the computations can be arranged among neurons/layers in different ways. For simplicity, in the proposition, we consider a worst-case scenario situation (no nodes share coordinate components and we only consider the arrangement in the proposition statement).
  
  \begin{proposition}\label{prop:error-est}
    Let us assume that the \ac{fe} problem ~\eqref{eq:conv_diff_react_fe_weak_form} is well-posed and admits a unique solution $\tilde{u}_h$. The \ac{feinn} problem ~\eqref{eq:feinn_forward_loss} admits a global minimiser $u_\mathcal{N}$ such that $\tilde{\pi}_h(u_\mathcal{N}) = \tilde{u}_h$. 
  \end{proposition}
  
  \begin{proof}
    First, we note that the loss function differentiable (by Prop.~\ref{prop:differentiable}) and positive. Besides, from the statement of the problem and Prop.~\ref{prop:emulation}, one can readily check that there exists a $u_\mathcal{N}$ such that $\mathcal{R}_h(\tilde{\pi}_h(u_\mathcal{N})) = \mathcal{R}_h(\tilde{u}_h) = 0$ and thus $\mathscr{L}(u_\mathcal{N}) = 0$, i.e. $u_\mathcal{N}$ is a global minimum of the cost function.           
  \end{proof}

As a result, the \ac{feinn} method can exhibit optimal convergence rates (the ones of \ac{fem}), provided the \ac{nn} is expressive enough compared to the \ac{fe} space. In Sec.~\ref{subsec:forward_exp}, we experimentally observe this behaviour. This analysis is different from the one in \cite{BerroneIVPINN2022}, which, using a completely different approach, proves sub-optimal results in a different setting. The numerical experiments in \cite{BerroneIVPINN2022} and in Sec.~\ref{subsec:forward_exp} show that \acp{ivpinn} can also recover optimal convergence rates. In fact, the results above can straightforwardly be extended to \acp{ivpinn}. The sub-optimality in \cite{BerroneIVPINN2022} is related to the choice of the residual norm, the $\ell^2$ norm of the residual vector. Sharper estimates could likely be obtained with the new residual norms suggested in Sec.~\ref{subsec:forward_loss}. 
\section{Inverse problem discretisation using neural networks}\label{sec:method_inverse}

In this section, we consider a \ac{pde}-constrained inverse problem that combines observations of the state variable $u$ and a partially known model ~\eqref{eq:conv_diff_react_strong_form}. Let us represent with $\pmb{\Lambda}$ the collection of unknown model parameters. It can include the physical coefficients, forcing terms and Dirichlet and Neumann boundary values. We parametrise $\pmb{\Lambda}$ with one or several \acp{nn}, e.g., as the ones proposed for the state variable in Sec.~\ref{subsec:feinns}, which will be represented with $\pmb{\Lambda}_\mathcal{N}$. Again, $n_0=d$, while $n_L$ depends on whether the unknown model parameter of the specific problem is a scalar-valued ($n_L=1$), vector-valued ($n_L=d$) or tensor-valued ~($n_L=d^2$) field. 

Let us denote with $\mathcal{R}(\pmb{\Lambda},u)$ the \ac{pde} residual in ~\eqref{eq:conv_diff_react_weak_residual}, where we make explicit its dependence with respect to the unknown model parameters (idem for $\mathcal{R}_h$). For integration purposes, we consider the interpolation of the model parameters onto \ac{fe} spaces, which we represent with $\pmb{\pi}_h(\pmb{\Lambda}_\mathcal{N})$. The discrete model parameter \ac{fe} spaces can in general be different to $U_h$ (just as their infinite-dimensional counterpart spaces might be different to $U$) and do not require imposition of boundary conditions. Besides, the interpolation can be restricted to different boundary regions for Dirichlet and Neumann values. 
If we consider a discontinuous nodal \ac{fe} space with nodes on the quadrature points of the Gaussian quadrature being used for integration (as in the numerical experiments), the interpolated and non-interpolated methods are equivalent. Thus, the interpolant simply accounts for the integration error being committed when integrating the \acp{nn} for the unknown model parameters.

Let us consider a measurement operator $\mathcal{D}: U \rightarrow \mathbb{R}^{M}$ and the corresponding vector of observations $\mathbf{d} \in \mathbb{R}^M$. The loss function for the inverse problem must contain the standard data misfit term and a term that accounts for the PDE residual. The method is understood as a (PDE-)constrained minimisation problem. As a result, the PDE residual is weighted by a (dynamically adapted) penalty coefficient.  We consider the loss functional:
\begin{equation} \label{eq:inverse_loss}
  \mathscr{L}(\pmb{\Lambda},u) \doteq \|\mathbf{d} - \mathcal{D}(u)\|_{\ell^2} + \alpha \|\mathcal{R}_h(\pmb{\pi}_h(\pmb{\Lambda}),\tilde{\pi}_h(u))\|_Y, 
\end{equation} 
for any of the choices of the residual norm discussed above and $\alpha \in \mathbb{R}^+$ is a penalty coefficient for the weak imposition of the \ac{pde} constraint. The inverse problem reads:
\begin{equation} \label{eq:inverse_min}
  u_\mathcal{N}, \pmb{\Lambda}_\mathcal{N} \in \underset{w_\mathcal{N}, \pmb{\Xi}_\mathcal{N} \in \mathcal{N}_u \times \mathcal{N}_{\pmb{\Lambda}}}{\mathrm{arg \, min}}  \mathscr{L}( \pmb{\Xi}_\mathcal{N} ,  w_\mathcal{N}).
\end{equation}

We refer to ~\cite{InversePenaltyMethod2015} for an application of penalty methods to inverse problems. However, their approach is more akin to the adjoint method, where they eliminate the state. We note that our approach is a \emph{one-loop} minimisation algorithm, i.e., one can minimise for both the state and unknown model parameters at the same time. This differs from adjoint methods, in which the loss function and the minimisation is in terms of $\pmb{\Lambda}$ only, but the state $u(\pmb{\Lambda})$ is constrained to be the solution of the (discrete) PDE at each iterate of $\pmb{\Lambda}$.

To alleviate the challenges associated with the training of the loss function described in ~\eqref{eq:inverse_min} and enhance the robustness of our method, we propose the following algorithm. The motivation behind its design is to exploit the excellent properties of \acp{nn} for data fitting. First, we train the state \ac{nn} with the observations. Next, we train the unknown model parameters \acp{nn} with the PDE residual, but freeze the state variable to the value obtained in the previous step. These steps are computationally lightweight because they do not involve differential operators in the training processes. These two initial steps are finally used as initialisation for the one-loop minimisation in ~\eqref{eq:inverse_min}. We summarise the algorithm below:

\begin{itemize}

\item Step 1 (Data fitting): Train the state neural network to fit the observed data, using standard \ac{nn} initialisation:
\[
u_\mathcal{N}^0 = \underset{w_\mathcal{N} \in \mathcal{N}_u }{\mathrm{arg \, min}} \|\mathbf{d} - \mathcal{D}(w_\mathcal{N})\|_{\ell^2}.
\] 

\item Step 2 (Unknown model parameters initialisation): Train the model parameter \acp{nn} with the PDE residual for the fixed state $u_\mathcal{N}^0$ computed in Step 1, using standard \ac{nn} initialisation: 
\[
  \pmb{\Lambda}_\mathcal{N}^0 = \underset{\pmb{\Xi}_\mathcal{N} \in \mathcal{N}_{\pmb{\Lambda}}}{\mathrm{arg \, min}} \|\mathcal{R}_h(\pmb{\pi}_h(\pmb{\Xi}_\mathcal{N}),\tilde{\pi}_h(u^0_\mathcal{N}))\|_Y.
\] 

\item Step 3 (Fully coupled minimisation): Train both the state and model parameter \acp{nn} the full loss function ~\eqref{eq:inverse_loss}, starting from $u_\mathcal{N}^0$ and $\pmb{\Lambda}_\mathcal{N}^0$. 
\end{itemize}

It is important to point out that the three-step training process is facilitated by the incorporation of \acp{nn}. We attempted to apply the same methodology using \ac{fe} functions directly, but the outcomes were unsatisfactory. This is attributed to the local support of \ac{fe} functions, which limits the adjustment of the values of the free nodes that are directly influenced by the observations. In contrast, \acp{nn} with their global support, allow for parameter tuning across the entire domain.

\section{Implementation} \label{sec:implementation}
We rewrite \eqref{eq:inverse_loss} in the following algebraic form 
\begin{equation} \label{eq:inverse_loss_discrete}
    \mathscr{L}(\pmb{\theta}_\lambda, \pmb{\theta}_u) = \norm{\mathbf{e}(\mathbf{u}_h(\pmb{\theta}_u))}_{\ell^2} + \alpha \norm{\mathbf{r}_h(\mathbf{u}_h(\pmb{\theta}_u), \pmb{\lambda}_h(\pmb{\theta}_\lambda))}_{\ell^1},
\end{equation} 
where $\mathbf{e} \doteq \mathbf{d} - \mathcal{D}_h \mathbf{u}_h$ is the data misfit error, $\mathbf{r}_h$ is the variational residual vector,  $\mathbf{u}_h$, $\boldsymbol{\lambda}_h$ are the vectors of \acp{dof} of $\tilde{\pi}_h(u_\mathcal{N}(\pmb{\theta}_u))$ and $\pmb{\pi}_h(\pmb{\Lambda}_{\mathcal{N}}(\pmb{\theta}_\mathcal{\lambda}))$ of the \ac{nn} realisations $u_\mathcal{N}(\pmb{\theta}_u)$ and $\pmb{\Lambda}_{\mathcal{N}}(\pmb{\theta}_\mathcal{\lambda})$ for the arrays of parameters $\pmb{\theta}_u$ and $\pmb{\theta}_\lambda$, respectively. We have chosen the $\ell^1$ residual norm in (\ref{eq:inverse_loss_discrete}) because it is the one we have used in the numerical tests for inverse problems in Sec.~\ref{sec:experiments}. However, the proposed implementation is general and can be easily adapted to other choices of residual norms proposed above.  

We describe below an implementation of \acp{feinn} using Julia packages, even though the proposed implementation is general. In Julia, we rely on the existing packages \texttt{Flux.jl} \cite{Flux2018,FluxOSS2018} for the neural network part and \texttt{Gridap.jl} \cite{Gridap2020,Gridap2022} for the FEM part. We employ \texttt{ChainRules.jl}~\cite{ChainRules} to automatically propagate user-defined rules across the code.

To minimise the loss function \eqref{eq:inverse_loss_discrete} with gradient-based training algorithms, these gradients are required:
\begin{equation*}
    \frac{\partial \mathscr{L}}{\partial \pmb{\theta}_u} = \left(\frac{\partial \mathscr{L}}{\partial \mathbf{r}_h} \frac{\partial \mathbf{r}_h}{\partial \mathbf{u}_h} + \frac{\partial \mathscr{L}}{\partial \mathbf{e}} \frac{\partial \mathbf{e}}{\partial \mathbf{u}_h}\right) \frac{\partial \mathbf{u}_h}{\partial \pmb{\theta}_u}, \qquad 
    \frac{\partial \mathscr{L}}{\partial \pmb{\theta}_\lambda} = \frac{\partial \mathscr{L}}{\partial \mathbf{r}_h} \frac{\partial \mathbf{r}_h}{\partial \pmb{\lambda}_h} \frac{\partial \pmb{\lambda}_h}{\partial \pmb{\theta}_\lambda}.
\end{equation*}
Existing chain rules in \texttt{ChainRules.jl} can readily handle $\partial \mathscr{L} / \partial \mathbf{r}_h$ and $\partial \mathscr{L} / \partial \mathbf{e}$. We need to define specific rules for the automatic differentiation of the following tasks:

\begin{itemize}
  \item The interpolation of a \ac{nn} onto a \ac{fe} space in  $\partial \mathbf{u}_h/\partial \pmb{\theta}_u$ and $\partial \pmb{\lambda}_h / \partial \pmb{\theta}_{\lambda}$;
  \item The computation of the \ac{fe} residual in $\partial \mathbf{r}_h/\partial \mathbf{u}_h$ and $\partial \mathbf{r}_h / \partial \pmb{\lambda}_h$;
  \item The measurement operator $\mathcal{D}$ on the \ac{fe} state in $\partial \mathbf{e}/\partial \mathbf{u}_h$.
\end{itemize}

It is important to highlight that we never explicitly construct the global Jacobian matrices in our implementation. To evaluate the gradient $\partial \mathscr{L} / \partial \pmb{\theta}_{\lambda}$, we utilise \texttt{Gridap.jl} to compute the Jacobian $\partial \mathbf{r}_h / \partial \pmb{\lambda}_h$ cell-wise (i.e., at each cell of $\mathcal{T}_h$ separately), and restrict the vector $\partial \mathscr{L} / \partial \mathbf{r}_h$ to each cell. By performing the \ac{vjp} within each cell for $\partial \mathscr{L} / \partial \mathbf{r}_h$ and $\partial \mathbf{r}_h / \partial \pmb{\lambda}_h$, we obtain the cell-wise vectors that can be assembled to form $\partial \mathscr{L} / \partial \pmb{\lambda}_h$. With the help of \texttt{Flux.jl}, we can calculate the gradient $\partial \mathscr{L} / \partial \pmb{\theta}_{\lambda}$ by performing the \ac{vjp} for $\partial \mathscr{L} / \partial \pmb{\lambda}_h$ and $\partial \pmb{\lambda}_h / \partial \pmb{\theta}_{\lambda}$, without explicitly constructing the Jacobian $\partial \pmb{\lambda}_h / \partial \pmb{\theta}_{\lambda}$. This cell-wise approach recasts most of the floating point operations required to compute the gradients in terms of dense matrix-vector products. This results in a reduction of the computational times and memory requirements.

The gradient $\partial \mathscr{L} / \partial \pmb{\theta}_u$ has two contributions, corresponding to the \ac{fe} residual and data misfit terms. The same process described above is applied to compute the former contribution. 
The contribution of the data misfit term involves the computation of $\partial \mathscr{L} / \partial \mathbf{e} \ \partial \mathbf{e}/\partial \mathbf{u}_h$, which has not been discussed so far. In our implementation, it also follows an efficient cell-wise approach. In particular, we identify those cells with at least one observation point and, for these cells, we evaluate the cell shape functions at the observation points. This is nothing but the restriction of $\partial \mathbf{e}/\partial \mathbf{u}_h$ to the observation points and \acp{dof} of the cell. We then restrict the vector $\partial \mathscr{L} / \partial \mathbf{e}$ to these cells, and compute the \ac{vjp} among these vector and Jacobian restrictions. Finally, we assemble the resulting cell-wise vector contributions to obtain the data misfit global contribution vector to the vector $\partial \mathscr{L} /\partial \mathbf{u}_h$.

Once all the rules for Jacobian computations are appropriately defined, \texttt{ChainRules.jl} seamlessly combine them, enabling smooth gradient computation during the training process.

Let us finish the section with a discussion about computational cost. In \acp{feinn} and \acp{ivpinn}, one computes the spatial derivatives in the residual on \ac{fe} functions in $\partial \mathbf{r}_h / \partial  \mathbf{u}_h$ and the derivatives of pointwise evaluations of the \ac{nn} with respect to parameters in $\partial \mathbf{u}_h / \partial \pmb{\theta}$ separately. The expression of the polynomial derivatives is straightforward and the parameter differentiation is the one required in standard data fitting (and thus, highly optimised in machine learning frameworks). On the contrary, in standard \acp{pinn} the residual is not evaluated with the projection $\mathbf{u}_h$ but the \ac{nn} itself. One must compute $\partial \mathbf{r}_h / \partial \pmb{\theta}$ directly. It involves \emph{nested} derivatives (in terms of parameters and input features) that are more expensive (and less common in data science).
  % Any reference here?
  %\footnote{E.g., the computation of nested derivatives in \texttt{Flux.jl} is orders of magnitude more expensive than derivatives with respect to parameters only. }
  % }
\section{Numerical experiments} \label{sec:experiments}

\subsection{Forward problems} \label{subsec:forward_exp}
We use the standard $L^2$ and $H^1$ error norms to evaluate the precision of the 
approximation $u^{id}$  
for forward problems:
\begin{equation*}
    e_{L^2(\Omega)}(u^{id}) = \ltwonorm{u - u^{id}}, \qquad 
    e_{H^1(\Omega)}(u^{id}) = \honenorm{u - u^{id}},
\end{equation*}
where $u$ is the true state, $\ltwonorm{\cdot} = \sqrt{\int_\Omega |\cdot|^2}$, and $\honenorm{\cdot} = \sqrt{\int_\Omega |\cdot|^2 + |\pmb{\nabla} (\cdot)|^2}$. 
The integrals in these terms are evaluated with Gauss quadrature rule, and a sufficient number of quadrature points are used to guarantee accuracy. Note that, in the forward problem experiments, $u^{id}$ can either be a \ac{nn} 
or its interpolation onto a suitable \ac{fe} space. We will specify which representation is used explicitly when necessary. 

As for the experiments, we first compare \acp{feinn} with \acp{ivpinn} by solving the forward convection-diffusion-reaction problem~\eqref{eq:conv_diff_react_strong_form}. Next, we shift to the Poisson equation, i.e., problem~\eqref{eq:conv_diff_react_strong_form} with $\pmb{\beta} = \mathbf{0}$ and $\sigma = 0$, and analyse the impact of preconditioning on accelerating convergence during the training process. Finally, we showcase the effectiveness of \acp{feinn}  in complex geometries by solving a Poisson problem in a domain characterised by irregular shapes. It is worth noting that a comprehensive comparison 
in terms of computational cost and accuracy between \acp{ivpinn}, \acp{pinn}, and \acp{vpinn} has already been conducted in ~\cite{BerroneIVPINN2022}. In these experiments, the accuracy of \acp{ivpinn} is similar or better than the other \acp{pinn} being analysed for a given number of \ac{nn} evaluations. The computational cost of \acp{ivpinn} is reported to be lower than standard \ac{pinn} approaches, which is explained by the different cost of differentiation in each case, as explained in Sec.~\ref{sec:implementation}. As a result, we restrict ourselves to the comparison between \acp{feinn} and \acp{ivpinn} and refer the reader to ~\cite{BerroneIVPINN2022} for the relative merit of \acp{feinn} over \acp{pinn}.

In all the experiments in this section, we adopt the \ac{nn} architecture
in~\cite{BerroneIVPINN2022}, namely $L = 5$ layers, $n = 50$ neurons for each hidden layer, and $\rho = \tanh$ as activation function, so that we can readily compare these results with the ones in~\cite{BerroneIVPINN2022} for standard \acp{pinn}. Besides, this choice strikes a good balance between the finest \ac{fe} resolution being used and the \ac{nn} expressivity. Indeed, we have experimentally observed that increasing the expressiveness of the \ac{nn} (additional number of layers and/or neurons per layer) for the finest \ac{fe} mesh being used in our experiments does not noticeably improve the results.

In addition, we employ Petrov-Galerkin discretisations, i.e., we use a linearised test space $V_h$ as defined in Sec.~\ref{sec:fem-appr}. Unless otherwise specified, we adopt the $\ell^2$ norm in the loss function \eqref{eq:feinn_forward_loss}. 
In all the experiments in this section and Sec.~\ref{sec:inverse}, we use the Glorot uniform method~\cite{Glorot2010} for \ac{nn} parameter initialisation and the BFGS optimiser in \texttt{Optim.jl}~\cite{Optimjl2018}.\footnote{We have experimentally observed that L-BFGS is not as effective as BFGS for the problems considered in this paper.} 

\subsubsection{Convection-diffusion-reaction equation with a smooth solution}  \label{sec:advection_diffusion_eq_smooth}
We replicate most of the experiment settings in~\cite[Convergence test \#1]{BerroneIVPINN2022}, allowing the interested reader to check how other \acp{pinn} perform in similar experiments by looking at this reference. Specifically, the problem is defined on a square domain $\Omega = [0, 1]^2$, $\Gamma_D$ are spanned by the left and right sides, and $\Gamma_N$ by the the top and bottom ones.  We choose the following analytical functions for the model parameters:
\begin{equation*}
    \kappa(x, y) = 2 + \sin(x + 2y), \qquad 
    \pmb{\beta}(x, y) = \left[\sqrt{x - y^2 + 5},\ \sqrt{y - x^2 + 5}\right]^{\rm T}, \qquad 
    \sigma(x, y) = \rm{e}^{\frac{x}{2} - \frac{y}{3}} + 2,
\end{equation*}
and pick $f$, $g$ and $\eta$ such that the exact solution is: 
\begin{equation*}
    u(x, y) = \sin(3.2x(x-y))\cos(x+4.3y)+\sin(4.6(x+2y))\cos(2.6(y-2x)).
\end{equation*}
We discretise the domain using uniform meshes of quadrilateral elements of equal size.

It is crucial to emphasize that \acp{ivpinn} and \acp{feinn} share a fundamental idea at their core: the interpolation of \acp{nn} (or their product with a function for \acp{ivpinn}) onto a corresponding \ac{fe} space. The primary distinction lies in the approach used to impose the Dirichlet boundary condition. \acp{feinn} rely on the trial FE space to enforce the boundary condition, using an interpolation that enforces zero trace on $\Gamma_D$. \acp{ivpinn}, however, rely on an offset function $\bar{u}$ and a distance function $\Phi$, where $\bar{u}$ is as smooth as $u$ and satisfies the Dirichlet boundary condition and $\Phi \in \tilde{U}$. The authors propose in \cite{BerroneIVPINN2022} to train an auxiliary neural network or use data transfinite interpolation to compute the lifting $\bar{u}$ of the Dirichlet data $g$. So, the true state can be expressed as $\Phi \circ {u}_\mathcal{N} + \bar{u}$.  \acp{ivpinn}  interpolate this expression onto the \ac{fe} space.
The interpolated \ac{nn} composition now belongs to $\tilde{U}_h$ due to the property of $\Phi$, and the full expression (approximately) satisfies the Dirichlet boundary condition because of the existence of $\bar{u}$. The loss function of the method reads:
\[
u_\mathcal{N} \in \mathrm{arg} \ \underset{w_\mathcal{N}}{\mathrm{min}} \ \| \hat{\mathbf{r}}_h(\pi_h(\Phi \circ w_\mathcal{N} + \bar{u}))\|, \quad [\hat{\mathbf{r}}_h]_i \doteq \ell(\varphi^i) - a(\pi_h(\Phi \circ w_\mathcal{N} + \bar{u}),\varphi^i),
\] 
where $\{\varphi^i\}_{i=1}^N$ are the shape functions that span $V_h$.
In our numerical experiments, we have considered the training of an auxiliary \ac{nn} to approximate $\bar{u}$, but the results were not satisfactory. (Probably, because we are computing a function in $\Omega$ with data on $\Gamma_D$ only.) 
We have considered instead a discrete harmonic extension (i.e., a \ac{fe} approximation of the Poisson problem with $g$ on $\Gamma_D$) to approximate $\bar{u}$. 
Since we consider a trivial square domain, the distance function $\Phi$ can readily be defined as the product of the linear polynomials, i.e., $\Phi(x, y) = x(1-x)$; note that $\Gamma_D$ only includes the left and right sides of the squared domain.

In addition to evaluating the performance of \acp{feinn} and \acp{ivpinn}, we also examine how the \acp{nn} generalise. For \acp{ivpinn}, we compute the error of the \ac{nn} composition $\Phi \circ u_{\mathcal{N}} + \bar{u}$, while, in the case of \acp{feinn}, we compute the error of $u_{\mathcal{N}}$ directly. We emphasise that this setting aligns with the principles of \ac{nn} training: we train the \ac{nn} with data in a set of points (the nodes of the mesh), and if the training is effective, we expect the \ac{nn} to yield low error on the whole domain $\bar{\Omega}$.

For the first experiment, we investigate the impact of mesh refinement on the approximation error. Keeping $k_U = 6$ fixed, we discretise the domain using a uniform mesh of quadrilaterals with different levels of refinement. To account for the impact of \ac{nn} initialisation on both \acp{feinn} and \acp{ivpinn}, we run 10 experiments with different initialisations for each mesh resolution. Fig. \ref{fig:advection_diffusion_h_refinement_smooth_l2_err} and \ref{fig:advection_diffusion_h_refinement_smooth_h1_err} illustrate the $L^2$ errors and $H^1$ errors, respectively, for the different methods versus mesh size. The curves labelled as ``FEM'' refer to the errors associated to the \ac{fem}
solution, those labelled as ``\ac{feinn}'' and ``\ac{ivpinn}'' to the errors of the interpolated \acp{nn} resulting from either method, and, finally, the label tag ``(\ac{nn} only)'' is used to refer to the (generalisation) error associated to the \ac{nn} itself (i.e., not to its \ac{fe} interpolation).   Due to the negligible variance in the errors of the interpolated \acp{nn} for both \acp{feinn} and \acp{ivpinn}, we present the average error among those obtained for the 10 experiments. We also provide the slopes of the \ac{fem} convergence curves in Fig.~\ref{fig:advection_diffusion_h_refinement_smooth_l2_err} and \ref{fig:advection_diffusion_h_refinement_smooth_h1_err}. The computed slopes closely match the expected theoretical values, validating the \ac{fem} solution and the accuracy of the error computation. Based on the observations from Fig.~\ref{fig:advection_diffusion_h_refinement_smooth}, \acp{feinn} not only generalise better compared to \acp{ivpinn}, they also have the potential to outperform \ac{fem}. This capability of \acp{feinn} is not coincidental, as all errors associated to the \acp{nn} resulting from \acp{feinn}, consistently remain below the \ac{fem} convergence curve. Additionally, we observe that as the mesh becomes finer, \acp{ivpinn} starts to struggle. 
While more training iterations may reduce the errors of \acp{ivpinn}, it is worth noting that the number of training iterations reaches the prescribed limit of 30,000 for the three finest mesh resolutions.   
It is also interesting to compare the distribution of errors among the \acp{nn} resulting from \acp{ivpinn} and \acp{feinn}. We observe a high sensitivity of the errors to \ac{nn} initialisation for \acp{ivpinn}, whereas the errors tend to cluster for \acp{feinn}. We also observe that the $L^2$ error of \ac{fem} gets closer to that of the non-interpolated \ac{nn} resulting from \ac{feinn} as the mesh is refined. This behaviour is expected, since the \ac{fe} mesh is being refined while the \ac{nn} architecture is fixed. There is a point in which the \ac{nn} is not expressive enough to represent the optimal \ac{fe} solution and thus, Prop.~\ref{prop:error-est} does not hold any more.

\begin{figure}
    \centering
    \begin{subfigure}[t]{0.48\textwidth}
        \includegraphics[width=\textwidth]{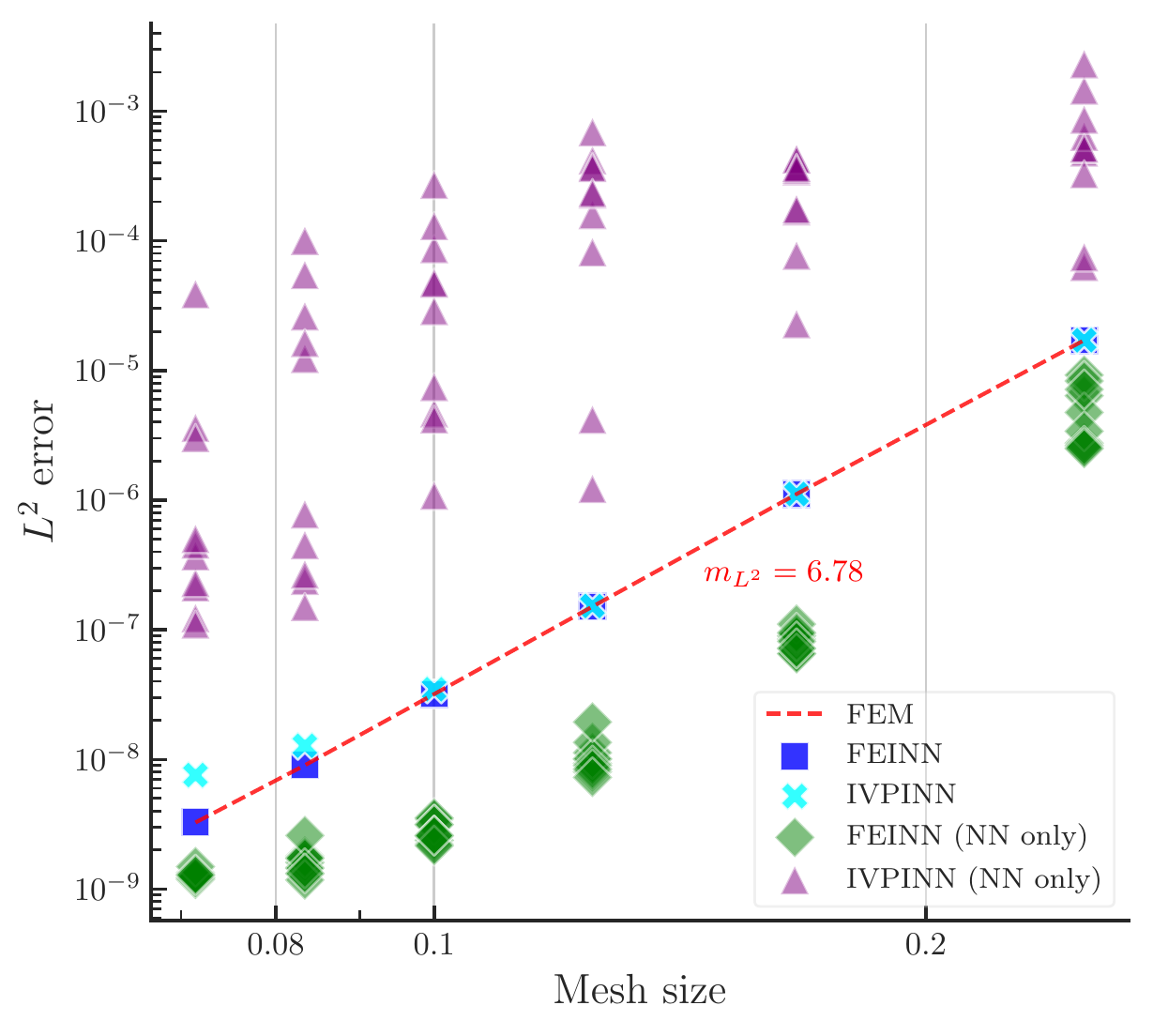}
        \caption{}
        \label{fig:advection_diffusion_h_refinement_smooth_l2_err}
    \end{subfigure}
    \begin{subfigure}[t]{0.48\textwidth}
        \includegraphics[width=\textwidth]{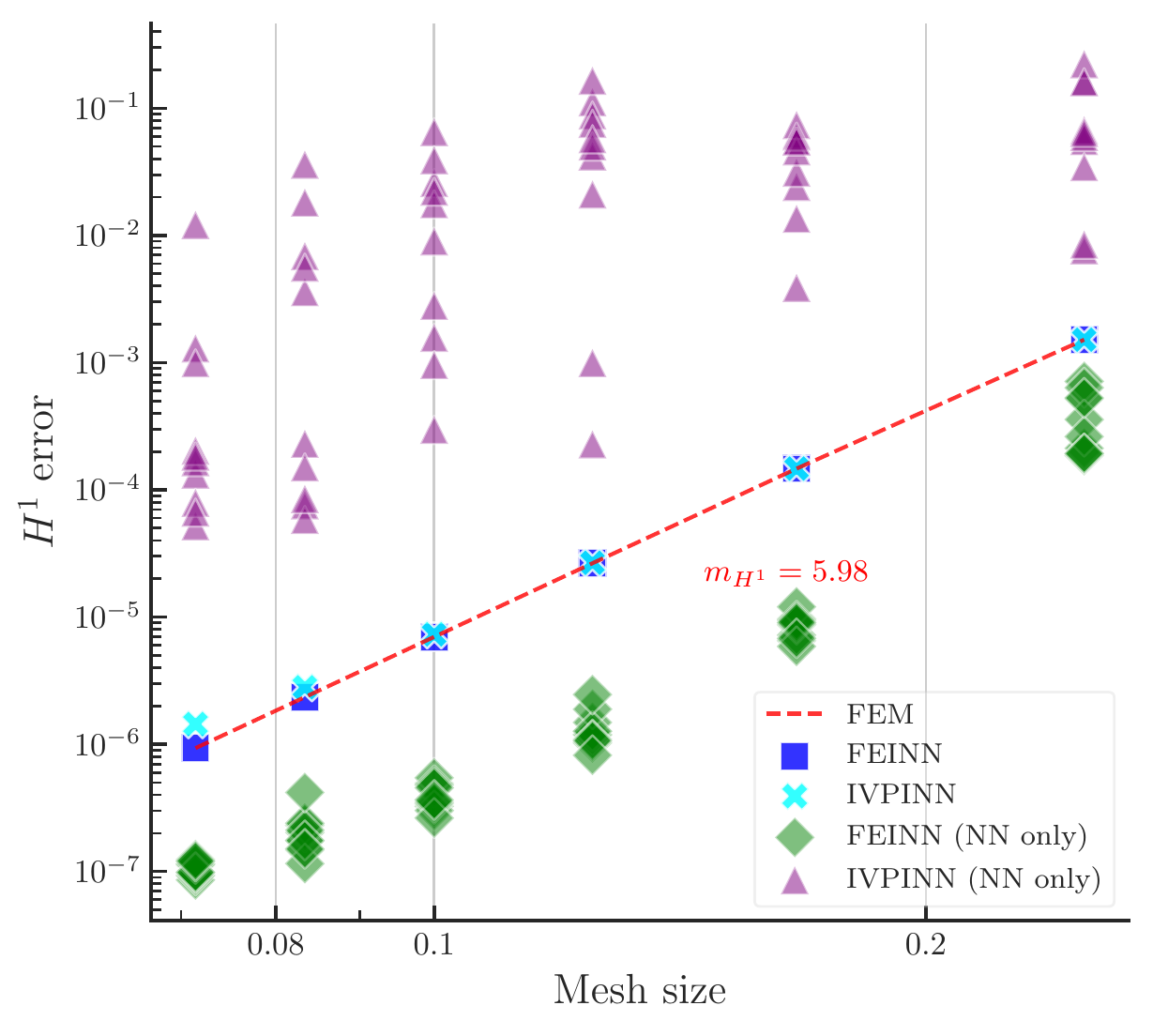}
        \caption{}
        \label{fig:advection_diffusion_h_refinement_smooth_h1_err}
    \end{subfigure}
     
    \caption{Convergence of errors with respect to the mesh size of the trial space for the forward convection-diffusion-reaction problem with a smooth solution.}
    \label{fig:advection_diffusion_h_refinement_smooth}
\end{figure}

Since $u \in \mathcal{C}^\infty(\bar{\Omega})$ in this problem, similar to \ac{fem}, we can also explore at which rate the error decays as we increase the polynomial order of the trial space (i.e., the \ac{nn} interpolation space). We maintain a fixed mesh consisting of $15\times15$ quadrilaterals, and increase $k_U$ from $1$ up to $6$. We perform 10 experiments for each order, with a different \ac{nn} initialisation for each experiment. Fig.~\ref{fig:advection_diffusion_p_refinement_smooth_l2_err} and \ref{fig:advection_diffusion_p_refinement_smooth_h1_err} depict the $L^2$ and $H^1$ errors, respectively, against $k_U$. Once again, we observe that \acp{feinn} have comparable performance to \ac{fem}, and more importantly, the non-interpolated \acp{nn} resulting from \acp{feinn} have lower errors than \ac{fem}. In some cases, these can outperform \ac{fem} by more than two orders of magnitude. On the same mesh, the \ac{nn} obtained with \acp{feinn} is comparable to the \ac{fe} solution obtained using between one and two orders more.   
Overall, \acp{ivpinn} demonstrate a comparable level of performance to \ac{fem}, with the exception occurring at $k_U = 6$. After 30,000 training iterations, it fails to achieve the same performance as \ac{fem}. Notably, the non-interpolated \ac{nn} compositions from \acp{ivpinn} yield satisfactory results at lower orders, but as the order increases, they fail to reach the accuracy of \ac{fem}. The same comment about the expressivity limit of the \ac{nn} architecture applies here. As we increase the order, the improvement of \ac{feinn} becomes less pronounced, since we are keeping fix the \ac{nn} architecture.

\begin{figure}
    \centering
    \begin{subfigure}[t]{0.48\textwidth}
        \includegraphics[width=\textwidth]{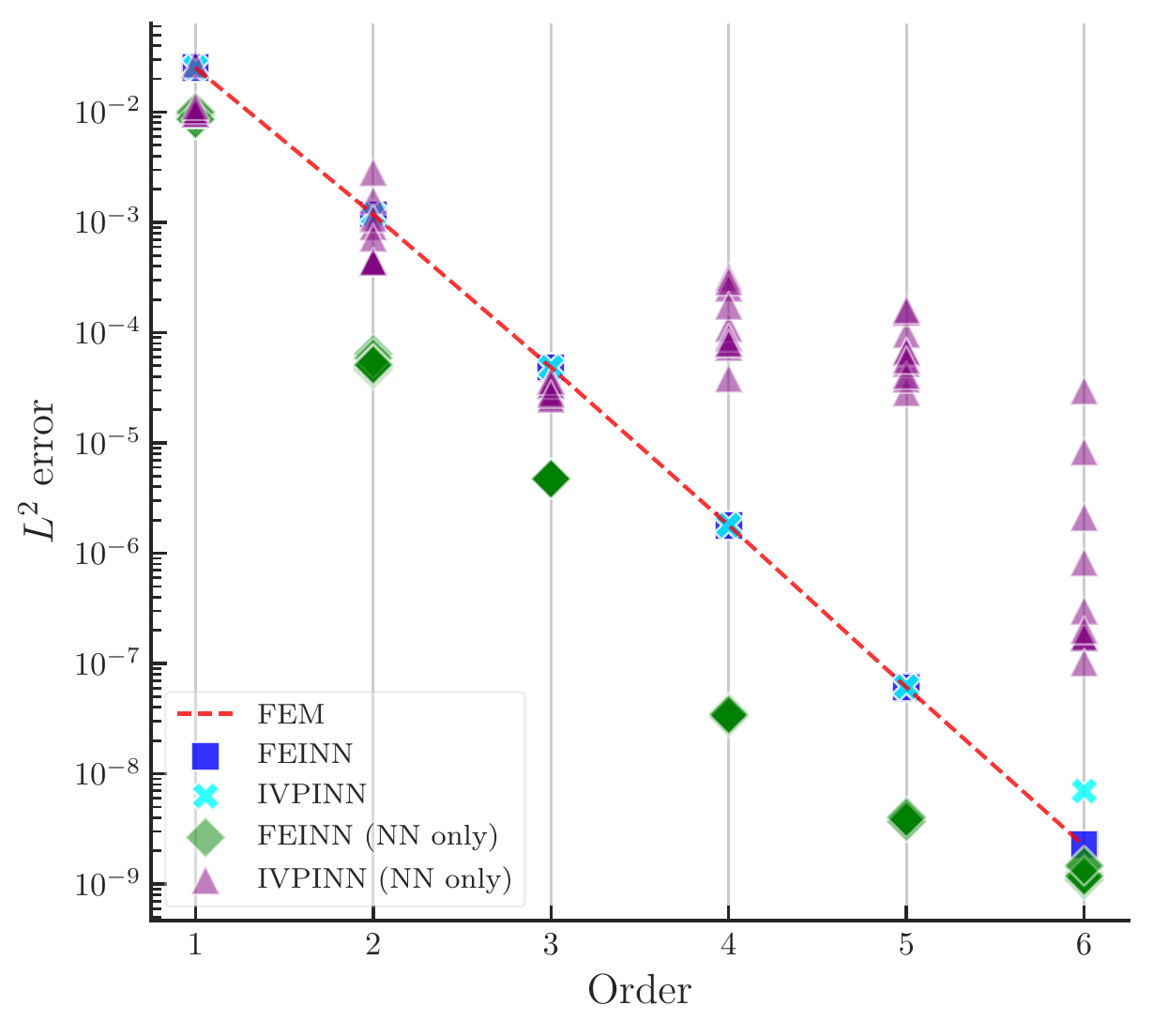}
        \caption{}
        \label{fig:advection_diffusion_p_refinement_smooth_l2_err}
    \end{subfigure}
    \begin{subfigure}[t]{0.48\textwidth}
        \includegraphics[width=\textwidth]{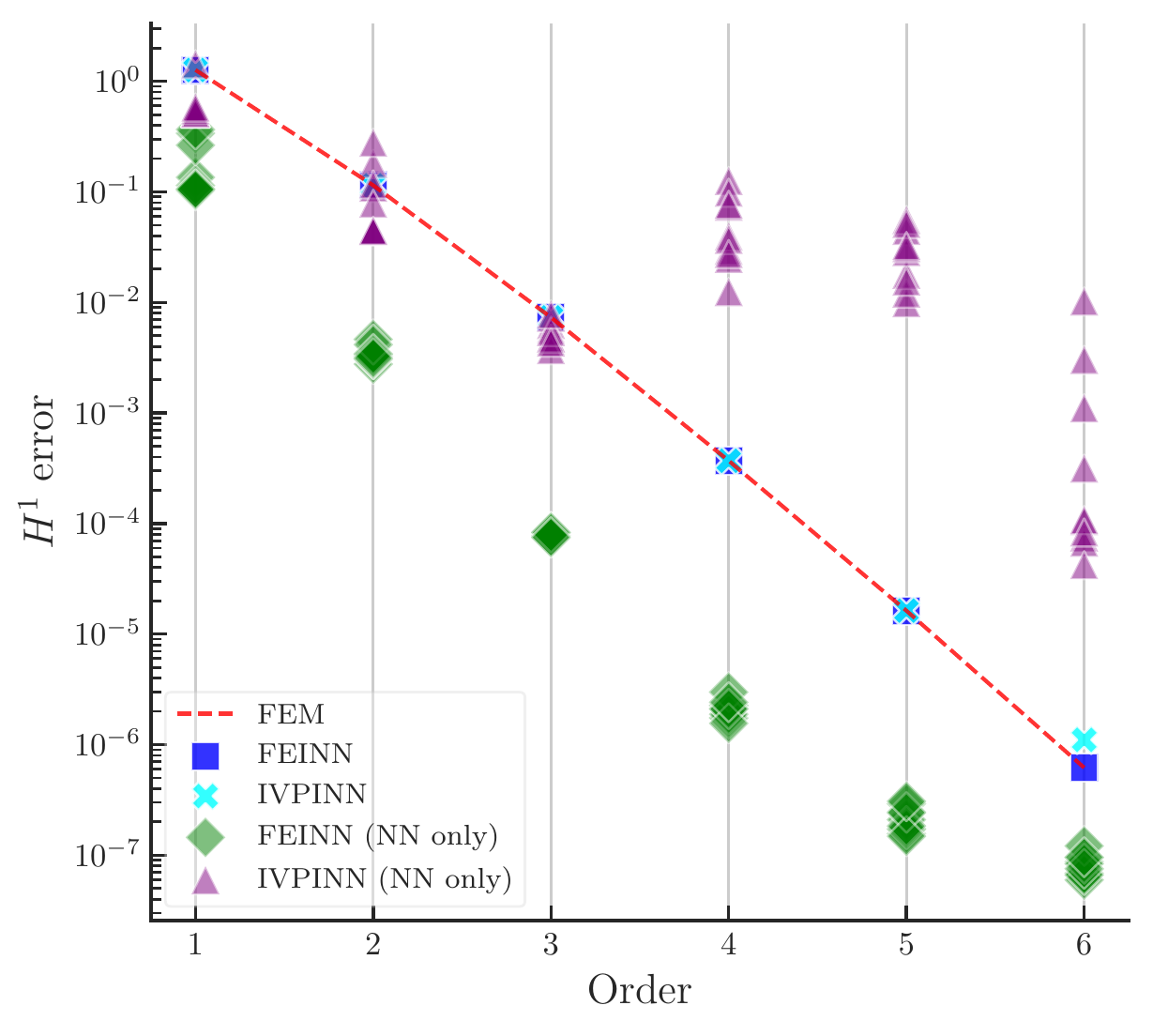}
        \caption{}
        \label{fig:advection_diffusion_p_refinement_smooth_h1_err}
    \end{subfigure}
     
    \caption{Convergence of errors with respect to the order of trial bases for the forward convection-diffusion-reaction problem with a smooth solution.}
\end{figure}

\begin{figure}
    \centering
    \begin{subfigure}[t]{0.48\textwidth}
        \includegraphics[width=\textwidth]{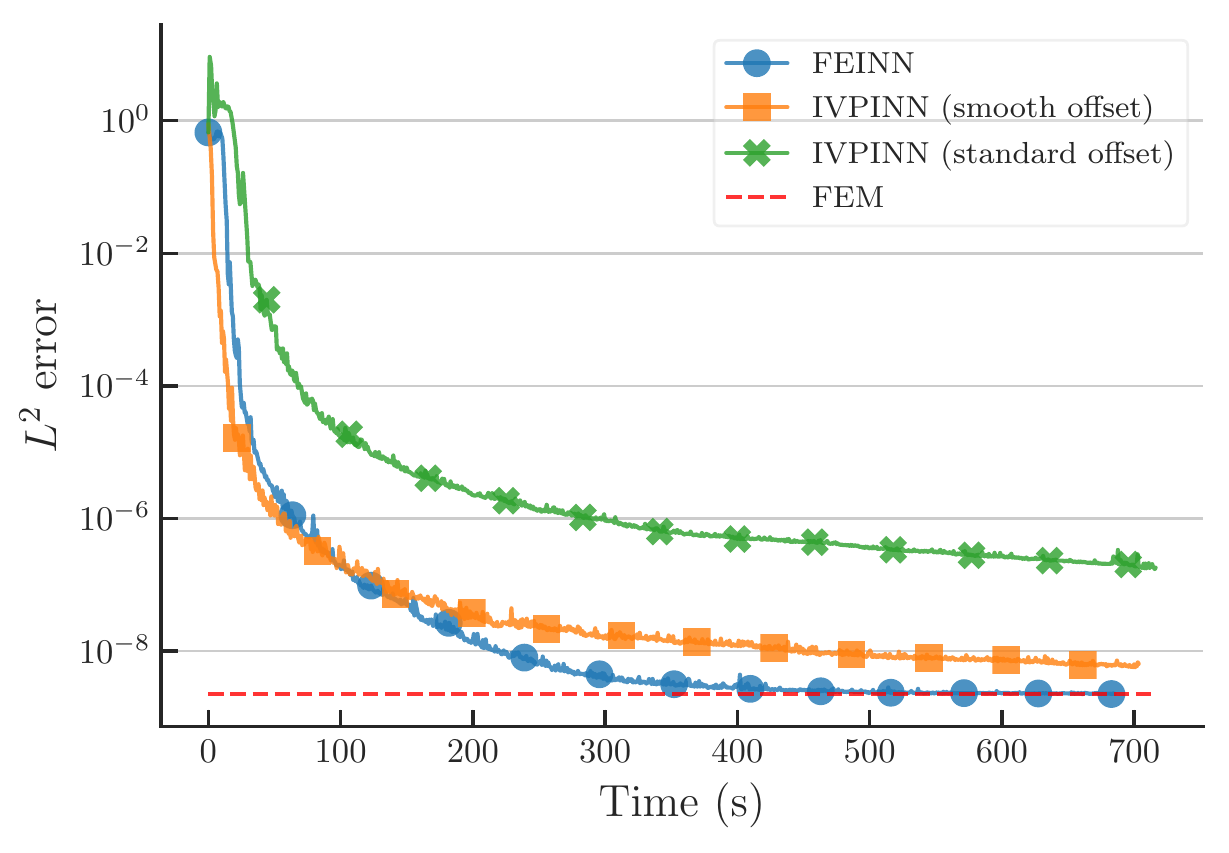}
        \caption{}
        \label{fig:advection_diffusion_smooth_l2err_vs_time}
    \end{subfigure}
    \begin{subfigure}[t]{0.48\textwidth}
        \includegraphics[width=\textwidth]{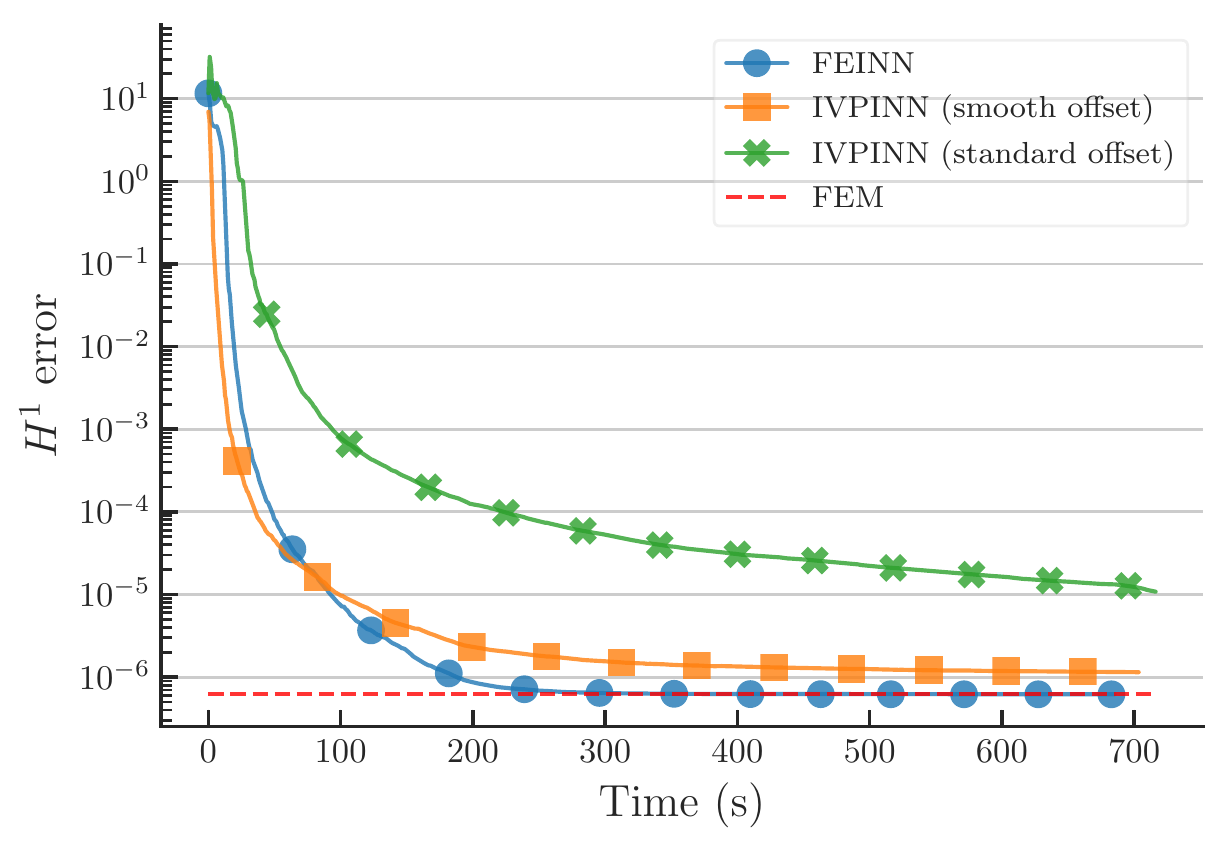}
        \caption{}
        \label{fig:advection_diffusion_smooth_h1err_vs_time}
    \end{subfigure}
    \caption{Comparison among \acp{feinn} and \acp{ivpinn} in terms of $L^2$ and $H^1$ errors versus computational time. Training was performed in both methods for a fixed number of 30,000 iterations.}
    \label{fig:advection_diffusion_smooth_errs_vs_time}
\end{figure}

In the sequel, we investigate how the computational cost and convergence rates of \acp{feinn} and \acp{ivpinn} compare. To this end, we solve the same problem so far in this section by training \acp{feinn} and \acp{ivpinn} for a fixed number of iterations, and then visualise at which rate the $L^2$ and $H^1$ errors decay with time.
This is reported in Fig.~\ref{fig:advection_diffusion_smooth_errs_vs_time}. 
We used $k_U = 6$, a mesh consisting of $15 \times 15$ quadrilaterals (resulting in a problem with 8,099 \acp{dof}), and a fixed number of 30,000 BFGS training iterations.
Besides, we consider two different choices of the offset function to study its impact on the performance of \acp{ivpinn}. In particular, the curve labelled ``(smooth offset)'' in Fig.~\ref{fig:advection_diffusion_smooth_errs_vs_time}  denotes the same offset function  used so far in this section (namely, a discrete harmonic extension, as per suggested in \cite{BerroneIVPINN2022}), while the one labelled ``(standard offset)'' denotes the offset function that one naturally uses in \ac{fem} (and we also use here with \acp{feinn}). We observe that both \acp{feinn} and \acp{ivpinn} have roughly the same computational cost per iteration, around 0.023 seconds per iteration on a GeForce RTX 3090 GPU. Moreover, as shown in Fig.~\ref{fig:advection_diffusion_p_refinement_smooth_l2_err} and ~\ref{fig:advection_diffusion_p_refinement_smooth_h1_err}, \acp{ivpinn} converge consistently slower than \acp{feinn}. Fig.~\ref{fig:advection_diffusion_smooth_errs_vs_time} also illustrates that the choice of offset function greatly influences the convergence rate of \acp{ivpinn}. Indeed, \ac{ivpinn} with the smooth offset function converges much faster than with the standard offset function, while for \acp{feinn}, we readily obtain a faster convergence rate without the need for a special offset function. It is also worth noting that the authors in~\cite{BerroneIVPINN2022} observe that \acp{ivpinn} are less computationally expensive than \acp{pinn} and \acp{vpinn}, and thus for this problem, which indicates that \acp{feinn} are also more efficient than the latter two methods.

\subsubsection{Convection-diffusion-reaction equation with a singular solution} \label{subsubsec:singular_cdr_eq_exp}
The second problem we solve is still~\eqref{eq:conv_diff_react_strong_form}, but with a singular solution. We adopt most of the settings in ~\cite[Convergence test \#2]{BerroneIVPINN2022}. The domain and boundaries are the same as those in Sec.~\ref{sec:advection_diffusion_eq_smooth}. The coefficients are $\kappa = 1$, $\pmb{\beta} = [2, 3]^{\rm T}$, and $\sigma = 4$. We pick $f$, $\eta$, and $g$ such that the true state is, in polar coordinates,
\begin{equation*}
    u(r, \theta) = r^{\frac{2}{3}}\sin(\frac{2}{3}(\theta + \frac{\pi}{2})).
\end{equation*}

\begin{figure}
    \centering
    \begin{subfigure}[t]{0.48\textwidth}
        \includegraphics[width=\textwidth]{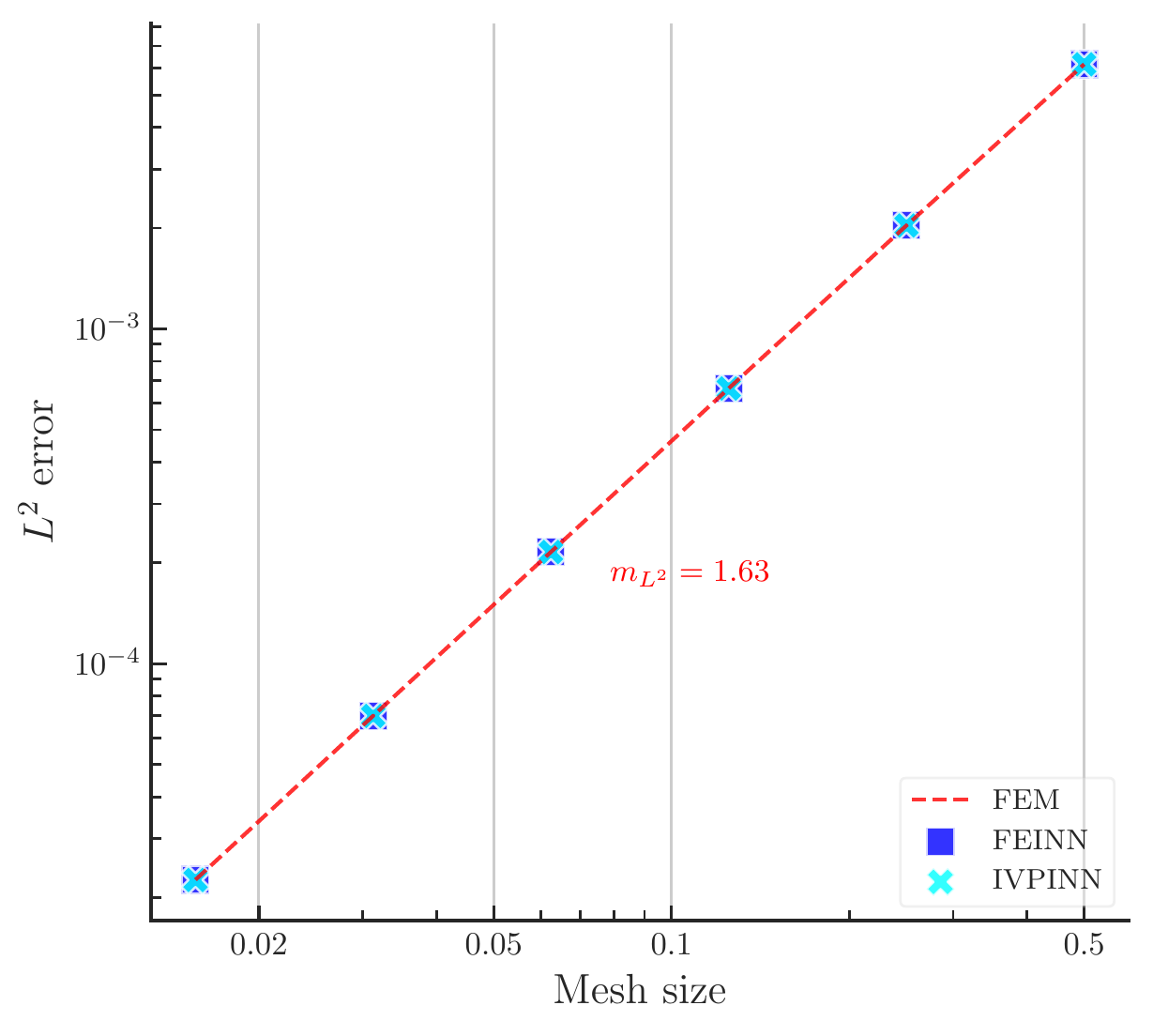}
        \caption{}
        \label{fig:advection_diffusion_h_refinement_less_smooth_l2_err}
    \end{subfigure}
    \begin{subfigure}[t]{0.48\textwidth}
        \includegraphics[width=\textwidth]{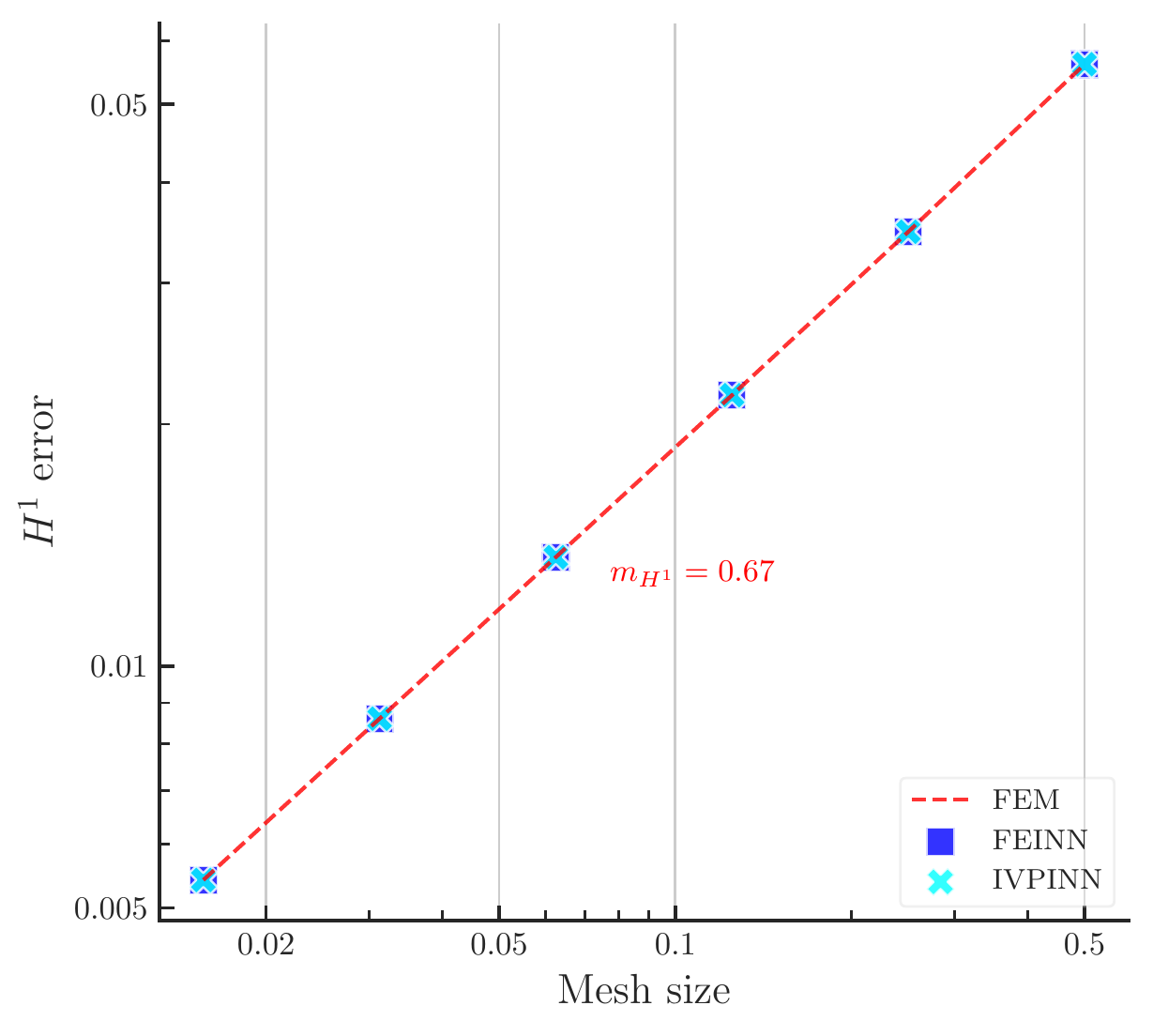}
        \caption{}
        \label{fig:advection_diffusion_h_refinement_less_smooth_h1_err}
    \end{subfigure}
     
    \caption{Convergence of errors with respect to the mesh size of the trial space for the forward convection-diffusion-reaction problem with a singular solution.}
    \label{fig:advection_diffusion_h_refinement_singular}
\end{figure}

Since $u \in H^{5/3 - \epsilon}(\Omega)$ for any $\epsilon > 0$, the expected $H^1$ error decay rate is around $2/3$. Consequently, increasing $k_U$ is unlikely to effectively reduce the error. Therefore, we keep $k_U = 2$, and focus our study on the impact of mesh refinement on error reduction. Fig.~\ref{fig:advection_diffusion_h_refinement_singular} depicts how $L^2$ and $H^1$ errors decay as we increase the mesh size. The errors of the non-interpolated \acp{nn} are not displayed in the plots, because their performance is relatively poor. This observation is consistent with previous findings in ~\cite{BerroneIVPINN2022}, which highlight the inferior performance of \acp{pinn} and \acp{vpinn} compared to \acp{ivpinn} in this singular solution scenario. Fig~\ref{fig:advection_diffusion_h_refinement_singular} show how the $L^2$ and $H^1$ errors change as the mesh size changes. We obtain the expected error decay rate in Fig.~\ref{fig:advection_diffusion_h_refinement_less_smooth_h1_err}. We conclude that both \acp{feinn} and \acp{ivpinn} perform well in addressing this singular problem, and they successfully overcome the limitations in \acp{nn} in this particular situation.

\subsubsection{The effect of preconditioning on Poisson equation with a singular solution}
In this experiment, we investigate whether preconditioning can effectively accelerate the training process, and examine the potential of leveraging widely used \ac{gmg} preconditioners from \ac{fem} to aid in the training of \acp{feinn}.
 
We only consider the $L^2$-norm of the preconditioned loss, i.e., \eqref{eq:preconditioner_l2}. At a purely algebraic level, we can rewrite \eqref{eq:preconditioner_l2} as:
\begin{equation} \label{eq:preconditioner_l2_discrete}
    \mathscr{L}(\pmb{\theta}_u) = \norm{\mathbf{B}^{-1}\mathbf{A}\mathbf{u}_h(\pmb{\theta}_u) - \mathbf{B}^{-1}\mathbf{f}}_{\ell^2},
\end{equation}
where $\mathbf{B}$ is the preconditioner, $\mathbf{A}$ is the coefficient matrix resulting from discretisation, $\pmb{\theta}_u$ are the parameters for $u_{\mathcal{N}}$, $\mathbf{u}_h$ is the vector of \acp{dof} of $\tilde{U}_h$, and $\mathbf{f}$ is the \ac{rhs} vector.
(We note that, since the mesh being used is (quasi-)uniform, we can replace the $L^2$-norm by the Euclidean $\ell^2$-norm; they differ by a scaling.)

We consider three types of preconditioners. The first one is $\mathbf{B}_{\rm inv} = \mathbf{A}$. Plugged into \eqref{eq:preconditioner_l2_discrete}, the loss becomes $\norm{\mathbf{u}_h(\pmb{\theta}_u) - \mathbf{A}^{-1}\mathbf{f}}_{\ell^2}$. This loss resembles the loss in data fitting tasks, and it should theoretically be easier for \acp{nn} to minimise. In the implementation, we compute a LU-decomposition of the preconditioner $\mathbf{B}_{\rm inv}$. Since the problem is steady-state and linear, we can reuse the factorisation. So, at each iteration during the training, one objective function evaluation requires one forward and one backward substitution. Since $\mathbf{B}_{\rm inv}$ is sparse, the cost for each substitution has linear complexity.

We also consider another preconditioner $\mathbf{B}_{\rm inv\_lin}$, which is defined as the matrix resulting from discretisation with $V_h$ as both trial and test \ac{fe} spaces. Note that $V_h$ is built out of a mesh resulting from the application of $k_U$ levels of uniform refinement to the mesh associated to $U_h$, since $k_V = 1$. The preconditioner $\mathbf{B}_{\rm inv\_lin}$ is computationally cheaper to invert than $\mathbf{B}_{\rm inv}$, since it is \ac{spd} and involves linear \ac{fem} bases only. The last (and cheapest to invert) one is, as mentioned before, a \ac{gmg} preconditioner $\mathbf{B}_{\rm GMG}$ of $\mathbf{B}_{\rm inv\_lin}$. The application of $\mathbf{B}_{\rm GMG}$ scales linearly with the number of \acp{dof} of the \ac{fe} space.

We now change to the Poisson equation. The problem is defined on $\Omega = [0,1]^2$, with $\Gamma_{\rm D} = \partial \Omega$ and $\kappa = 1$. Choose $f$ and $g$ such that the true state is the same as the singular $u$ in Sec.~\ref{subsubsec:singular_cdr_eq_exp}. We divide the domain uniformly into $64\times 64$ quadrilaterals, and then employ $k_U = 2$ or $k_U = 4$ for the \ac{nn} interpolation space. To evaluate the effectiveness of the aforementioned preconditioners, we perform four experiments for each order. Three of them employ the $\mathbf{B}_{\rm inv}$, $\mathbf{B}_{\rm inv\_lin}$, and $\mathbf{B}_{\rm GMG}$ preconditioners, respectively, while the fourth experiment serves as a baseline without any preconditioning, denoted as $\mathbf{B}_{\rm none}$. In all experiments, we use the same initial parameters for the \acp{nn} to ensure a fair comparison.

\begin{figure}
    \centering
    \begin{subfigure}[t]{0.48\textwidth}
        \includegraphics[width=\textwidth]{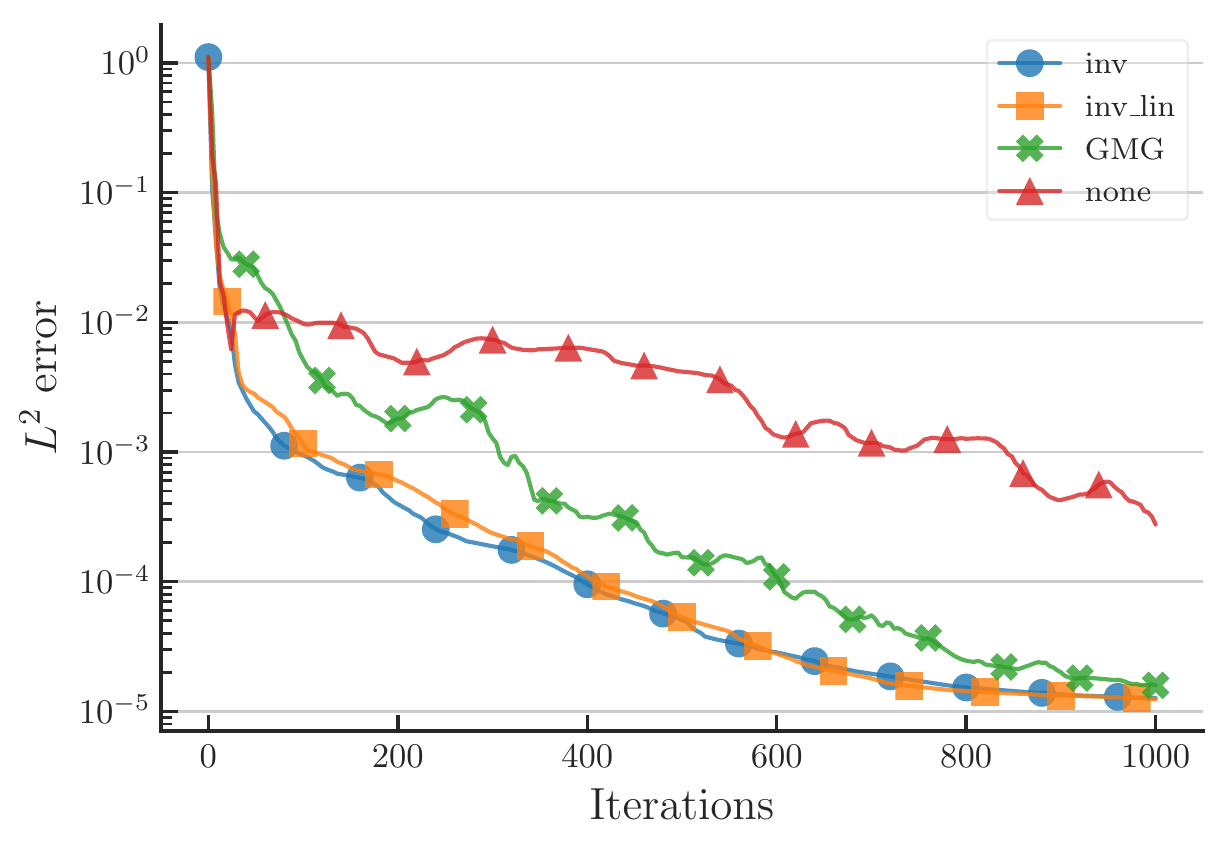}
        \caption{$k_U = 2$}
        \label{fig:poisson_preconditioners_comparison_order_2}
    \end{subfigure}
    \begin{subfigure}[t]{0.48\textwidth}
        \includegraphics[width=\textwidth]{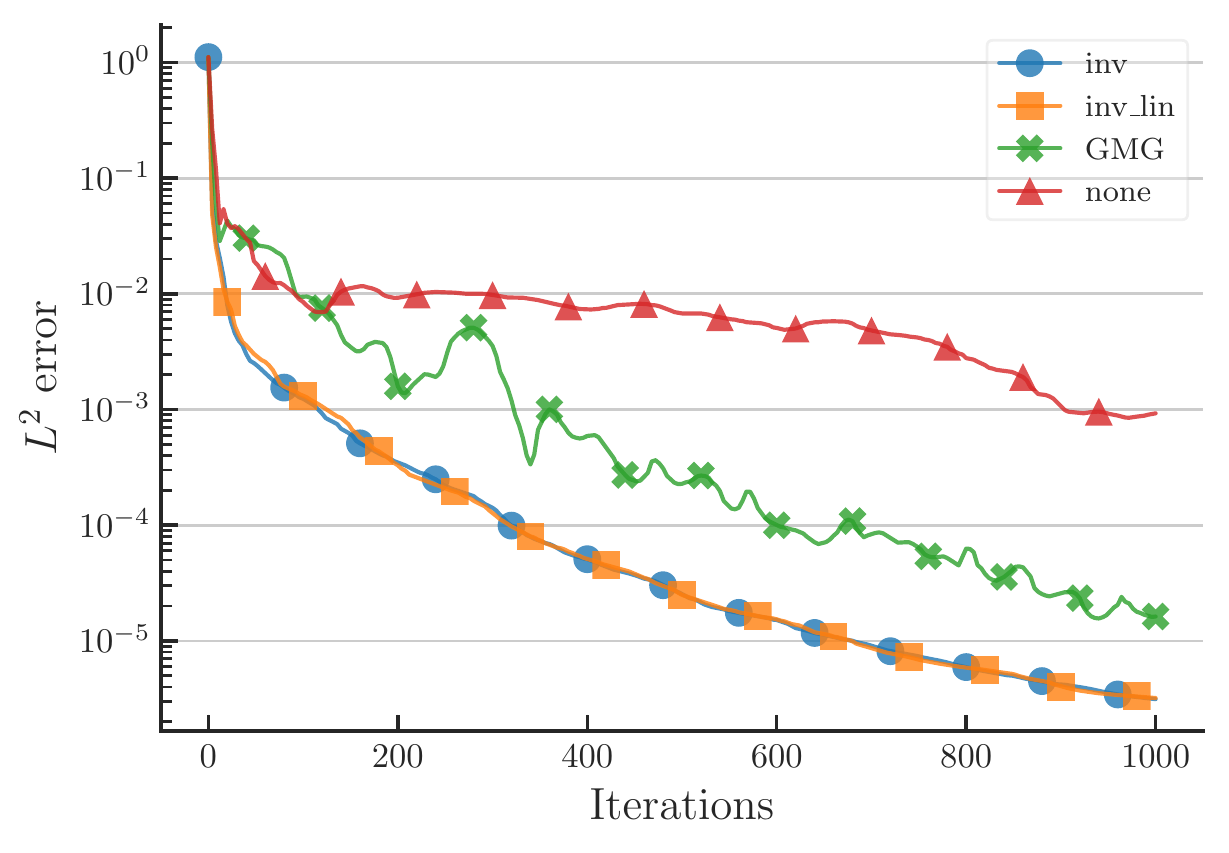}
        \caption{$k_U = 4$}
        \label{fig:poisson_preconditioners_comparison_order_4}
    \end{subfigure}
     
    \caption{$L^2$ error history during training of \acp{feinn} for the forward Poisson problem with a singular solution using different preconditioners.}
    \label{fig:poisson_preconditioners_comparison}
\end{figure}

Fig. ~\ref{fig:poisson_preconditioners_comparison} shows the $L^2$ error history of \acp{feinn} using different preconditioners during training for the first 1,000 iterations. We can extract several findings from the figure. Firstly, as $k_U$ increases, the training for the unpreconditioned loss becomes more challenging. This is evident from the flatter error curve for $\mathbf{B}_{\rm none}$ in Fig.~\ref{fig:poisson_preconditioners_comparison_order_4} compared to Fig.~\ref{fig:poisson_preconditioners_comparison_order_2}. Then, the preconditioners contribute to faster convergence as the error curves of the preconditioned \acp{feinn} are much steeper compared to the one without any preconditioner. Next, the cheaper $\mathbf{B}_{\rm inv\_lin}$ preconditioner is surprisingly as effective as the $\mathbf{B}_{\rm inv}$ preconditioner. Lastly, $\mathbf{B}_{\rm GMG}$ leads to a substantial acceleration of $L^2$ convergence for both order 2 and order 4. Specifically, in Fig.~\ref{fig:poisson_preconditioners_comparison_order_2}, for $k_U = 2$, the standard unpreconditioned loss function requires more than 800 iterations to reduce the $L^2$-error below $10^{-3}$, while $\mathbf{B}_{\rm inv\_lin}$ and $\mathbf{B}_{\rm inv}$ loss functions attain the same error in around 100 iterations and $\mathbf{B}_{\rm GMG}$ requires around 300 iterations. Besides, the \ac{gmg}-preconditioned \ac{feinn} achieves an error level that closely matches the \ac{feinn} preconditioned by the other preconditioners after around 900 iterations. Overall, the difference between the errors of these preconditioned \acp{feinn} and the error of the unpreconditioned \ac{feinn} exceeds one order of magnitude after enough iterations, and reach two others of magnitude in many cases. Similarly, for $k_U = 4$ as shown in Fig~\ref{fig:poisson_preconditioners_comparison_order_4}, the unpreconditioned case requires around 1,000 iterations to reduce the $L^2$-error below $10^{-3}$, $\mathbf{B}_{\rm GMG}$ preconditioned \ac{feinn} needs around 300 iterations, and $\mathbf{B}_{\rm inv\_lin}$ and $\mathbf{B}_{\rm inv}$ preconditioned \acp{feinn} only require around 100 iterations. In this second case, the difference between preconditioned and unpreconditioned training exceeds two orders of magnitude. Although the \ac{gmg} preconditioner may not be as effective as the other preconditioners, it still exhibits a remarkable reduction in the $L^2$ error compared to the unpreconditioned \ac{feinn}, reaching approximately two orders of magnitude after 400 iterations, while being a very cheap preconditioner.

\subsubsection{Poisson equation on a complex geometry} \label{subsubsec:poisson_complex_geometry}
In this section, we demonstrate the capabilities of \acp{feinn} in solving forward Poisson problems defined on general domains. We focus on a slightly modified version of ~\cite[Example~4]{PoissonIrregularDomain2002}. As shown in Fig.~\ref{fig:poisson_complex_geo_true_state}, the computational domain $\Omega$ features a bone-shaped region, which is parameterised by $(x(\theta), y(\theta))$. The parametric equations are defined as $x(\theta) = 0.6\cos(\theta) - 0.3\cos(3\theta)$ and $y(\theta) = 0.7\sin(\theta) - 0.07\sin(3\theta) + 0.2\sin(7\theta)$ with $\theta \in [0, 2\pi]$. Finding appropriate $\Phi$ and $\bar{u}$ for \acp{ivpinn} is very challenging for this irregular domain, so we only examine the performance of \acp{feinn} in our experiments. We consider the Poisson problem with $\kappa(x,y) = 2 + \sin(xy)$, $\Gamma_{\rm D} = \partial \Omega$. We choose $f$ such that the solution is $u(x,y) = \mathrm{e}^x(x^2\sin(y) + y^2)$.

In this study, our focus is on examining the impact of mesh refinement on \acp{feinn}. Consequently, we fix $k_U = 2$ and discretise $\Omega$ using unstructured triangular meshes with an increasing number of cells. In the loss function, we employ the $\ell^1$-norm for the residual vector. Although the $\ell^2$-norm is equally effective, we aim to showcase the flexibility in choosing the norm and to provide evidence supporting the suitability of the $\ell^1$ norm for the \ac{pde} loss. This is particularly relevant, as we consistently use the $\ell^1$ norm for the \ac{pde} part in the subsequent experiments for inverse problems. Similar to the previous sections, we conduct 10 experiments for each mesh resolution, each with distinct initialisations of the \acp{nn}. 

\begin{figure}
    \centering
    \begin{subfigure}[t]{0.48\textwidth}
        \includegraphics[width=\textwidth]{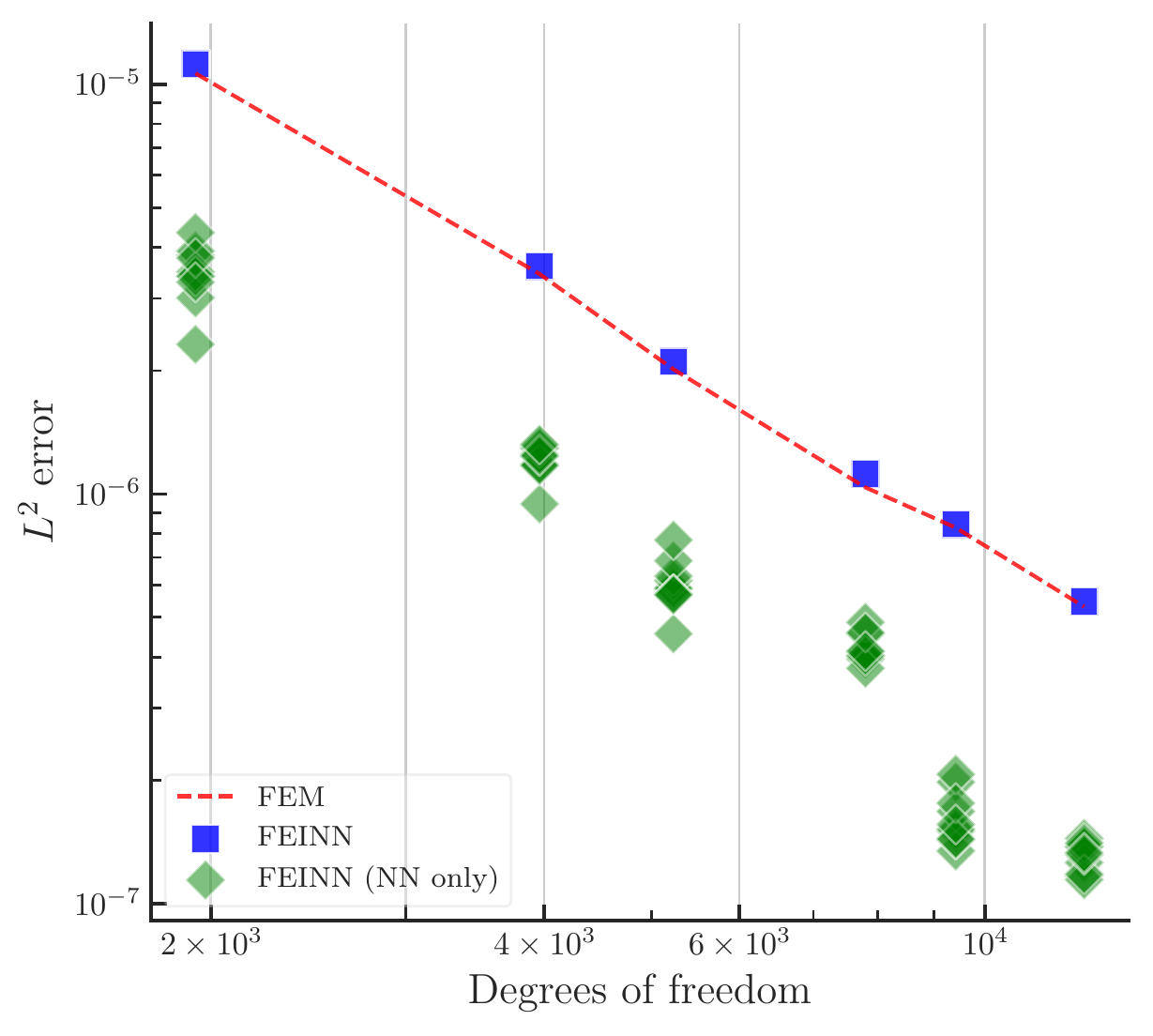}
        \caption{}
        \label{fig:poisson_h_refinement_complex_geo_l2_err}
    \end{subfigure}
    \begin{subfigure}[t]{0.48\textwidth}
        \includegraphics[width=\textwidth]{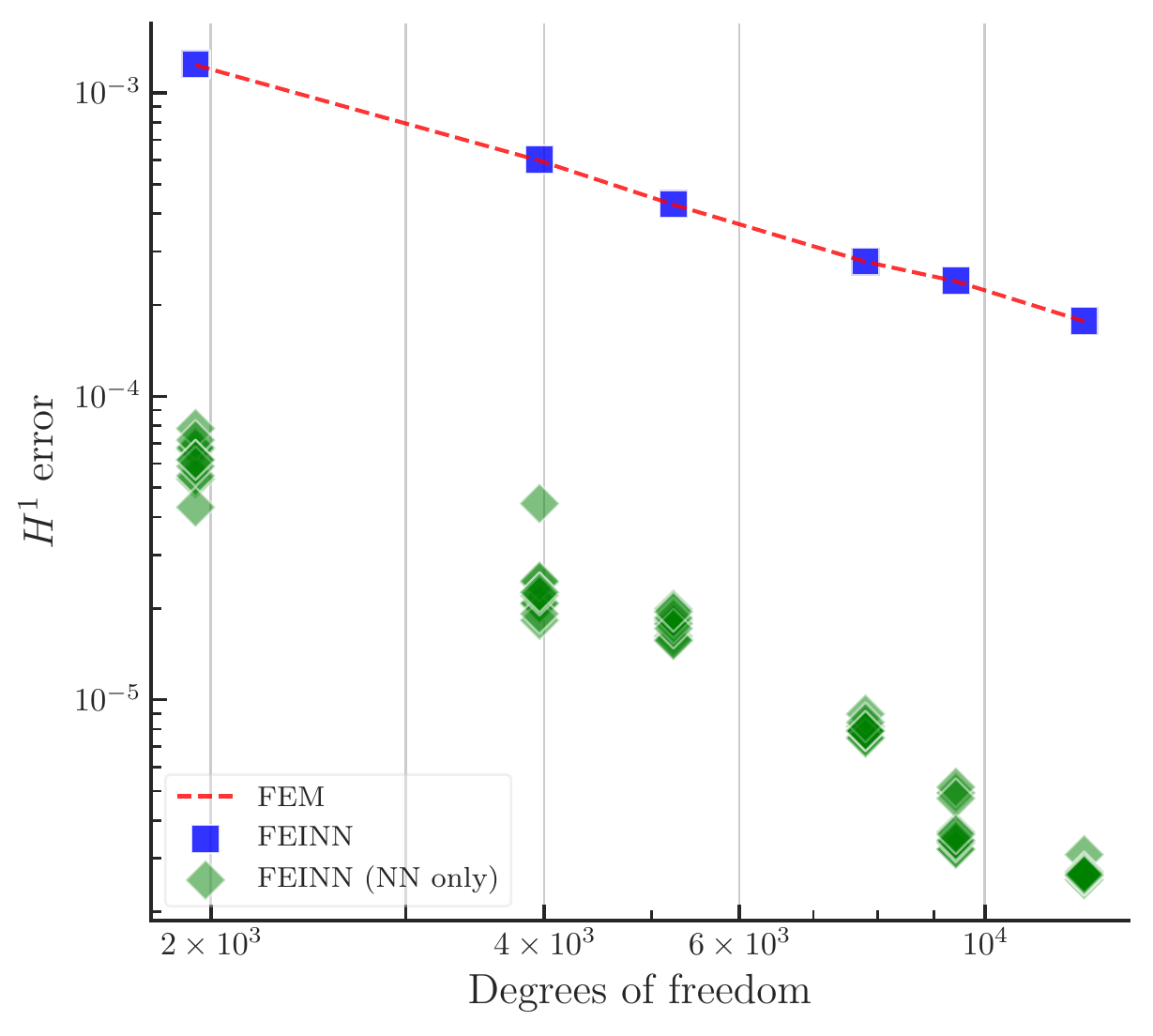}
        \caption{}
        \label{fig:poisson_h_refinement_complex_geo_h1_err}
    \end{subfigure}
     
    \caption{Convergence of errors with respect to \acp{dof} of the trial space for the forward Poisson problem on a bone-shaped geometry.}
    \label{fig:poisson_h_refinement_complex_geo}
\end{figure}

\begin{figure}
    \centering
    \begin{subfigure}[t]{0.32\textwidth}
        \includegraphics[width=\textwidth]{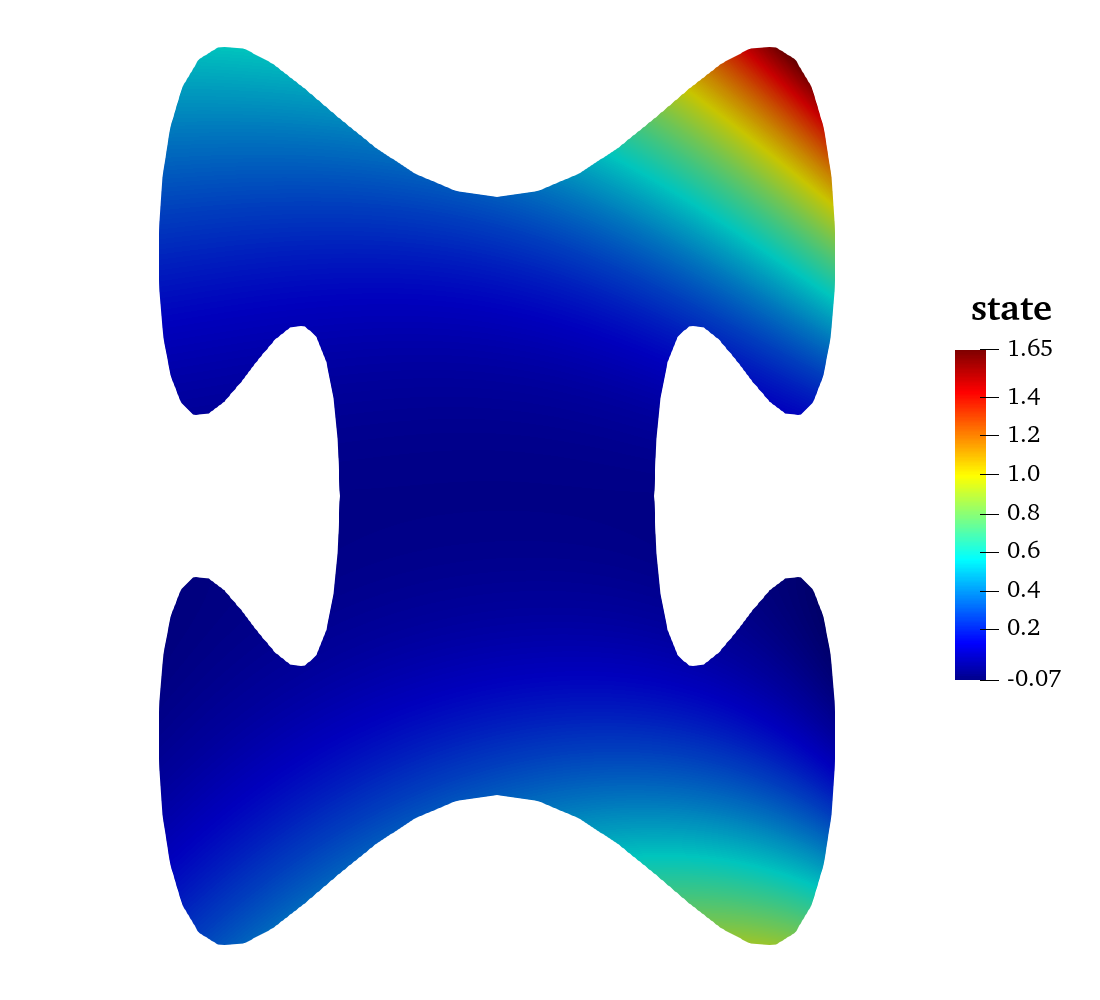}
        \caption{$u$}
        \label{fig:poisson_complex_geo_true_state}
    \end{subfigure}
    \begin{subfigure}[t]{0.32\textwidth}
        \includegraphics[width=\textwidth]{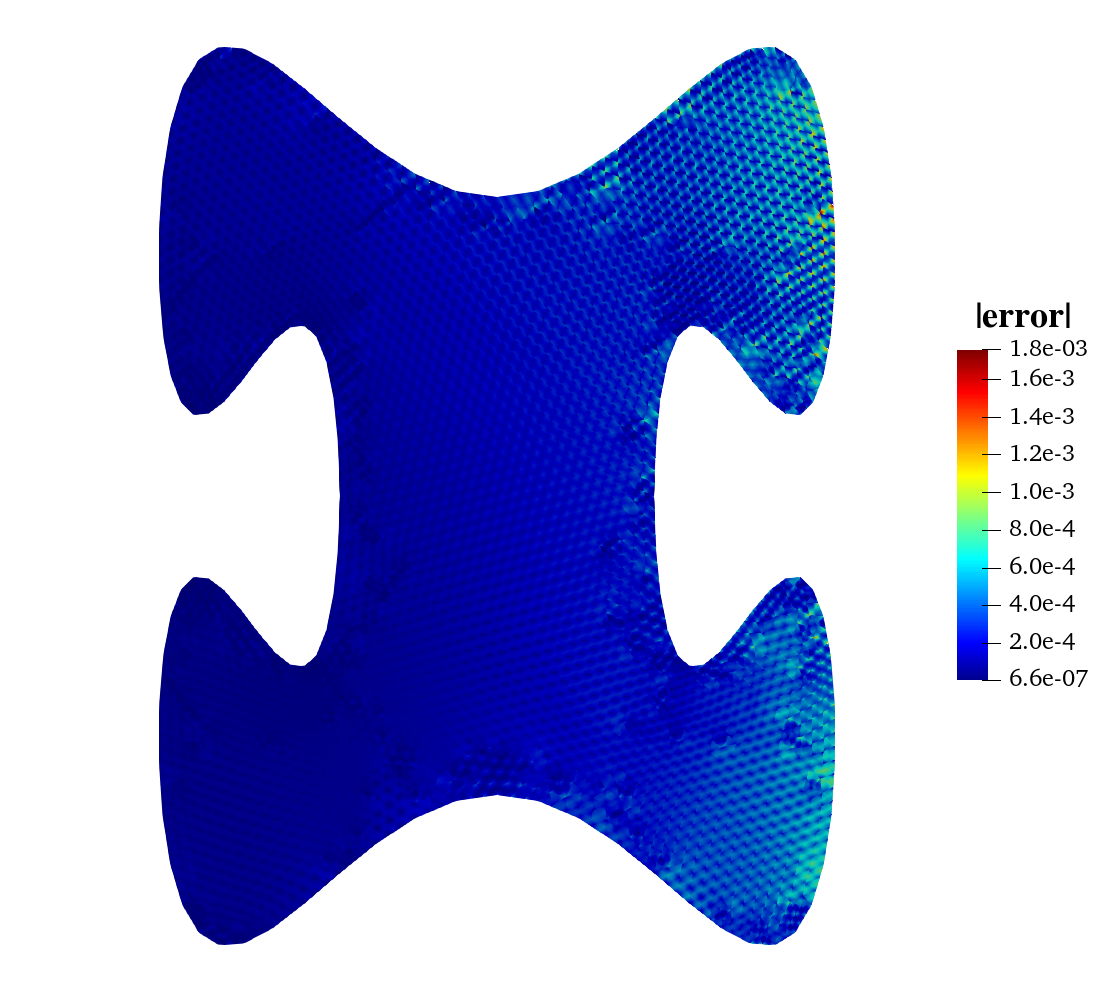}
        \caption{$|\pmb{\nabla}(u - (\tilde{\pi}_h(u_{\mathcal{N}}) + \bar{u}_h))|$}
        \label{fig:poisson_complex_geo_feinn_state_error}
    \end{subfigure}
    \begin{subfigure}[t]{0.32\textwidth}
        \includegraphics[width=\textwidth]{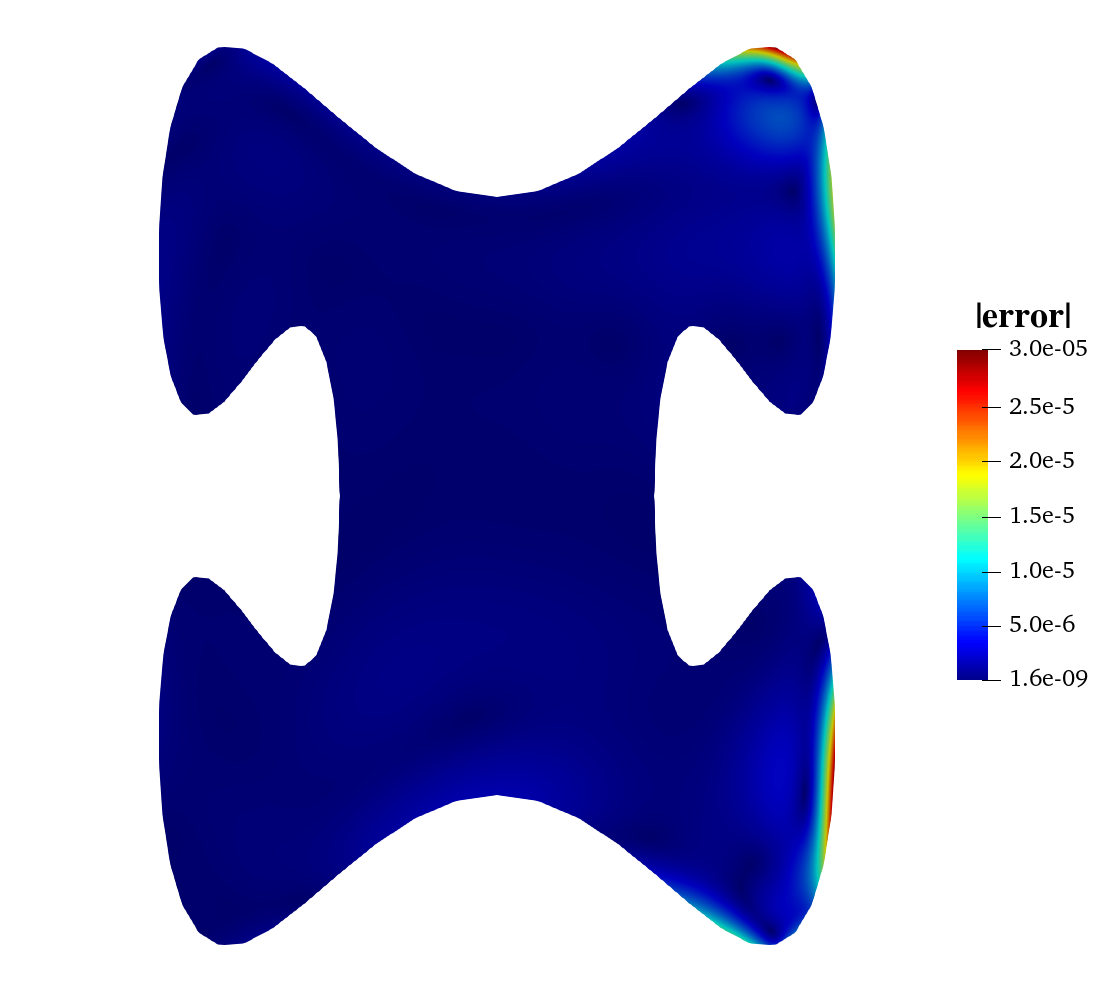}
        \caption{$|\pmb{\nabla}(u - u_{\mathcal{N}})|$}
        \label{fig:poisson_complex_geo_nn_state_error}
    \end{subfigure}
     
    \caption{True state and gradient error magnitude in \ac{feinn} and \ac{nn} solutions for the forward Poisson problem on a bone-shaped domain.}
    \label{fig:poisson_complex_geo_true_state_and_errors}
\end{figure}

Fig.~\ref{fig:poisson_h_refinement_complex_geo} illustrates the changes in $L^2$ and $H^1$ errors as the \acp{dof} in the \ac{fe} interpolation space increase. 
Overall, \acp{feinn} demonstrate almost identical performance to \ac{fem}. 
Importantly, similar to the findings for the forward convection-diffusion-reaction problem with a smooth solution, the non-interpolated \acp{nn} consistently outperform \ac{fem}. Notably, when the mesh is ``fine enough'', there is a remarkable two-order-of-magnitude difference in $H^1$ errors between the \acp{nn} and \ac{fem}, as illustrated in Fig.~\ref{fig:poisson_h_refinement_complex_geo_h1_err}.

To further confirm the superior performance of the \acp{nn} in terms of $H^1$ error, we present the point-wise gradient error magnitudes for the \ac{feinn} solution and the \ac{nn} solution in Fig.~\ref{fig:poisson_complex_geo_feinn_state_error} and \ref{fig:poisson_complex_geo_nn_state_error}, respectively. These figures correspond to one of our experiments conducted on the finest mesh. Notably, we observe a significant two-order-of-magnitude reduction in error magnitude for the \ac{nn} solution compared to the \ac{feinn} solution across most regions of the domain. Additionally, the lack of smoothness of the gradient of the interpolated solution in  \ac{feinn} on a low order $\mathcal{C}^0$ space is evident in Fig.~\ref{fig:poisson_complex_geo_feinn_state_error}. In contrast, one can observe the smoothness of the error of the \ac{feinn} trained \ac{nn} in Fig.~\ref{fig:poisson_complex_geo_nn_state_error}.

\subsection{Inverse problems}\label{sec:inverse}
In the experiments for inverse problems, we introduce the following relative $L^2$ and $H^1$ errors to measure the accuracy of an identified solution $z^{id}$:
\begin{equation*}
    \varepsilon_{L^2(\Omega)}(z^{id}) = \frac{\ltwonorm{z^{id} - z}}{\ltwonorm{z}}, \qquad 
    \varepsilon_{H^1(\Omega)}(z^{id}) = \frac{\honenorm{z^{id} - z}}{\honenorm{z}},
\end{equation*}
where $z$ is the ground truth.

The optimisation involving the penalty term (see (\ref{eq:inverse_loss})) occurs at Step 3, requiring the selection of the norm for $\mathcal{R}_h$ and the corresponding coefficient $\alpha$. In ~\cite[Ch.~17]{nocedal2006numerical}, the authors provide insights into the distinction between utilising $\ell^1$ and $\ell^2$ norms. According to ~\cite[Theorem~17.1]{nocedal2006numerical}, when employing the $\ell^2$-norm for $\mathcal{R}_h$, the minimiser of ~\eqref{eq:inverse_min} becomes a global solution to the inverse problem as $\alpha$ approaches to infinity. Furthermore, ~\cite[Theorem~17.3]{nocedal2006numerical} states that there exists an $\alpha^*$ such that the minimiser in ~\eqref{eq:inverse_min} for the $\ell^1$-norm of $\mathcal{R}_h$ is compelled to coincide with the solution of the inverse problem for any $\alpha \geq \alpha^*$. To avoid choosing an arbitrarily large $\alpha$, we opt to use the $\ell^1$ norm. Moreover, the authors propose ~\cite[Framework~ 17.2]{nocedal2006numerical} for adjusting the coefficient $\alpha$. Following this, we partition Step 3 into several sub-steps. We use a sequence of $\{ \alpha_k \}$ for these sub-steps, where $\alpha_k > \alpha_{k-1}$ for $k > 1$.\footnote{Against common experience in the inverse problem community~\cite{InversePenaltyMethod2015}, penalty coefficients for the PDE residual term are usually considered fixed in \acp{pinn} and related methods (see \cite{PINNs2019,Kharazmi2021}). Similarly, for forward problems, the Dirichlet penalty term (which is also a constraint in the minimisation) is usually kept fixed in these formulations.}

As mentioned before, we split the training process into three steps.  Although we have extensively tested training only \eqref{eq:inverse_loss_discrete}, the three-step strategy consistently yielded superior results. As a result, all the experiments in this section will follow this training process. We introduce the notation $[n_1, n_2, k \times n_3]$ to represent the number of iterations for each step: $n_1$ iterations for the data fitting step, followed by $n_2$ iterations for the model parameter initialisations step, and $k$ sub-steps in the coupled step, with each sub-step consisting of $n_3$ iterations. The sub-steps simply represent a new value of the penalty coefficient. We use the notation $\alpha = [\alpha_1,\alpha_2,...,\alpha_k]$, where $\alpha_1$, $\alpha_2$,..., and $\alpha_k$  are the penalty coefficients at each sub-step.

We employ the softplus activation function for \acp{feinn} in our inverse problem experiments, even though the $\tanh$ activation generally performs comparably or even better. We aim to explore alternative activation functions for \acp{nn} in the context of solving \ac{pde}-constrained problems using \acp{feinn}. We use linear \ac{fe} interpolation space for \acp{feinn}.\footnote{Inverse problems are ill-posed and affected by partial knowledge of the problem and noisy observations. High-order approximations are not necessary or even practical in these situations.} 
In the remaining experiments, unless otherwise specified, we consider $z^{id}$ to be the \ac{fe} interpolation of the \ac{nn} $z_{\mathcal{N}}$. 

\subsubsection{Poisson equation with partial observations} \label{subsubsec:partial_observation}
We begin our inverse problem experiments with a Poisson equation involving partial observations. Following the experiment presented in ~\cite[Sec. 3.1.3]{Hybrid_FEM-NN_2021}, we consider the computational domain $[0,1]^2$ with Dirichlet boundary conditions on the left, bottom, and top sides, and a Neumann boundary condition on the right side. The unknown state and diffusion coefficient (Fig.~\ref{fig:partial_observation_true_coefficient}) are:
\begin{equation*}
    u(x,y) = \sin(\pi x) \sin(\pi y), \qquad 
    \kappa(x,y) = 1+0.5\sin(2\pi x)\sin(2\pi y).
\end{equation*}
Fig.~\ref{fig:partial_observation_true_state} illustrates the true state, and our observations are limited to every \ac{dof} inside the white box located at the center of the figure. The objectives of this experiment are to reconstruct the partially known state and to recover the unknown diffusion coefficient. 

We discretise the domain by $50 \times 50$ quadrilaterals. Both \acp{nn}, $u_{\mathcal{N}}$ and $\kappa_{\mathcal{N}}$, have the same structure with $L = 2$ layers and each hidden layer has $n = 20$ neurons. To ensure the positivity of the diffusion coefficient, we apply a rectification function $r(x) = |x| + 0.01$ as the activation for the output layer of $\kappa_{\mathcal{N}}$. Although $r(x) = x^2 + 0.01$ also produces satisfactory results, it is more common to use an output layer with linear features. The training iterations are $[400, 400, 3 \times 400]$, and the penalty coefficients are $\alpha = [0.1, 0.3, 0.9]$.

\begin{figure}
    \centering
    \begin{subfigure}[t]{0.32\textwidth}
        \includegraphics[width=\textwidth]{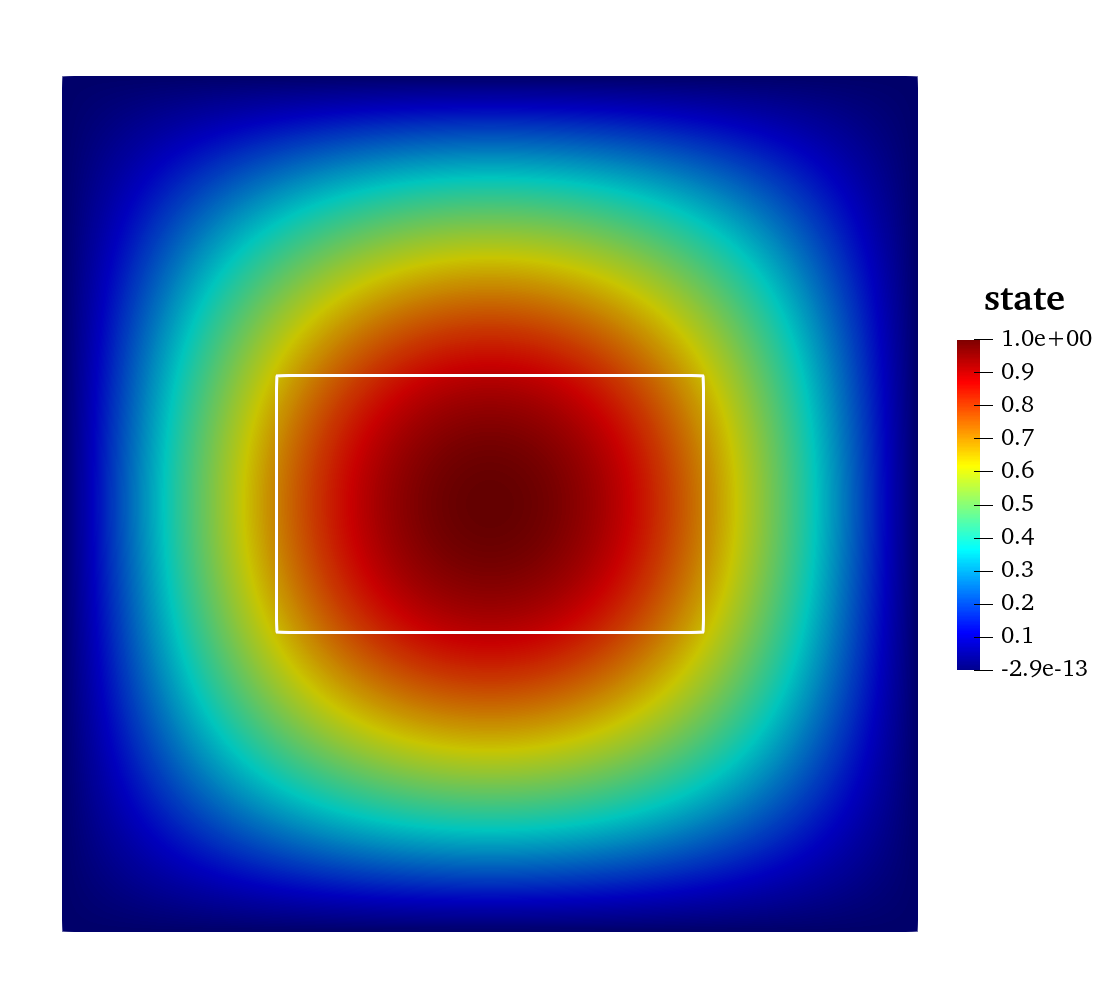}
        \caption{$u$}
        \label{fig:partial_observation_true_state}
    \end{subfigure}
    \begin{subfigure}[t]{0.32\textwidth}
        \includegraphics[width=\textwidth]{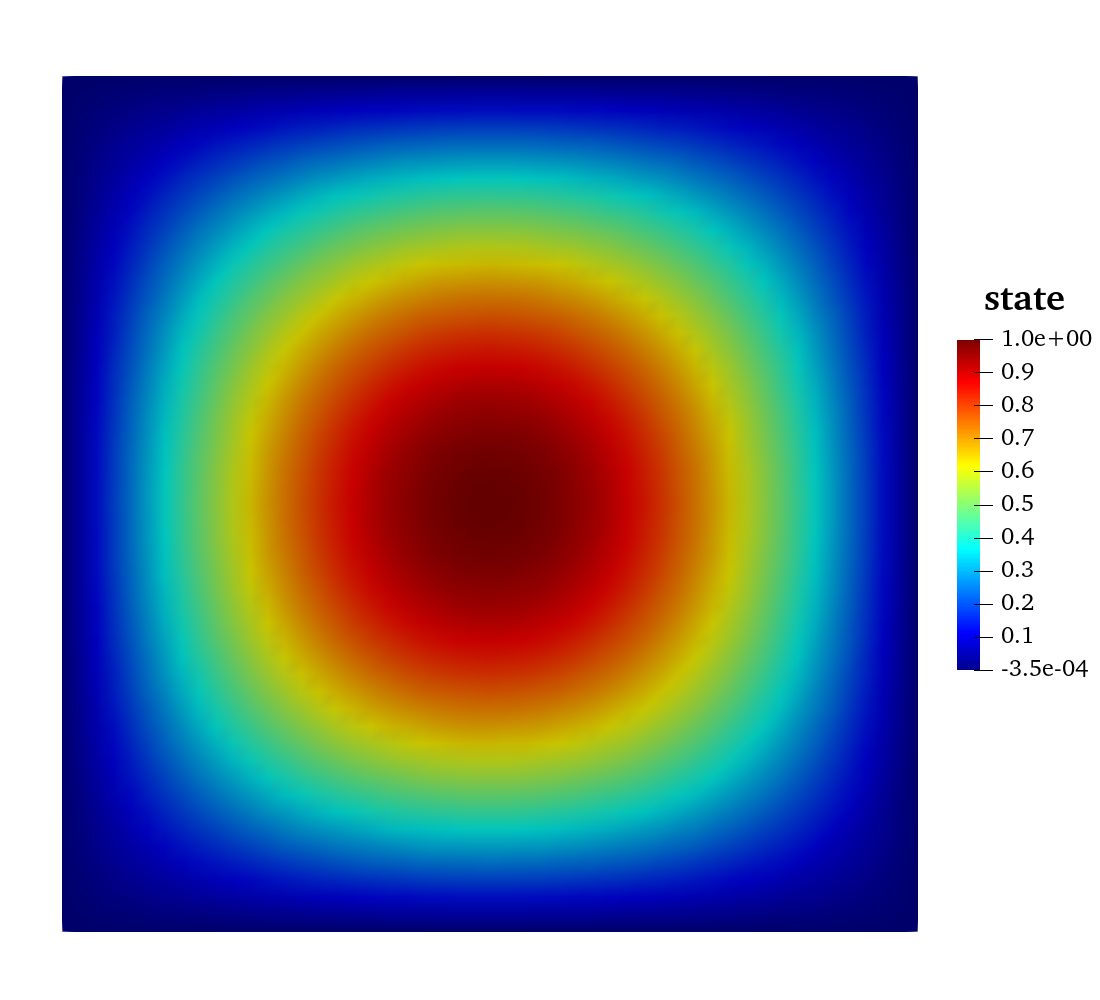}
        \caption{$u^{id}$}
        \label{fig:partial_observation_feinn_state}
    \end{subfigure}
    \begin{subfigure}[t]{0.32\textwidth}
        \includegraphics[width=\textwidth]{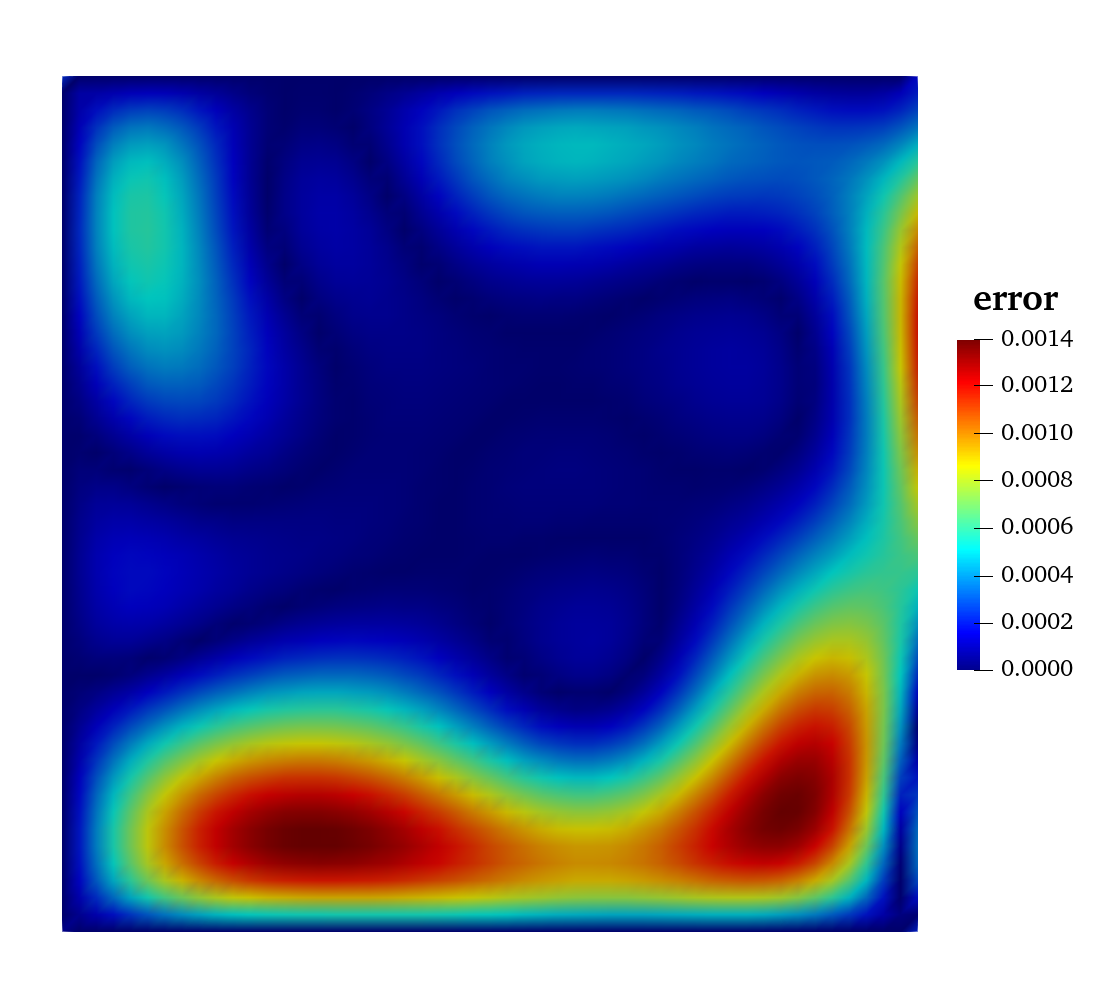}
        \caption{$|u - u^{id}|$}
        \label{fig:partial_observation_state_error}
    \end{subfigure}

    \centering
    \begin{subfigure}[t]{0.32\textwidth}
        \includegraphics[width=\textwidth]{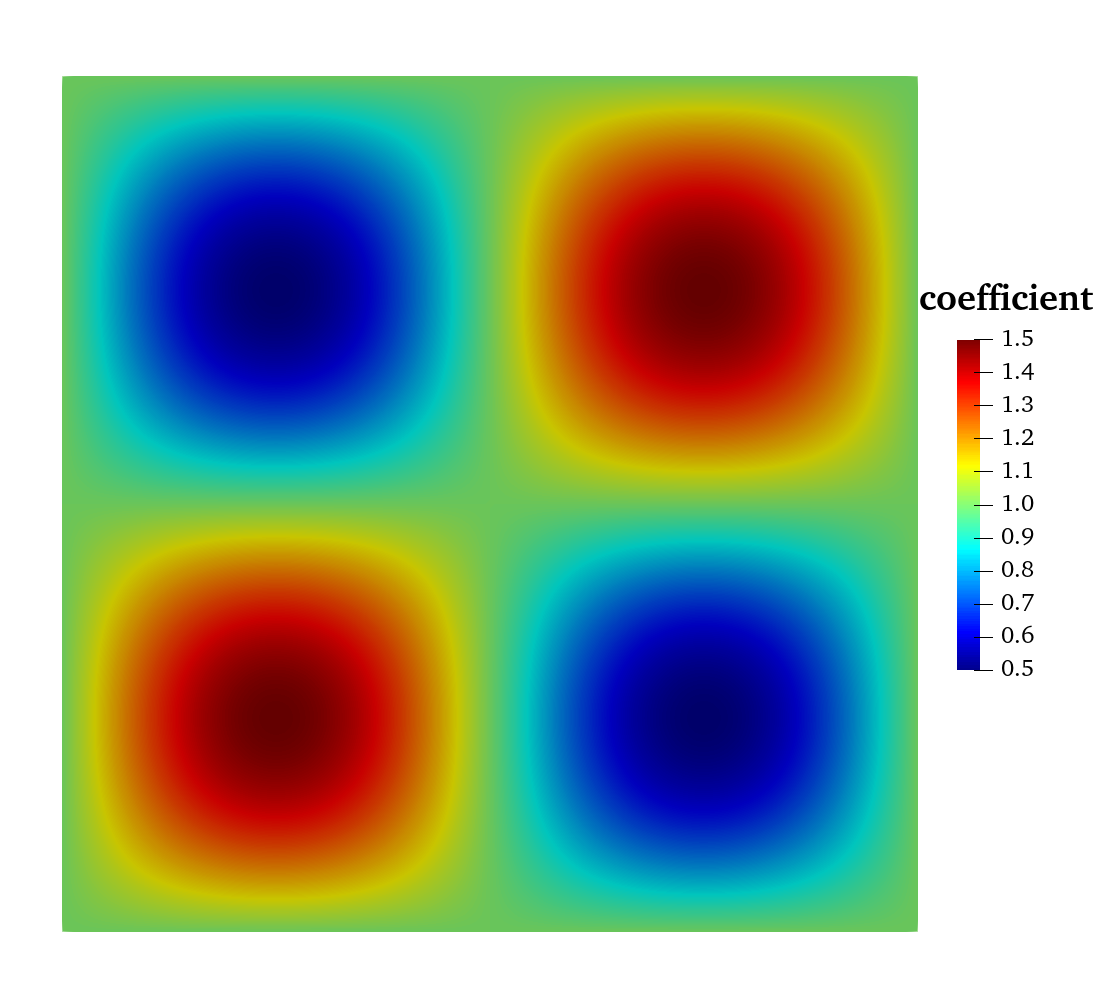}
        \caption{$\kappa$}
        \label{fig:partial_observation_true_coefficient}
    \end{subfigure}
    \begin{subfigure}[t]{0.32\textwidth}
        \includegraphics[width=\textwidth]{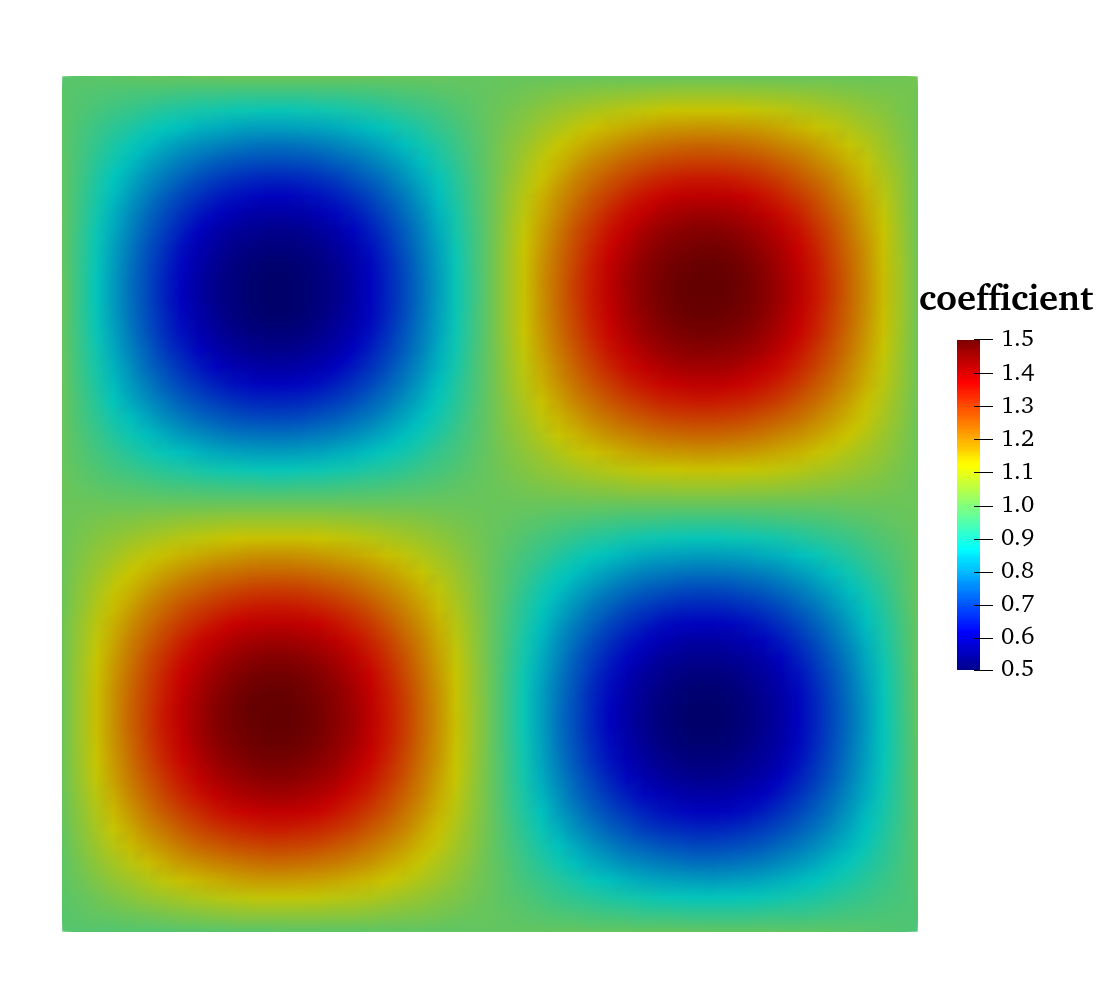}
        \caption{$\kappa^{id}$}
        \label{fig:partial_observation_feinn_coefficient}
    \end{subfigure}
    \begin{subfigure}[t]{0.32\textwidth}
        \includegraphics[width=\textwidth]{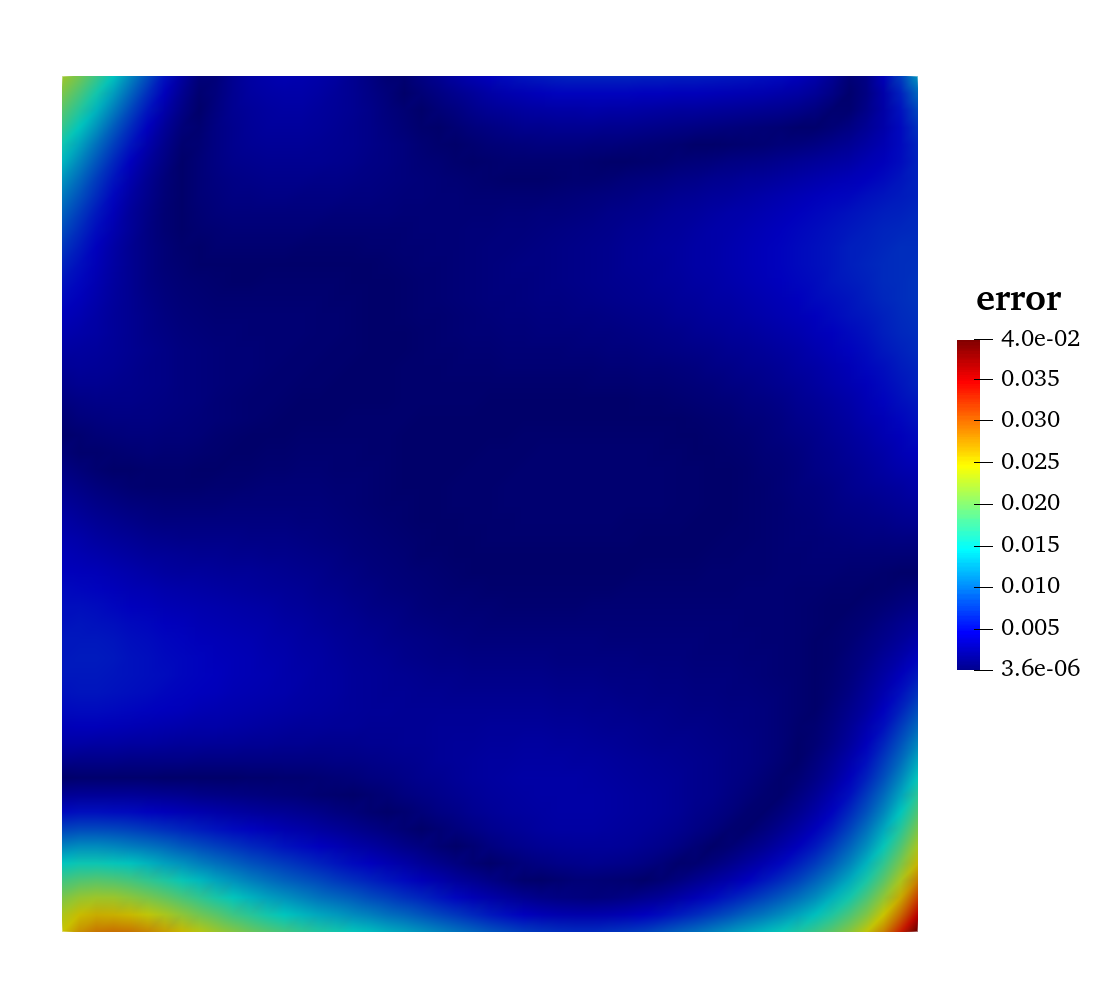}
        \caption{$|\kappa - \kappa^{id}|$}
        \label{fig:partial_observation_coefficient_error}
    \end{subfigure}
     
    \caption{Comparison of the true solutions (first column), the \ac{feinn} solutions (second column), and corresponding point-wise errors (third column) for the inverse Poisson problem with partial observations. The presented results are from a specific experiment. The first row depicts the state, while the second row represents the coefficient. The observations of $u$ are limited to the white box.}
    \label{fig:partial_observation_experiment_results}
\end{figure}

In Fig.~\ref{fig:partial_observation_experiment_results}, we display the \acp{feinn} solutions along with their corresponding errors in comparison to the true solutions. The identified state $u^{id}$ in Fig.~\ref{fig:partial_observation_feinn_state} closely resembles the true state $u$ in Fig.~\ref{fig:partial_observation_true_state}, accompanied by very small point-wise errors in Fig.~\ref{fig:partial_observation_state_error}. These observations highlight the effectiveness of \acp{feinn} at completing the partial observations. Fig.~\ref{fig:partial_observation_coefficient_error} displays the small point-wise error of $\kappa^{id}$, further confirming the accuracy of our approach on discovering the unknown diffusion coefficient.

\begin{figure}
    \centering
    \begin{subfigure}[t]{0.32\textwidth}
        \includegraphics[width=\textwidth]{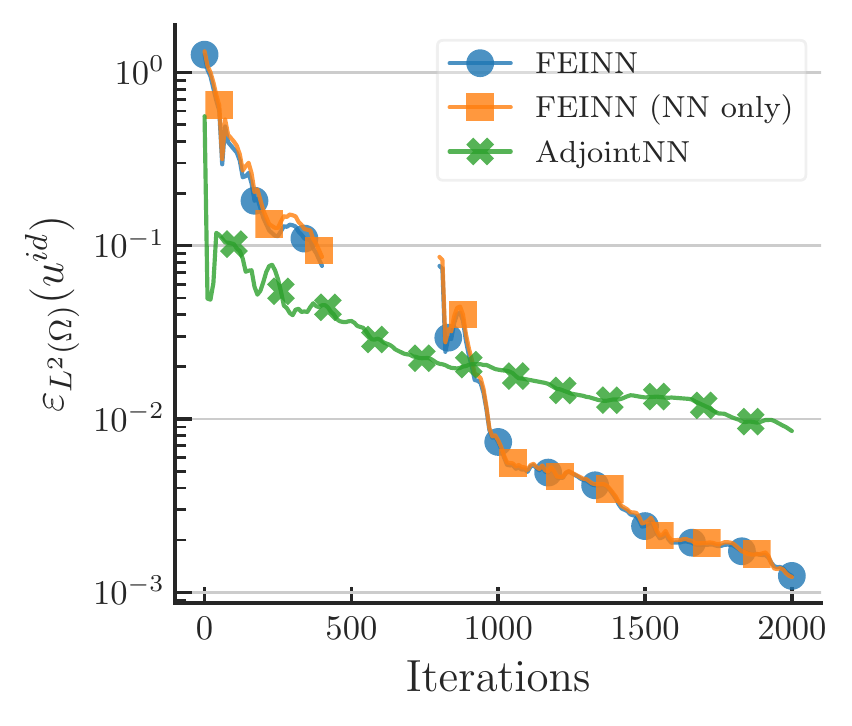}
        \caption{}
        \label{fig:partial_observation_ul2error_feinn_adjoint_nn_cmp}
    \end{subfigure}
    \begin{subfigure}[t]{0.32\textwidth}
        \includegraphics[width=\textwidth]{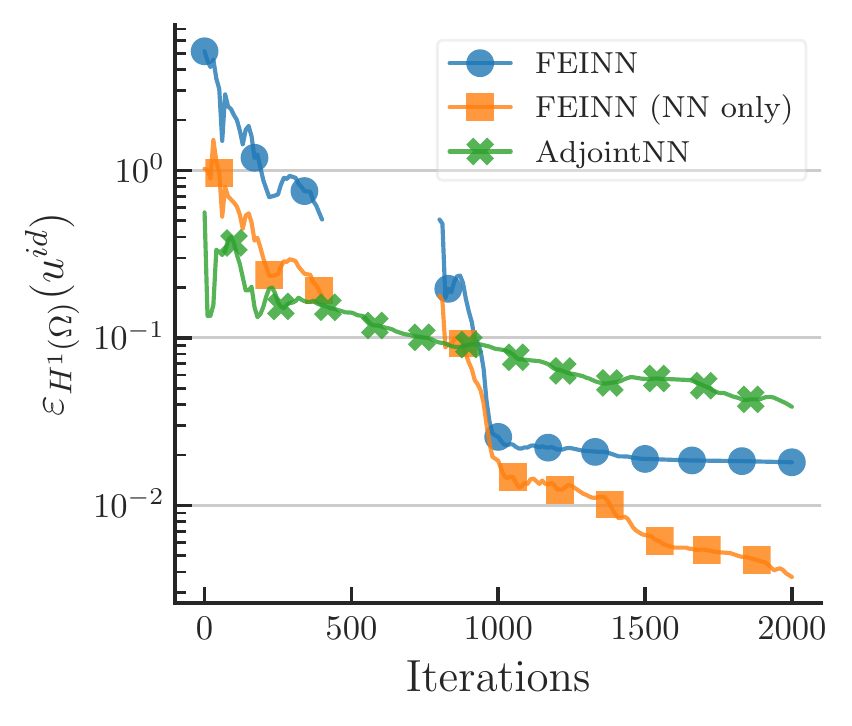}
        \caption{}
        \label{fig:partial_observation_uh1error_feinn_adjoint_nn_cmp}
    \end{subfigure}
    \begin{subfigure}[t]{0.32\textwidth}
        \includegraphics[width=\textwidth]{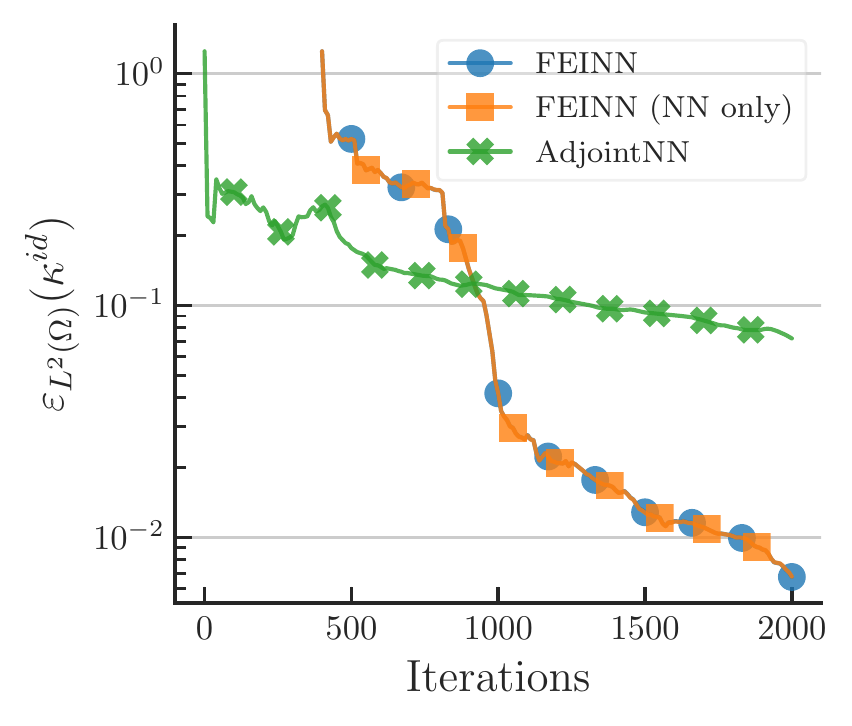}
        \caption{}
        \label{fig:partial_observation_kl2error_feinn_adjoint_nn_cmp}
    \end{subfigure}
     
    \caption{Comparison among \acp{feinn} and adjoint \ac{nn} in terms of relative errors during training for the inverse Poisson problem with partial observations. The optimisation loop was run for 2,000 iterations in both cases.}
    \label{fig:partial_observation_errors_feinn_adjoint_nn_cmp}
\end{figure}

In order to also consider the relative merits of \acp{feinn} compared to other approaches proposed in the literature,  
Fig.~\ref{fig:partial_observation_errors_feinn_adjoint_nn_cmp}
reports the relative error history for the 
state and coefficient throughout the training process 
for \acp{feinn} and our Julia implementation of the adjoint-based \ac{nn} method (adjoint NN).
While adjoint NN approximates the unknown diffusion coefficient with a neural network, 
it still approximates the state using a \ac{fe} space, and uses the adjoint solver to compute the gradient of the data misfit with respect to the \ac{nn} parameters, resulting in a two-loop optimisation process. Adjoint \ac{nn} was first introduced in~\cite{NNAugumentedFEM2017}, and then further explored in~\cite{Hybrid_FEM-NN_2021}.
In the experiment, both methods ran their optimisations for 2,000 iterations with identical $\kappa_{\mathcal{N}}$ structures and initialisations. 
The gaps in the state error curves in Fig.~\ref{fig:partial_observation_ul2error_feinn_adjoint_nn_cmp} and ~\ref{fig:partial_observation_uh1error_feinn_adjoint_nn_cmp} for \acp{feinn} correspond to the second model parameter initialisation step, where $u_{\mathcal{N}}$ is not trained. 
Similarly, the coefficient error curve in Fig.~\ref{fig:partial_observation_kl2error_feinn_adjoint_nn_cmp} for \acp{feinn} starts at iteration 401 as $\kappa_{\mathcal{N}}$ is not trained in the initial data fitting step. The shapes of the curves for \acp{feinn} align with the motivation behind the three-step training process, where the first and second steps aim to lead $u_{\mathcal{N}}$ and $\kappa_{\mathcal{N}}$ to a good initialisation, while the third step focuses on further improving the accuracy. The experiments were performed on a single core of an AMD Ryzen Threadripper 3960X CPU, and we also report the computational cost for both methods: for \acp{feinn}, the average training time is 0.028 seconds per iteration, whereas adjoint \ac{nn} requires 0.040 seconds per iteration.\footnote{It is important to note that the computational cost of these two methods is not easily comparable, since they have different computational requirements. The adjoint method involves (non)linear solvers per external iteration, while \acp{feinn} must compute the differentiation of the \ac{nn} with respect to parameters not only for the physical coefficients but also the state variable. Thus, the relative cost of these methods will be influenced by various factors, including the structure of \acp{nn}, the implementation of the (non)linear solver, the specific problem being addressed, etc.}
Notably, benefiting from our three-step training strategy and an additional network for state approximation, 
Fig.~\ref{fig:partial_observation_errors_feinn_adjoint_nn_cmp} reveals that \acp{feinn} have the potential to yield superior accuracy compared to adjoint \ac{nn}, as all \ac{feinn} curves remain below the error curves of adjoint \ac{nn} after approximately 800 iterations.
{Moreover, we also plot the error curves for the non-interpolated \acp{nn} in Fig.~\ref{fig:partial_observation_errors_feinn_adjoint_nn_cmp}. Similar to the findings in Sec.~\ref{sec:advection_diffusion_eq_smooth}, the smoothness of \ac{nn} contributes to improved $H^1$ accuracy of a smooth $u$.}      

\begin{figure}
    \centering
    \includegraphics[width=\textwidth]{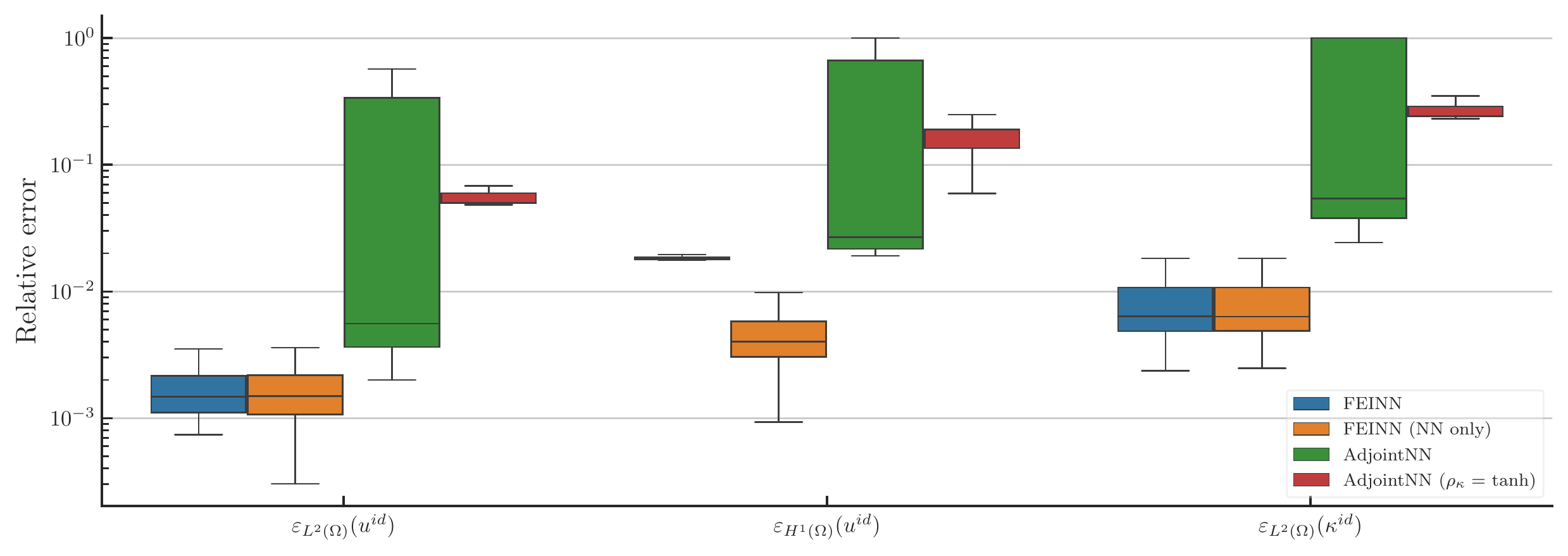}
    \caption{Comparison among \acp{feinn} and adjoint \ac{nn} in terms of relative errors (depicted using box plots) with different initialisations for the inverse Poisson problem with partial observations.}
    \label{fig:partial_observation_methods_cmp_boxplot}
\end{figure}

In the sequel, we also compare the robustness with respect to \ac{nn} initialisation of \acp{feinn} and adjoint \ac{nn}. We solve the inverse problem 100 times with different \ac{nn} initialisations. The same \ac{nn} structure and parameter initialisation of $\kappa_{\mathcal{N}}$ were used for both methods in order to have a fair comparison. In Fig.~\ref{fig:partial_observation_methods_cmp_boxplot}, we depict with box plots the relative errors for the state and the diffusion coefficient from these 100 experiments. Whiskers in the box plot represent the minimum and maximum values within 1.5 times the interquartile range. 
 
Let us first comment on the results obtained with \acp{feinn}. Most of the errors for the state $\varepsilon_{L^2(\Omega)}(u^{id})$ and $\varepsilon_{H^1(\Omega)}(u^{id})$ are very small, with the largest $\varepsilon_{L^2(\Omega)}(u^{id})$ below $0.8\%$. Besides, the majority of the relative coefficient errors $\varepsilon_{L^2(\Omega)}(\kappa^{id})$ are below $1\%$. Consequently, we conclude that \acp{feinn} are robust with respect to initialisation in solving this inverse problem with partial observations. Again, the label tag ``(NN only)'' of \acp{feinn} denotes the errors of the \acp{nn} themselves. We observe that the $L^2$ errors for $\kappa_{\mathcal{N}}$ are nearly equivalent to their interpolated counterparts. However, consistent with the findings in Fig.~\ref{fig:partial_observation_errors_feinn_adjoint_nn_cmp}, since $u$ is smooth, $u_{\mathcal{N}}$ surpasses their interpolations in $H^1$ accuracy, with potential for improved $L^2$ accuracy.

The results corresponding to adjoint \ac{nn} are presented in Fig.~\ref{fig:partial_observation_methods_cmp_boxplot} as box plots labelled ``AdjointNN''. During training, we observe that a good initialisation for $\kappa_{\mathcal{N}}$ is imperative, otherwise the optimisation quits prematurely as the gradient norm drops below $10^{-10}$. This occurrence results in considerably adverse outcomes, at times with $\varepsilon_{L^2(\Omega)}(\kappa^{id})$ exceeding $10^8$. To enhance visual clarity, when constructing the box plots, errors surpassing 1 are standardised to 1. As Fig.~\ref{fig:partial_observation_methods_cmp_boxplot} indicates, we also explored the activation function $\tanh$ as proposed in~\cite{Hybrid_FEM-NN_2021} (tagged as ``($\rho_\kappa = \tanh$)''). However, neither of these configurations produce results outperforming those achieved by \acp{feinn}. Therefore, the adjoint \ac{nn} method clearly shows less robustness than \acp{feinn} in this partial observations situation.

\subsubsection{Poisson equation with noisy observations} \label{subsubsec:noisy_observation}
In this experiment, we explore the effectiveness of \acp{feinn} in solving an inverse Poisson problem with noisy data. Following the settings in ~\cite[Sec. 3.1.2]{Hybrid_FEM-NN_2021}, we consider the true state (Fig.~\ref{fig:poisson_noisy_observation_true_state}) and diffusion coefficient (Fig.~\ref{fig:poisson_noisy_observation_true_coefficient}) as:
\begin{equation*}
    u(x, y) = \sin(\pi x) \sin(\pi y), \qquad
    \kappa(x, y) = \frac{1}{1 + x^2 + y^2 + (x-1)^2 + (y-1)^2}.   
\end{equation*}
The domain $\Omega$, its discretisation and boundary conditions remain the same as in Sec.~\ref{subsubsec:partial_observation}. The state at each \ac{dof} is known but contaminated with Gaussian noise $\epsilon \sim N(0, 0.05^2)$. The objectives of this experiment are to reconstruct the state from the noisy data and to estimate the unknown diffusion coefficient. 

The structures for $u_{\mathcal{N}}$ and $\kappa_{\mathcal{N}}$ are the same as the ones in Sec.~\ref{subsubsec:partial_observation}. We again apply $r(x) = |x|+0.01$ to the output layer of $\kappa_{\mathcal{N}}$ to ensure a positive diffusion coefficient. The training iterations are $[300, 100, 2\times300]$, with a total of 1,000, matching the setup in~\cite[Sec. 3.1.2]{Hybrid_FEM-NN_2021}. The penalty coefficients are $\alpha = [1.0, 3.0]$.

\begin{figure}
    \centering
    \begin{subfigure}[t]{0.32\textwidth}
        \includegraphics[width=\textwidth]{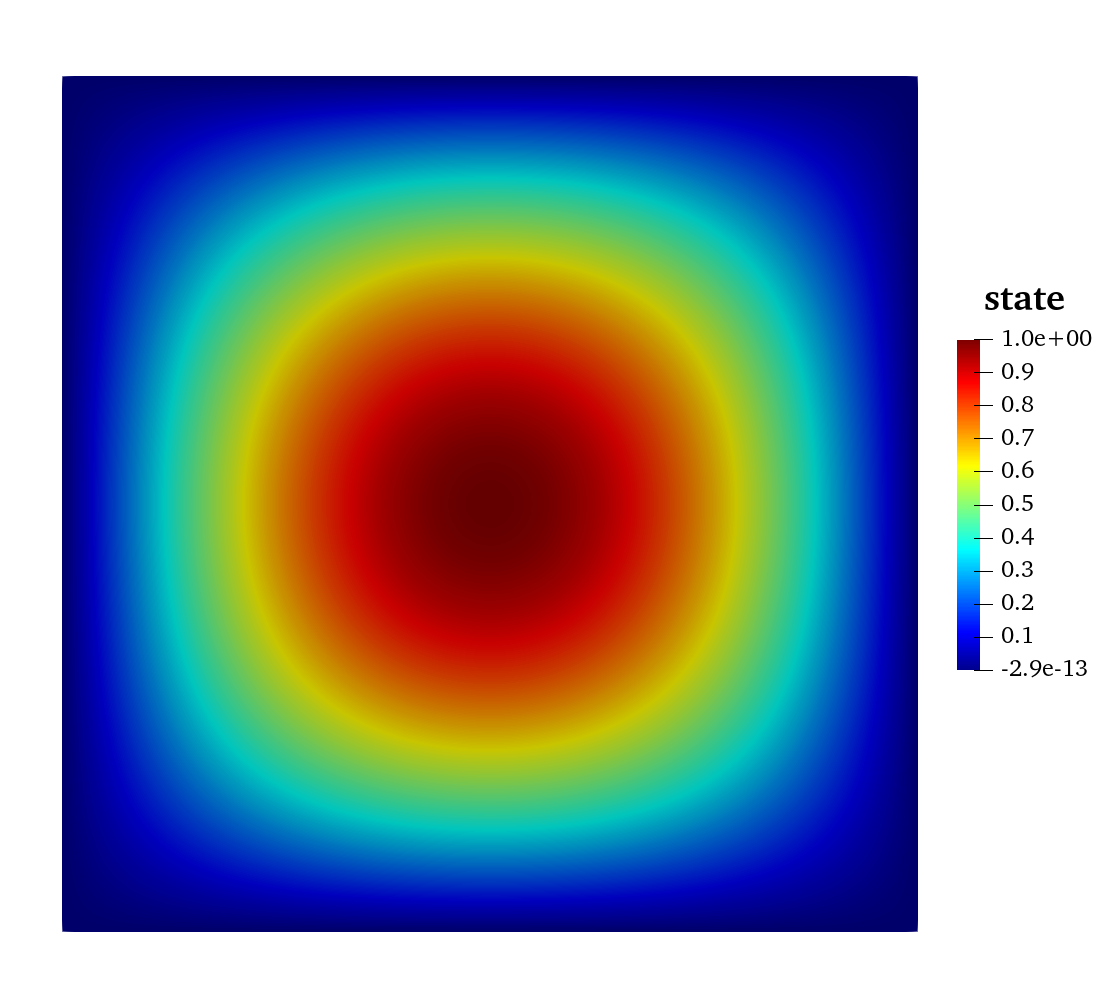}
        \caption{$u$}
        \label{fig:poisson_noisy_observation_true_state}
    \end{subfigure}
    \begin{subfigure}[t]{0.32\textwidth}
        \includegraphics[width=\textwidth]{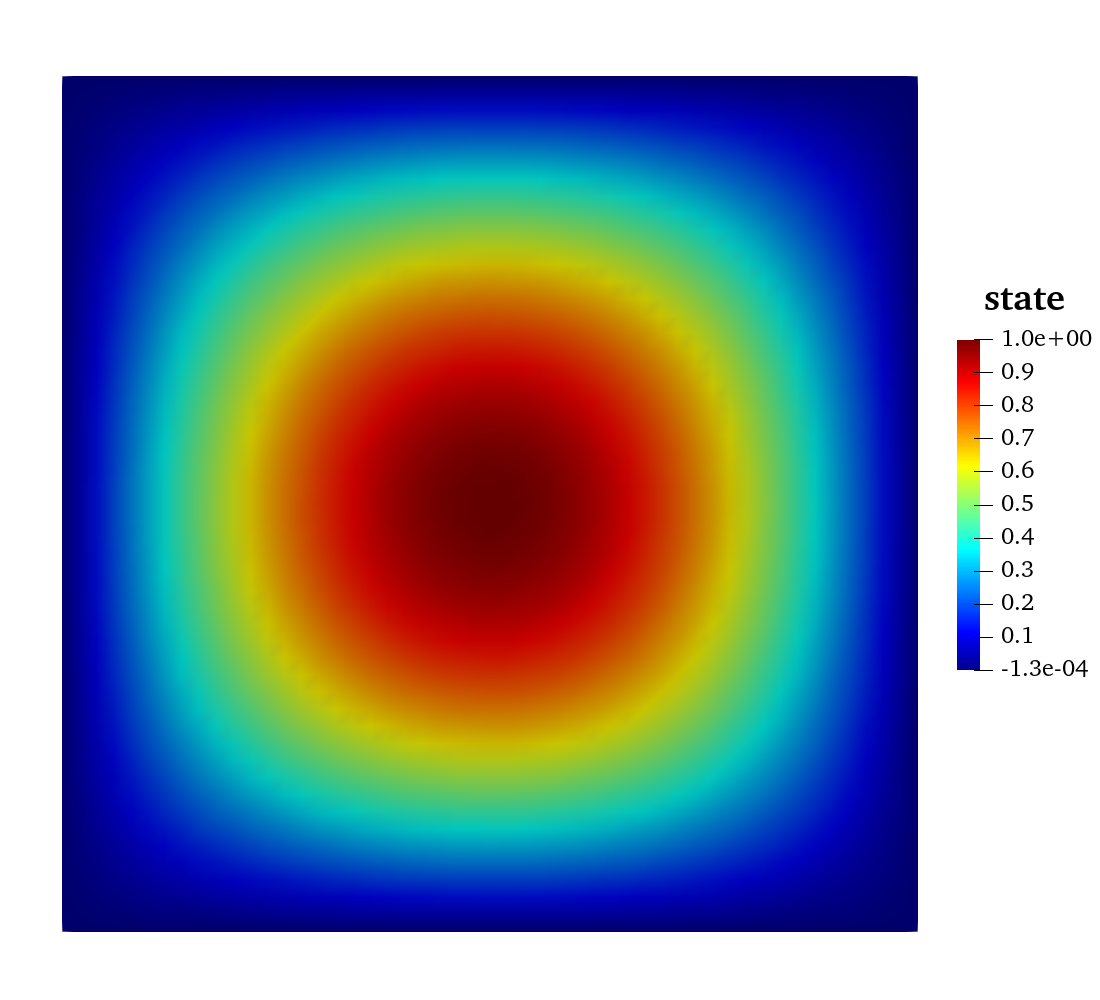}
        \caption{$u^{id}$}
        \label{fig:poisson_noisy_observation_feinn_state}
    \end{subfigure}
    \begin{subfigure}[t]{0.32\textwidth}
        \includegraphics[width=\textwidth]{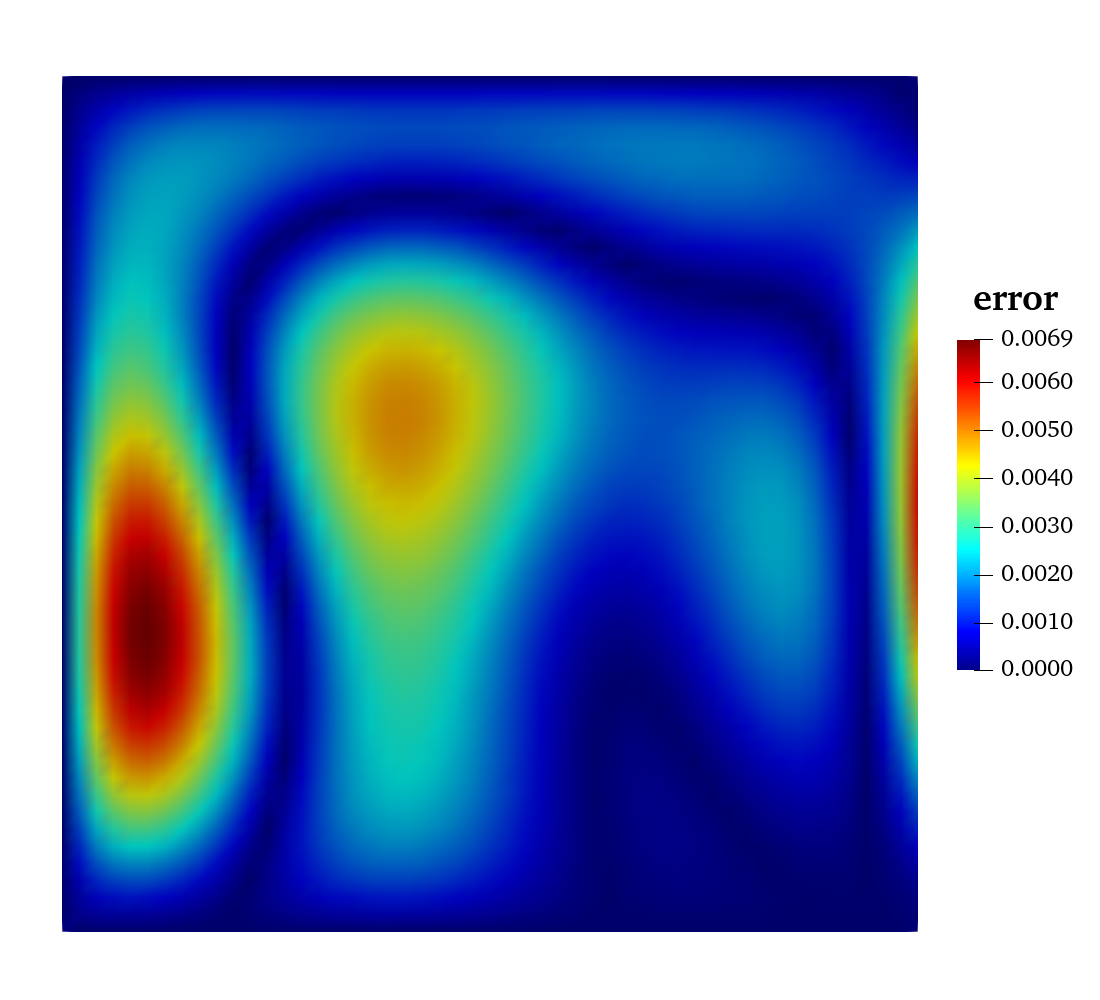}
        \caption{$|u - u^{id}|$}
        \label{fig:poisson_noisy_observation_state_error}
    \end{subfigure}

    \centering
    \begin{subfigure}[t]{0.32\textwidth}
        \includegraphics[width=\textwidth]{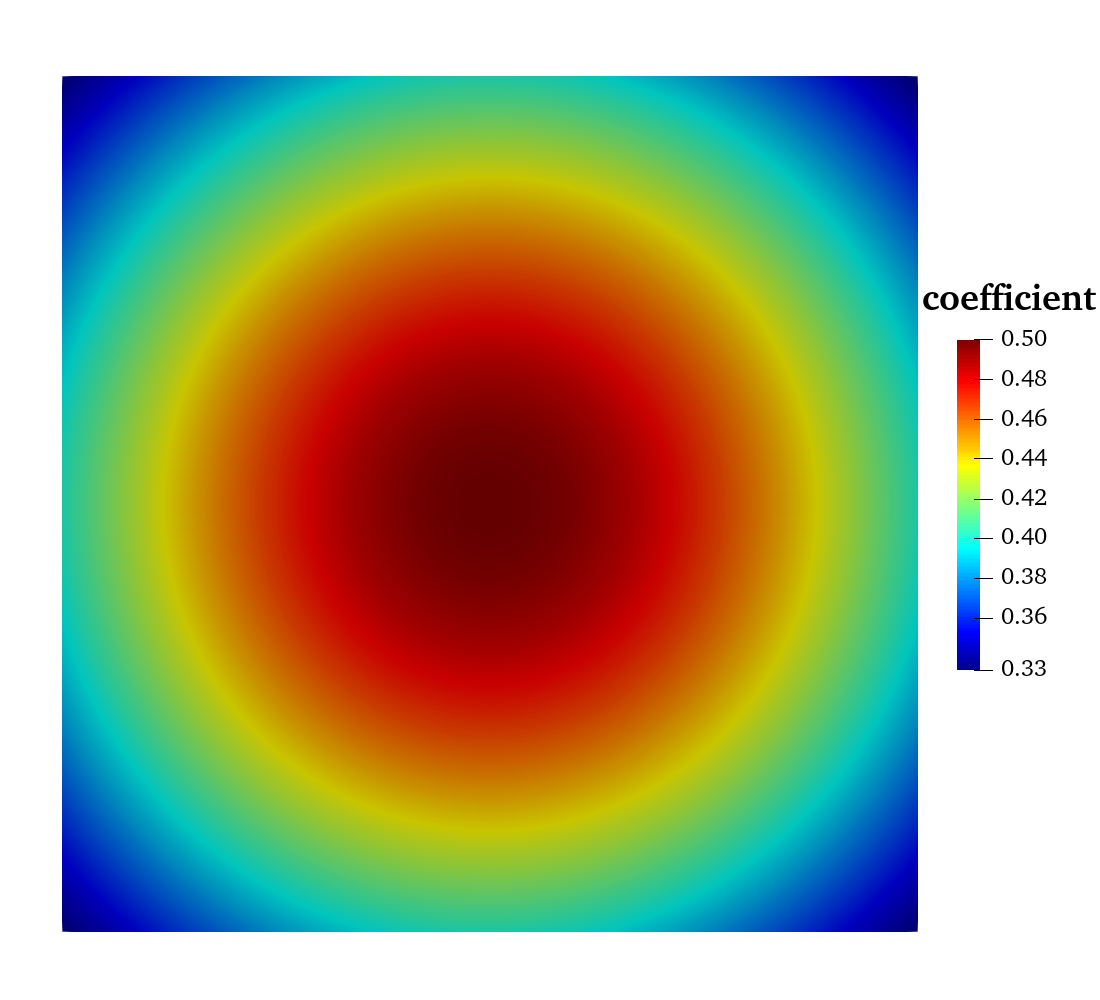}
        \caption{$\kappa$}
        \label{fig:poisson_noisy_observation_true_coefficient}
    \end{subfigure}
    \begin{subfigure}[t]{0.32\textwidth}
        \includegraphics[width=\textwidth]{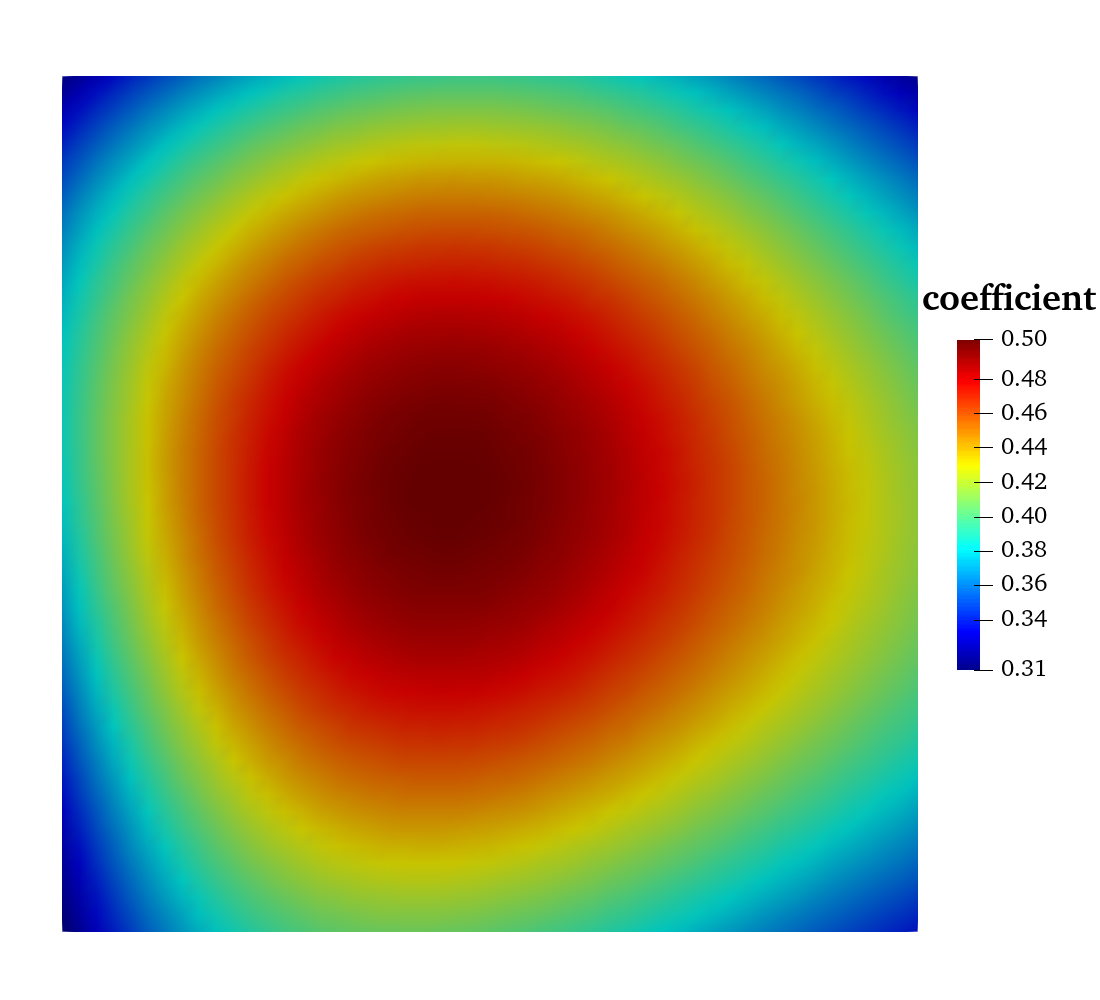}
        \caption{$\kappa^{id}$}
        \label{fig:poisson_noisy_observation_feinn_coefficient}
    \end{subfigure}
    \begin{subfigure}[t]{0.32\textwidth}
        \includegraphics[width=\textwidth]{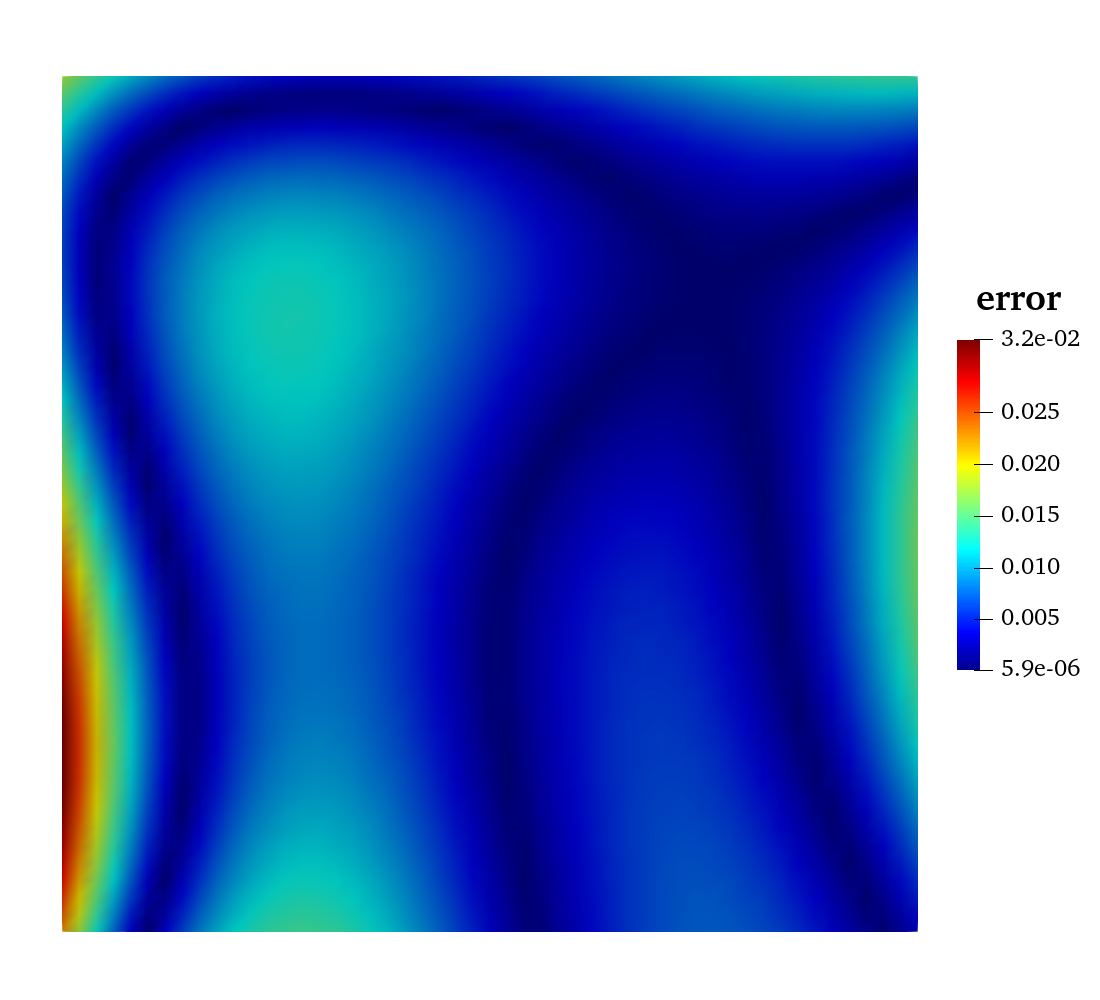}
        \caption{$|\kappa - \kappa^{id}|$}
        \label{fig:poisson_noisy_observation_coefficient_error}
    \end{subfigure}
     
    \caption{Comparison of the true solutions (first column), the \ac{feinn} solutions (second column), and corresponding point-wise errors (third column) for the inverse Poisson problem with noisy observations. The presented results are from a specific experiment. The first row depicts the state, while the second row represents the coefficient.}
    \label{fig:poisson_noisy_observation_results}
\end{figure}

In Fig.~\ref{fig:poisson_noisy_observation_results}, the last two columns display the outcomes from one of our experiments. The identified state $u^{id}$ in Fig.~\ref{fig:poisson_noisy_observation_feinn_state} and its low point-wise error in Fig.~\ref{fig:poisson_noisy_observation_state_error} validate the \acp{feinn} capability of recovering the state despite the presence of noise in the data. Fig.~\ref{fig:poisson_noisy_observation_feinn_coefficient} shows the identified diffusion coefficient $\kappa^{id}$, which, although visually slightly different from $\kappa$ in Fig.~\ref{fig:poisson_noisy_observation_true_coefficient}, still captures its pattern very well. The point-wise error in Fig.~\ref{fig:poisson_noisy_observation_coefficient_error} further confirms that \acp{feinn} effectively predict the values of the diffusion coefficient.

\begin{figure}
    \centering
    \begin{subfigure}[t]{0.32\textwidth}
        \includegraphics[width=\textwidth]{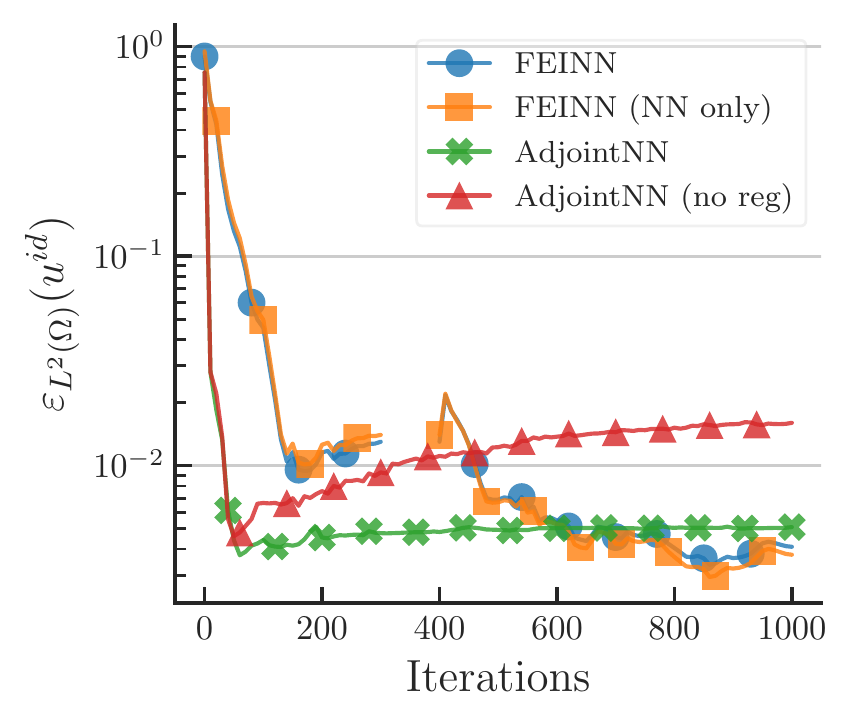}
        \caption{}
        \label{fig:noisy_observation_ul2error_feinn_adjoint_nn_cmp}
    \end{subfigure}
    \begin{subfigure}[t]{0.32\textwidth}
        \includegraphics[width=\textwidth]{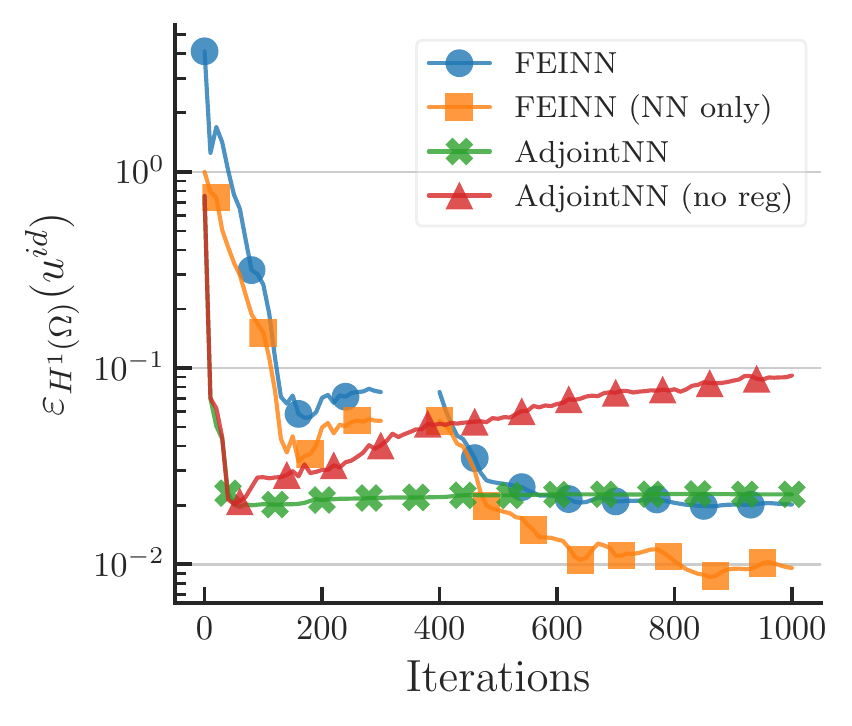}
        \caption{}
        \label{fig:noisy_observation_uh1error_feinn_adjoint_nn_cmp}
    \end{subfigure}
    \begin{subfigure}[t]{0.32\textwidth}
        \includegraphics[width=\textwidth]{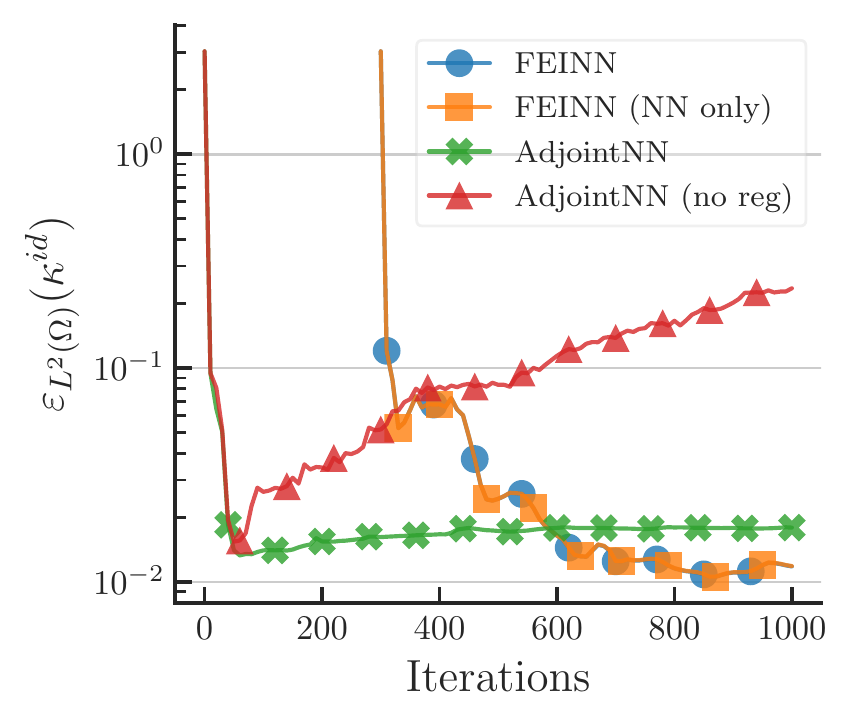}
        \caption{}
        \label{fig:noisy_observation_kl2error_feinn_adjoint_nn_cmp}
    \end{subfigure}
    \caption{Comparison among \acp{feinn} and adjoint \ac{nn} in terms of relative errors during training for the inverse Poisson problem with noisy observations. The optimisation loop was run for 1,000 iterations in both cases.}
    \label{fig:noisy_observation_errors_feinn_adjoint_nn_cmp}
\end{figure}

The error history plots for \acp{feinn} and adjoint \ac{nn} when applied to the inverse Poisson problem with noisy observations are shown in Fig.~\ref{fig:noisy_observation_errors_feinn_adjoint_nn_cmp}.
The optimisation loop was run in both cases up to 1,000 iterations, and we used the same $\kappa_{\mathcal{N}}$ architecture and parameter initialisation.
Consistent with the findings in ~\cite{Hybrid_FEM-NN_2021}, the loss function in the adjoint method requires explicit regularisation. This is evident as the error curves corresponding to adjoint \ac{nn} with no regularisation (label tag ``(no reg)'') start increasing very shortly after the optimisation begins, while the results are much improved by using the regularisation proposed in \cite{Hybrid_FEM-NN_2021}.\footnote{We use $\ell^1$ regularisation on $\kappa_{\mathcal{N}}$. After testing various regularisation coefficients, we have concluded that the best results are obtained for $10^{-3}$.}
Notably, even \emph{without any regularisation}, the \ac{feinn} errors are very stably decreasing. \acp{feinn} could possibly benefit from effective regularisation, but we have not explored this option to keep the method simple and less tuning-dependent. In terms of computational cost, \acp{feinn} demand 0.025 seconds per iteration, while adjoint \ac{nn} takes 0.043 seconds per iteration. {Additionally, in terms of $u$, the errors of the (non-interpolated) \acp{nn} are frequently below their interpolation counterparts during training. This indicates that $u_{\mathcal{N}}$ possesses the capacity to improve $u$ accuracy despite the noisy observations.}

\begin{figure}
    \centering
    \begin{subfigure}[t]{0.68\textwidth}
        \includegraphics[width=\textwidth]{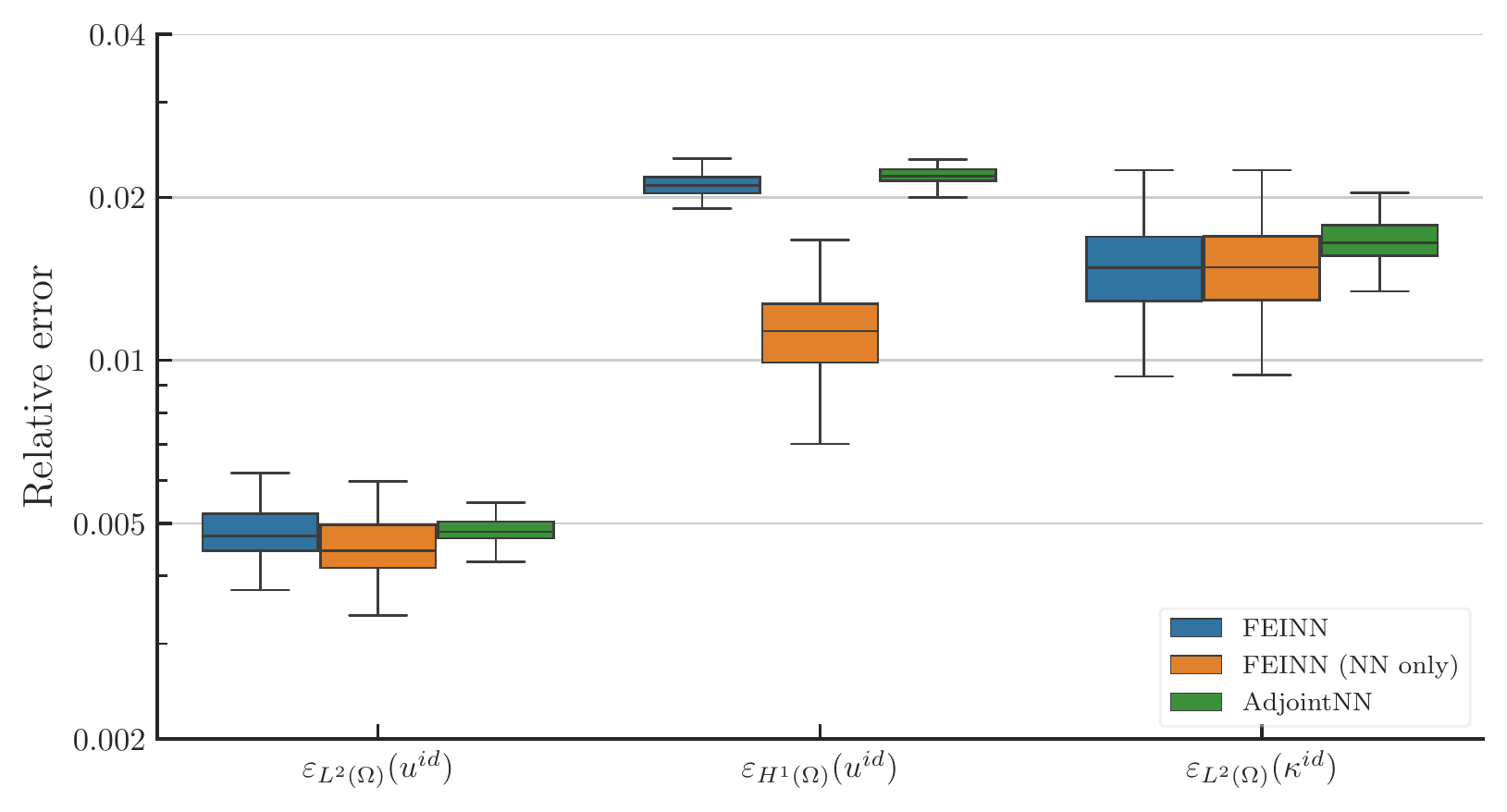}
        \caption{}
        \label{fig:noisy_observation_methods_cmp_boxplot}
    \end{subfigure}
    \begin{subfigure}[t]{0.28\textwidth}
        \includegraphics[width=\textwidth]{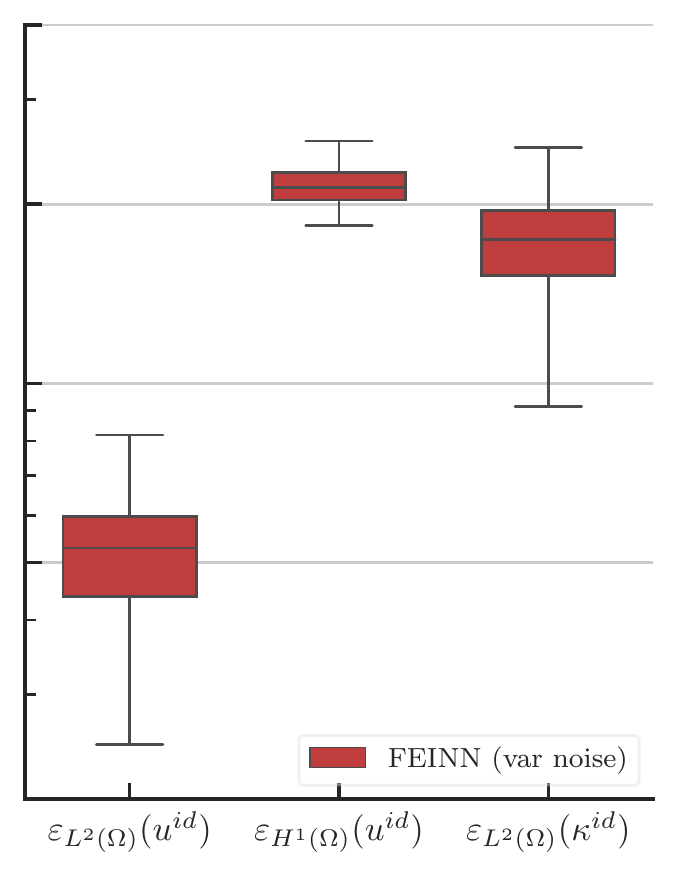}
        \caption{}
        \label{fig:noisy_observation_random_noisy_study}
    \end{subfigure}

    \caption{Left: comparison among \acp{feinn} and adjoint \ac{nn} in terms of relative errors (depicted using box plots) with different \ac{nn} initialisations for the inverse Poisson problem with noisy observations. Right: box plots depicting the relative errors of \acp{feinn} trained with Gaussian noise generated by different random seeds.}
\end{figure}

Let us assess the robustness of \acp{feinn} with respect to \ac{nn} initialisation. We generate the Gaussian noise with the same random seed and repeat the experiment 100 times with differently initialised \acp{nn}. The resulting box plots are shown in Fig.~\ref{fig:noisy_observation_methods_cmp_boxplot}, where label ``\ac{feinn}'' is for the interpolated \acp{nn} and ``\ac{feinn} (\ac{nn} only)'' is for the non-interpolated ones.
{We observe that the smoothness of $u_{\mathcal{N}}$ contributes to enhanced accuracy, as both boxes of $\varepsilon_{L^2(\Omega)}(u^{id})$ and $\varepsilon_{H^1(\Omega)}(u^{id})$ for $u_{\mathcal{N}}$ are positioned lower than their interpolation counterparts.} Besides, \acp{feinn} generally produce very good results, with $\varepsilon_{L^2(\Omega)}(u^{id})$ mostly below $0.6\%$, and $\varepsilon_{L^2(\Omega)}(\kappa^{id})$ mostly under $2\%$. 

In this experiment, To compare the performance of \acp{feinn} against adjoint \ac{nn} with regularisation (as described above), we provide the results for the latter method in the same figure (labelled as ``AdjointNN''). 
We observe that the boxes of $\varepsilon_{H^1(\Omega)}(u^{id})$ and $\varepsilon_{L^2(\Omega)}(\kappa^{id})$ of adjoint \ac{nn} are positioned higher than that of \acp{feinn}, suggesting that \acp{feinn} generally achieve better accuracy in terms of these two relative errors. Furthermore, the state \ac{nn} in \acp{feinn} generalises well and is far more accurate than the the \ac{fe} interpolation. In contrast, adjoint \ac{nn} relies on a \ac{fe} function for state approximation, lacking such capability of \acp{feinn}. 

In this example, we are also interested in exploring how the variability of noise affects \acp{feinn} accuracy. We fix the \ac{nn} initialisation and the distribution of the Gaussian noise ($\epsilon \sim N(0, 0.05^2)$), and repeat the experiment 100 times with different random noise seeds. The resulting box plots are shown in Fig.~\ref{fig:noisy_observation_random_noisy_study} with label ``\ac{feinn} (var noise)''. We observe that the noise randomness impacts the accuracy of \acp{feinn} more than \ac{nn} initialisation randomness, with broader error boxes. Nonetheless, \acp{feinn} are still robust in this scenario, since most $\varepsilon_{L^2(\Omega)}(u^{id})$ are below $0.8\%$ and most $\varepsilon_{L^2(\Omega)}(\kappa^{id})$ are less than $3\%$.

\subsubsection{Inverse heat conduction problem} \label{subsubsec:ihcp}
In our final experiment for this paper, we attack an \ac{ihcp}. In many heat transfer applications, the boundary values are either unavailable or difficult to measure over the entire surface. The goal of \acp{ihcp} is to estimate the surface temperature (Dirichlet boundary value), and/or heat flux (Neumann boundary value), based on temperature data measured at certain points within the domain ~\cite{FEM_IHCP1997}. Our example combines the challenges in ~\cite{FEM_IHCP1997} and ~\cite{PINNsIHCP2022}, where we consider a two-layered half-tube cross-section as the computational domain $\Omega$, as shown in Fig.~\ref{fig:inverse_heat_conduction_true_state}. The domain $\Omega$ can be described in polar coordinates as $\theta \in [0,\pi]$ and $r \in [0.05, 0.11]$. The tube is composed of two layers of media, with a diffusion coefficient of $\kappa_1 = 1$ for $r \in [0.05, 0.08]$, and $\kappa_2 = 100$ for $r \in [0.08, 0.11]$. The unknown boundary values are, in polar coordinates,
\begin{equation*}
    g(\theta, r) = 200, \ r = 0.05, \qquad
    \eta(\theta, r) = -100 -50 \sin(\theta), \ r = 0.11.
\end{equation*}
The horizontal section of the tube is also a Neumann boundary, with known $\eta = 0$.

\begin{figure}
    \centering

    \begin{subfigure}[t]{0.48\textwidth}
        \includegraphics[width=\textwidth]{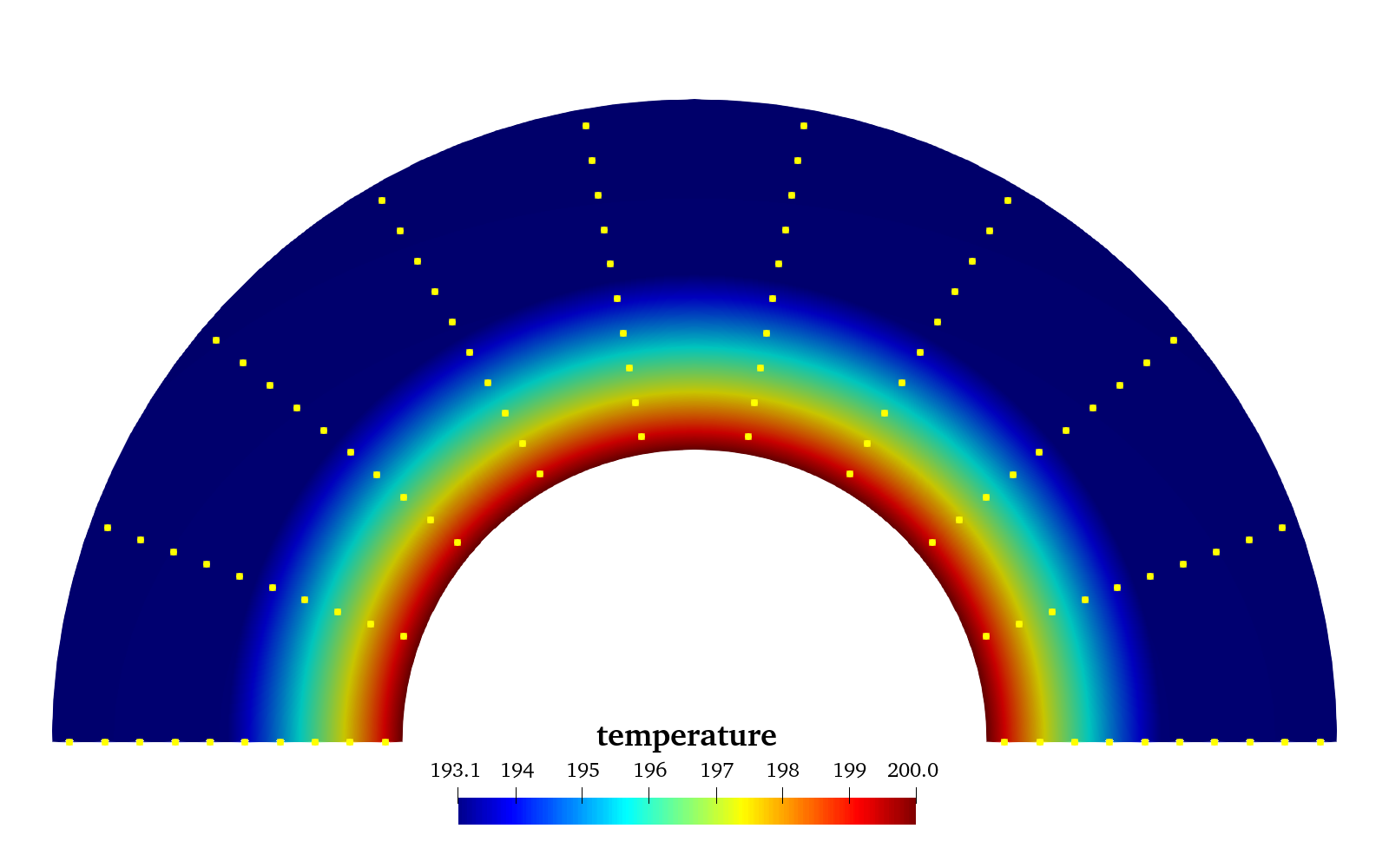}
        \caption{}
        \label{fig:inverse_heat_conduction_true_state}
    \end{subfigure}
    \begin{subfigure}[t]{0.48\textwidth}
        \includegraphics[width=\textwidth]{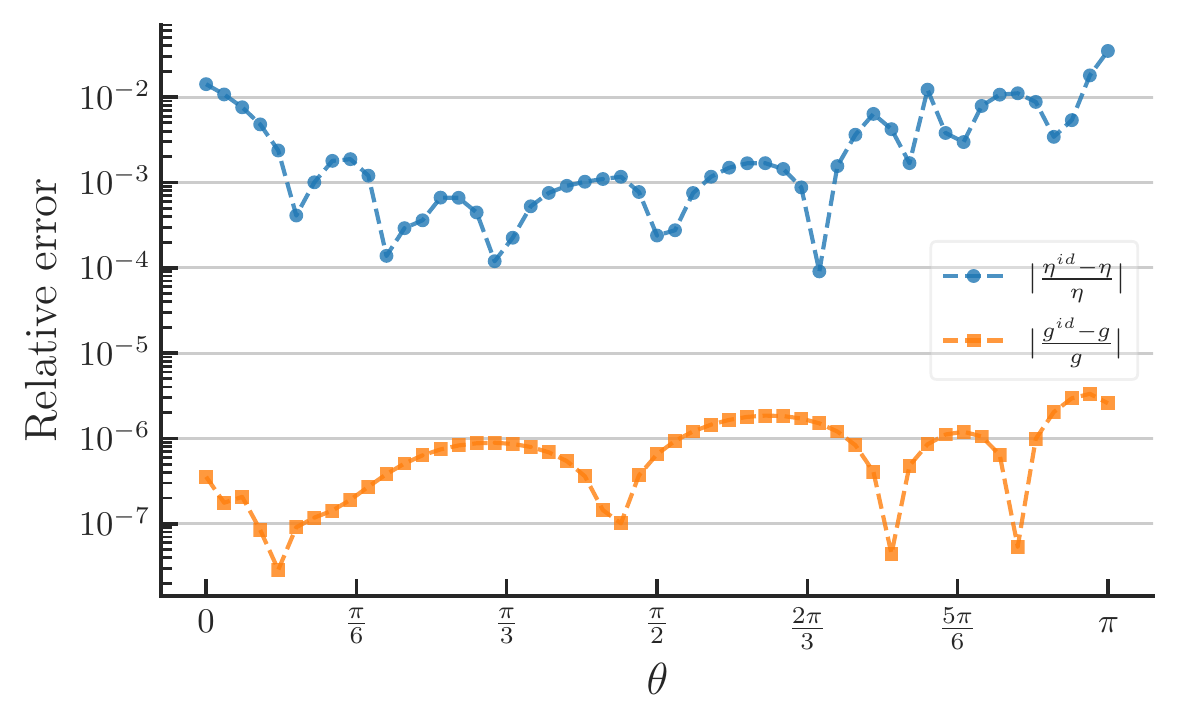}
        \caption{}
        \label{fig:inverse_heat_conduction_pointwise_error}
    \end{subfigure}
    \caption{True temperature distribution and the relative point-wise errors of \acp{feinn} on the Dirichlet and Neumann boundaries for the \ac{ihcp}. Yellow dots on the temperature figure indicate observation locations.}
\end{figure}

We discretise the domain with $2\times 50 \times 50$ triangles and solve the forward problem using \ac{fem} with the aforementioned boundary conditions. Fig.~\ref{fig:inverse_heat_conduction_true_state} shows the \ac{fem} solution of the temperature, and we use the temperature at the yellow dots as our observations. Since the temperature has different patterns in the two layers due to the discontinuity in the diffusion coefficient, we use a deeper \ac{nn} with 6 layers and 20 neurons for each hidden layer ($L_u = 6$, $n_u = 20$) as $u_{\mathcal{N}}$. The Dirichlet boundary value $g$ is just a part of $u$, so in this problem, $u_{\mathcal{N}}$ is defined over the whole domain, including the Dirichlet boundary. We evaluate $u^{id}$ on the Dirichlet boundary to obtain $g^{id}$. We train another \ac{nn} $\eta_{\mathcal{N}}$ with $L_\eta = 3$ and $n_\eta = 20$ on the Neumann boundary to predict the Neumann boundary value. We set the number of training iterations to $[700, 200, 3\times700]$, and use the penalty coefficients $\alpha = [0.001, 0.003, 0.009]$.

\begin{figure}
    \centering
    \begin{subfigure}[t]{0.48\textwidth}
        \includegraphics[width=\textwidth]{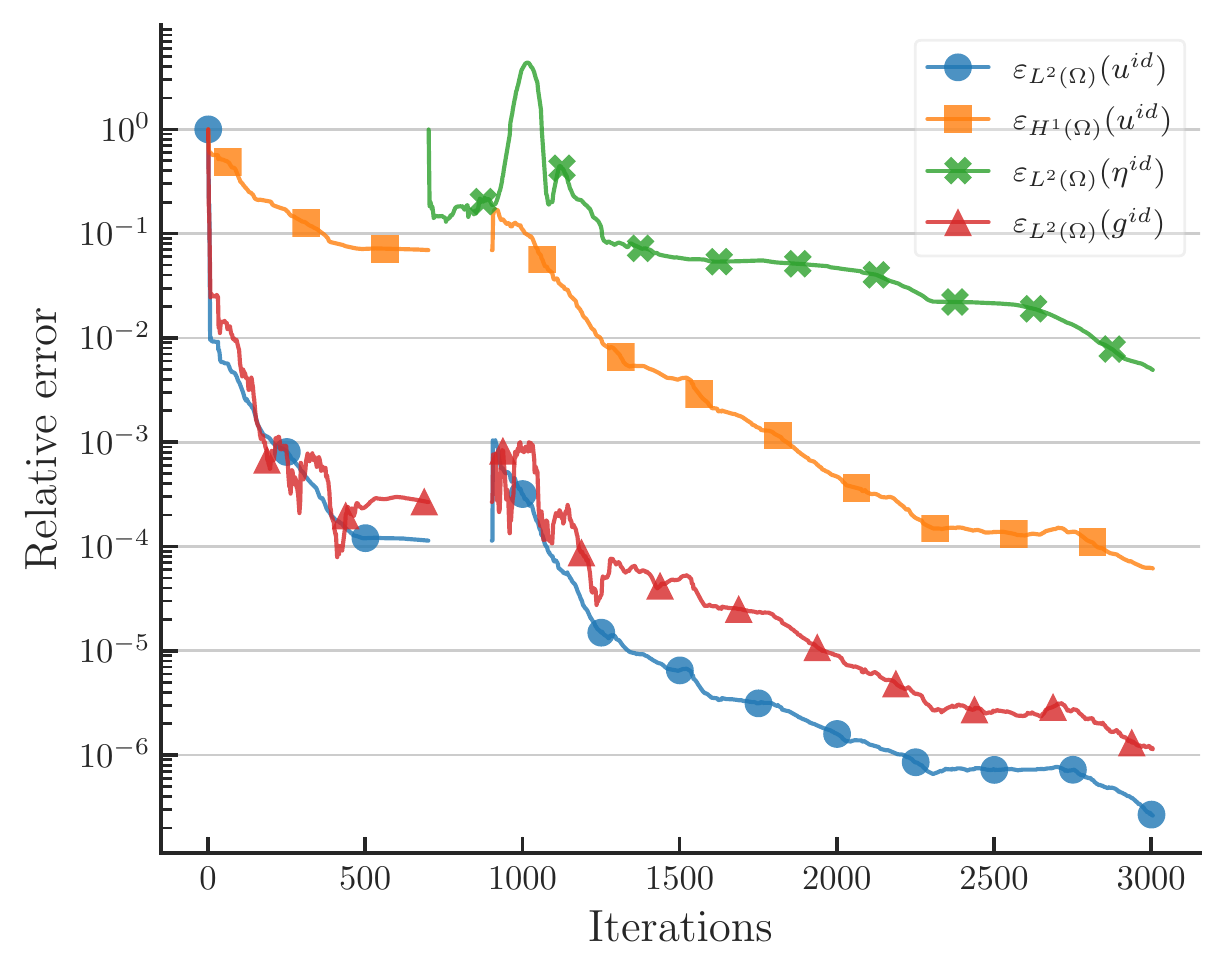}
        \caption{}
        \label{fig:inverse_heat_conduction_error_history}
    \end{subfigure}
    \begin{subfigure}[t]{0.48\textwidth}
        \includegraphics[width=\textwidth, trim={0 -0.42cm 0 0.42cm}, clip]{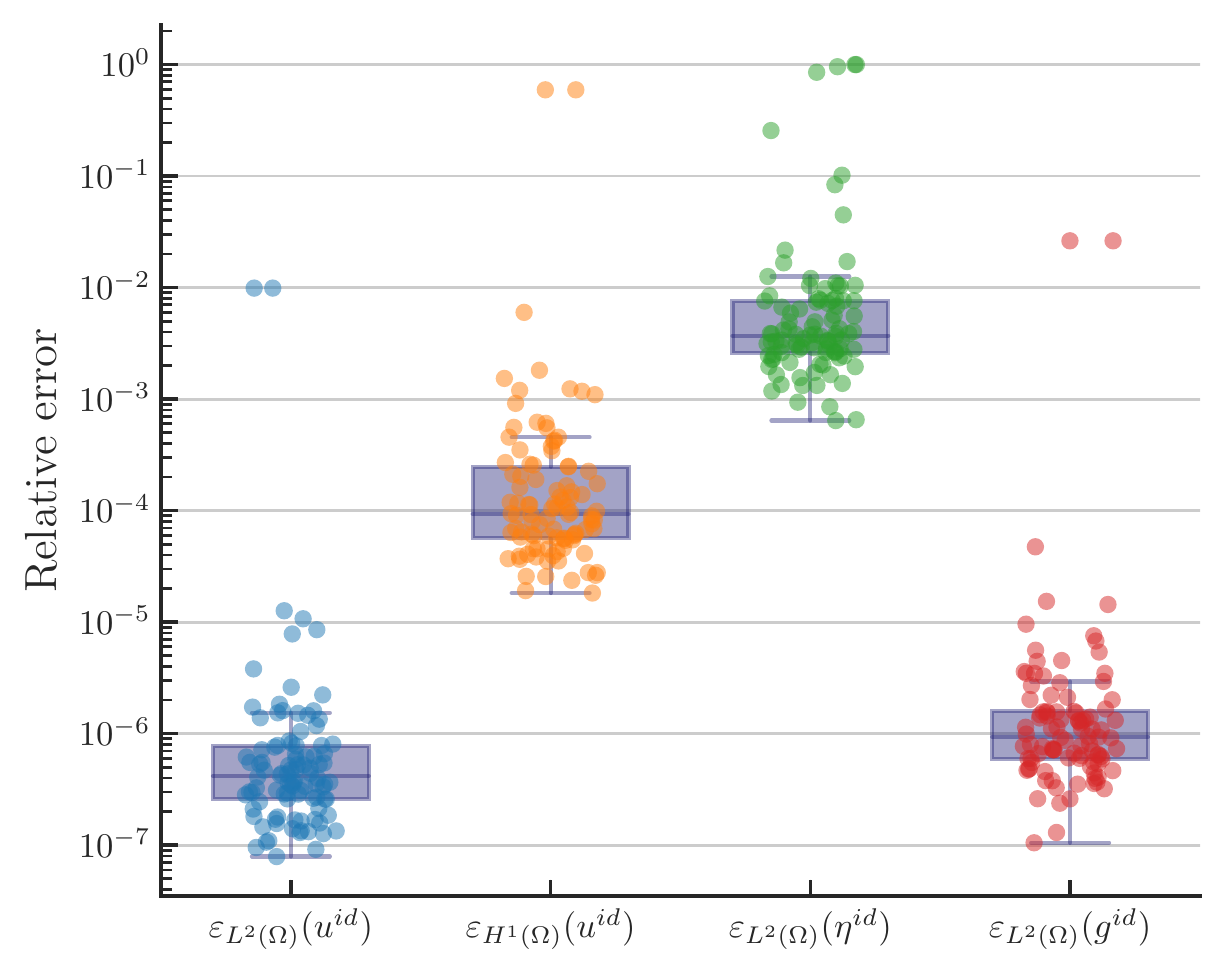}
        \caption{}
        \label{fig:inverse_heat_conduction_network_init_study_boxplot}
    \end{subfigure}
     
    \caption{The error history during training from a specific experiment and the box plots of the relative errors of \acp{feinn} with different initialisations for the \ac{ihcp}.}
\end{figure}

Fig.~\ref{fig:inverse_heat_conduction_pointwise_error} shows the relative point-wise errors of \acp{feinn} solutions for the boundary values $\eta$ and $g$, obtained from one of our experiments. The errors at most of the Neumann boundary points are below $1\%$, indicating accurate recovery. Besides, the identified Dirichlet value $g^{id}$ is even more accurate, with a maximum error of approximately $3 \times 10^{-6}$. Overall, \acp{feinn} excel at accurately reconstructing the boundary values. Fig.~\ref{fig:inverse_heat_conduction_error_history} depicts the history of relative errors during training from the same experiment. We observe that the data step and model parameter initialisation step reduce corresponding errors as expected, and the errors steadily decrease after a few hundred iterations of adjustment in the coupled step. 

To study \acp{feinn}' reliability in solving \acp{ihcp}, we repeat the experiment 100 times with different \ac{nn} initialisations. The resulting errors are presented in Fig.~\ref{fig:inverse_heat_conduction_network_init_study_boxplot} as box plots along with the original data points. Even though the number of observations (100) is much smaller that the \acp{dof} (2,601) of the trial space, \acp{feinn} recover the temperature distribution accurately, with most $\varepsilon_{L^2(\Omega)}(u^{id})$ errors below $10^{-5}$, and only a few outliers with higher errors. The proposed formulation demonstrates robustness despite a significant discontinuity in the diffusion coefficient, a limited number of observations, and no regularisation. Furthermore, the majority of experiments (at least $90\%$) yield remarkably low errors.

\section{Conclusions} \label{sec:conclusions}

In this paper, we propose a general framework, called \acp{feinn}, to approximate forward and inverse problems governed by \emph{low-dimensional} \acp{pde}, by combining \acp{nn} and \acp{fe} to overcome some of the limitations (numerical integration error, treatment of Dirichlet boundary conditions, lack of solid mathematical foundations) of existing approaches proposed in the literature to approximate \acp{pde} with \acp{nn}, such as, e.g., \acp{pinn}. For forward problems, we interpolate the \ac{nn} onto the \ac{fe} space with zero traces (non-homogeneous Dirichlet boundary conditions are enforced via a standard offset \ac{fe} function), and evaluate the \ac{fe} residual for the resulting \ac{fe} function. The loss function is the norm of the \ac{fe} residual. We propose different norms, and suggest the use of standard \ac{fe} preconditioners (e.g., a fixed number of \ac{gmg} cycles) to end up with a well-posed loss function in the limit $h \downarrow 0$. For inverse problems, the unknown model parameters are parametrised via \acp{nn}, which can also be interpolated onto \ac{fe} spaces. The loss function in this case combines the data misfit term with a penalty term for the PDE residual. We propose a three step algorithm to speed up the training process of the resulting formulation, where we perform two cheap data fitting steps (no differential operators involved) to provide a good initialisation for a fully coupled minimisation step.

We have conducted numerous numerical experiments to assess the computational performance and accuracy of \acp{feinn}.  We use forward convection-diffusion-reaction problems to compare \acp{feinn} against \acp{ivpinn}, a recently proposed related method which mainly differs in the treatment of Dirichlet boundary conditions and has been proven to be superior to other \ac{pinn} formulations in certain situations~\cite{BerroneIVPINN2022}. The computational cost per iteration of \acp{ivpinn} and \acp{feinn} is virtually the same. However, \acp{ivpinn} struggle to keep the convergence of \acp{feinn} (and reach the \ac{fem} error) as we increase mesh resolution or polynomial order.
Additionally, the (non-interpolated) \acp{nn} trained with \acp{feinn} exhibits excellent generalisation, with superior performance compared to the \ac{fe} solution and the non-interpolated \ac{nn} composition of \acp{ivpinn}. For singular solutions, both \acp{feinn} and \acp{ivpinn} have comparable performance to \ac{fem}. We evaluate the effect of the residual norm and show how preconditioned norms accelerate the training. Moreover, experiments performed on a non-trivial geometry highlights the capability \acp{feinn} handling complex geometries and Dirichlet boundary condition effortlessly, which is not the case of \acp{ivpinn} or standard \acp{pinn}. 

In the experiments for the inverse problems, we show that \acp{feinn} are capable of estimating unknown diffusion coefficient from partial or noisy observations of the state and recovering the unknown boundary values from discrete observations. We additionally compare the performance of \acp{feinn} against the adjoint-based \ac{nn} method~\cite{NNAugumentedFEM2017,Hybrid_FEM-NN_2021}. The numerical results demonstrate that \acp{feinn} exhibit greater robustness for partial observations and are comparable for noisy observations. However, adjoint methods require the tunning of the regularisation term to be effective, while \acp{feinn} are robust without regularisation. The conducted experiments also prove that the three-step training process employed by \acp{feinn} is a sound strategy.

This work can be extended in many directions. First, one could consider transient and/or nonlinear \acp{pde}, in which \acp{nn} and non-convex optimisation have additional benefits compared to standard linearisation and iterative linear solvers in \ac{fem}. Besides, while this work concentrates on problems in $H^1$, the framework can be extended to problems in $H(\mathrm{ curl})$ and $H(\mathrm{div})$ spaces, combined with compatible \ac{fem}~\cite{Arnold2006,FEMPAR_Hcurl2019}. To target large scale problems, one could design domain decomposition~\cite{D3M2020,XPINNs2022} and partition of unity methods~\cite{FBPINNs2021} to end up with suitable algorithms for massively parallel distributed-memory platforms and exploit existing parallel \ac{fe} frameworks \texttt{GridapDistributed.jl}~\cite{GridapDistributed2022}. Lastly, we want to explore in the future the usage of adaptive meshes~\cite{radaptiveDL2022} to exploit the nonlinear approximability of \acp{nn} within the same training loop.

\section{Acknowledgments}

This research was partially funded by the Australian Government through the Australian Research Council
(project numbers DP210103092 and DP220103160). This work was also supported by computational resources provided by the Australian Government through NCI under the NCMAS and ANU Merit Allocation Schemes. W. Li acknowledges the support from the Laboratory for Turbulence Research in Aerospace and Combustion (LTRAC) at Monash University through the use of their HPC Clusters.

\printbibliography

\end{document}